\documentclass[11pt, reqno]{amsart}

\usepackage{amsthm,amssymb,amstext,amscd,amsfonts,amsbsy,amsxtra,latexsym,amsmath,xcolor,mathrsfs,fancybox,upgreek, soul}
\usepackage[english]{babel}
\usepackage[all,cmtip]{xy}
\usepackage[latin1]{inputenc}
\usepackage{cancel}
\usepackage[draft]{hyperref}
\usepackage{comment,enumitem}
\usepackage{mdframed}
\allowdisplaybreaks

\newtheorem{theorem}{Theorem}
\newtheorem{lemma}[theorem]{Lemma}

\newtheorem{proposition}[theorem]{Proposition}

\theoremstyle{definition}
\newtheorem{definition}{Definition}

\theoremstyle{remark}

\newtheorem*{remark}{Remark}
\newtheorem*{example}{Example}
\numberwithin{theorem}{section}
\numberwithin{equation}{section}

\newcommand{\N}{\mathbb{N}}
\newcommand{\Z}{\mathbb{Z}}
\newcommand{\R}{\mathbb{R}}
\newcommand{\C}{\mathbb{C}}

\newcommand{\Q}{\mathbb{Q}}

\newcommand{\SL}{{\text {\rm SL}}}

\newcommand{\sgn}{\operatorname{sgn}}

\newcommand{\im}{\textnormal{Im}}

\def\H{\mathbb{H}}
\renewcommand{\pmod}[1]{\  \,  \left( \mathrm{mod} \,  #1 \right)}

\begin{document}

\title[Higher depth quantum modular forms]{Higher depth quantum modular forms, multiple Eichler integrals, and $\frak{sl}_3$ false theta functions}

\author{Kathrin Bringmann, Jonas Kaszian, Antun Milas}

\address{Mathematical Institute, University of Cologne, Weyertal 86-90, 50931 Cologne, Germany}
\email{kbringma@math.uni-koeln.de}

\address{Mathematical Institute, University of Cologne, Weyertal 86-90, 50931 Cologne, Germany}
\email{jkaszian@math.uni-koeln.de}

\address{Department of Mathematics and Statistics, SUNY-Albany, Albany, NY 12222, U.S.A.}
\email{amilas@albany.edu}

\begin{abstract} We introduce and study higher depth quantum modular forms. We construct two families of examples coming from rank two false theta functions, whose ``companions'' in the lower half-plane can be also realized both as double Eichler integrals and as non-holomorphic theta series having values of  ``double error'' functions as coefficients. In particular, we prove that the false theta functions of $\frak{sl}_3$, appearing in the character of the vertex algebra $W^0(p)_{A_2}$, can be written as the sum of two depth two quantum modular forms of positive integral weight.

\end{abstract}

\maketitle

\section{Introduction and statement of results}

In this paper, we study higher depth quantum modular forms which occur as rank two false theta functions coming from characters of the vertex algebra $W^0(p)_{A_2}$ for $p \geq 2$. Via asymptotic expansions we relate these to double Eichler integrals which may be viewed as purely non-holomorphic parts of indefinite theta functions.\\
\indent Let us first recall the classical rank one case. Note that the derivative of a modular form is typically not a modular form (only a so-called quasi-modular form). However, thanks to Bol's identity, differentiating a weight $2-k \in -\N$ modular form $k-1$ times returns a modular form of weight $k$. Thus it is natural to consider holomorphic Eichler integrals. That is, if $f(\tau)=\sum_{m\geq 1} c_f(m)q^m$ ($q:=e^{2\pi i\tau}$ with $\tau \in \mathbb{H}$ throughout) is a modular form of weight $k$, then set
\begin{equation}\label{Eichler1}
\widetilde f(\tau):= \sum_{m\geq 1} \frac{c_f(m)}{m^{k-1}}q^m.
\end{equation}
It easily follows, by Bol's identity and the modularity of $f$, that the following function is annihilated by differentiating $k-1$ times
\begin{equation}\label{PP}
R_f(\tau):=\widetilde f(\tau)-\tau^{k-2} \widetilde f\left(-\frac1{\tau}\right).
\end{equation}
This yields that $R_f$ is a polynomial of degree $k-2$ ($R_f$ is the so called \textit{period polynomial} of $f$). So in particular $R_f$ is much simpler than the starting function $\widetilde f$. Note that $\widetilde f$ may also be written as an integral, namely, up to constants it equals
\begin{equation}\label{holE}
\int_\tau^{i\infty}f(w) (w-\tau)^{k-2}dw.
\end{equation}
Similarly $R_f$ has an integral representation, namely up to constants it equals
$$
\int_0^{i\infty} f(w)(w-\tau)^{k-2}dw.
$$
A similar construction works for \textit{weakly holomorphic modular forms}, i.e., those meromorphic modular forms which may only grow as $v:=\im(\tau)\to\infty$. In this situation, \eqref{holE} needs to be regularized. Moreover, there is a ``companion integral'' (again regularized)

\begin{equation}\label{nonE}
\int_{-\overline{\tau}}^{i\infty} g(w)(w+\tau)^{k-2} dw,
\end{equation}
where $g$ is a certain weakly holomorphic modular form related to $f$
in the sense that the corresponding period polynomial, defined analogously to \eqref{PP},  basically agrees with $R_f$.\\
 \indent In contrast, for half-integral weight modular forms there is no half-derivative and thus Bol's identity does not apply. However, one can formally define the analogue of \eqref{Eichler1}. This was first investigated by Zagier \cite{ZagierVass,ZagierQuantum} in connection to Kontsevich's ``strange'' function
\begin{equation*}\label{kontstrange}K(q):=\sum_{m\geq0}(q;q)_m,\end{equation*}
where for $m\in\N_0\cup\{\infty\}$, $(a;q)_m:=\prod_{j=0}^{m-1}(1-aq^j)$ denotes the usual {\it $q$-Pochhammer symbol}. The function $K(q)$ does not converge on any open subset of $\C$, but converges as a finite sum for $q$ a root of unity. Zagier's study of $K$ depends on the  identity
\begin{equation}
\label{sumoftails}
\displaystyle\sum_{m\geq0}\left(\eta(\tau)-q^{\frac1{24}}\left(q;q\right)_m\right)=\eta(\tau)D\left(\tau\right)+\frac12\widetilde{\eta}(\tau),
\end{equation}
\noindent
with $\eta(\tau):=q^{\frac{1}{24}}(q;q)_{\infty}=\sum_{m\geq 1}(\frac{12}{m})q^{\frac{m^2}{24}}$,  $D(\tau):=-\frac12+\sum_{m\geq1}\frac{q^m}{1-q^m}$ and $\widetilde{\eta}(\tau):=\sum_{m \geq 1} \left(\frac{12}{m}\right)m q^{\frac{m^2}{24}}$,
where $\left(\frac{\,\cdot\,}{\,\cdot\,}\right)$ denotes the extended Jacobi symbol.
The key observation of Zagier is that in (\ref{sumoftails}), the functions $\eta(\tau)$ and $\eta(\tau)D(\tau)$ vanish of infinite order as $\tau\rightarrow\frac hk \in \mathbb{Q}$. So at a root of unity $\zeta$, $K(\zeta)$ is essentially the limiting value of the Eichler integral of $\eta$, which Zagier showed has quantum modular properties. Roughly speaking, Zagier defined ``quantum modular forms'' to be functions $f:\mathcal{Q}\rightarrow\C$ $(\mathcal{Q}\subseteq \Q)$, such that the error of modularity ($M = (\begin{smallmatrix}
a&b\\c&d
\end{smallmatrix})\in\SL_2(\Z)$)
\begin{equation}\label{errormod}
f(\tau)-(c\tau+d)^{-k} f(M\tau)
\end{equation}
is ``nice''. The definition is intentionally vague to include many examples; in this paper we require \eqref{errormod} to be real-analytic. For example, $\widetilde f$ (recall $k\in\mathbb Z$ in this case) is a quantum modular form, since $R_f$ is a polynomial and thus real-analytic. Additional examples appear in the study of limits of quantum invariants of $3$-manifolds and knots \cite{ZagierQuantum}, Kashaev invariants of torus knots/links \cite{Hikami,Hikami2},
and partial theta functions \cite{FOR}.

Motivated in part by vertex operator algebra theory, further (but similar) examples of quantum modular forms
were investigated in the setup of characters of vertex algebra modules in \cite{BM15} and \cite{CMW}. These examples are given by characters of $M_{r,s}$, the atypical irreducible modules of the $(1,p)$-singlet algebra  for $p \geq 2$ \cite{BM15,CM2014}.  For $r=1$
and $1 \leq s \leq p-1$, they take the particularly nice shape
$$
\text{ch}_{M_{1,s}}(\tau)=\frac{F_{p-s,p}(p\tau)}{\eta(\tau)},
$$
where 
$$
F_{j,p}(\tau):=\sum_{m\in\Z}\text{sgn}\left(m+\frac{j}{2p}\right)q^{\left(m+\frac{j}{2p}\right)^2}
$$
is a {\it false theta function}. The function $F_{j,p}$ is called ``false theta'' since getting rid of the $\sgn$-factor yields the theta function $\sum_{m\in\Z}q^{(m+\frac{j}{2p})^2}$, which is a modular form of weight $\frac12$.
The quantum modularity of $F_{j,p}$ is now given by relating it to a non-holomorphic Eichler integral, as in \eqref{nonE}. To be more precise, set
(correcting a typographical error in \cite{BM15})
$$
F_{j,p}^*(\tau):=- \sqrt{2} i \int_{-\overline{\tau}}^{i\infty}\frac{f_{j,p}(w)}{\left(-i(w+\tau)\right)^{\frac12}}dw,
$$
where $f_{j,p}$ is the cuspidal theta function of weight $\frac32$
$$
f_{j,p}(\tau):=\sum_{m\in\Z}\left(m+\frac{j}{2p}\right)q^{\left(m+\frac{j}{2p}\right)^2}.
$$
One can show that $F_{j,p}(\tau)$ agrees for $\tau=\frac{h}{k}$ with $F_{j,p}^*(\tau)$ up to infinite order \cite{BM15}. Quantum modularity then follows by the (mock) modular transformation of $F^\ast_{j, p}$ which we recall in Lemma \ref{lquant} below. By ``mock-modular'', we mean that the occurence of the extra term $r_{f,\frac{d}{c}}$ in Lemma \ref{lquant} prevents the function from being modular. However, there exists a ``modular completion'' in the sense that after multiplying it with a theta function, $F_{j,p}^*$ is the ``purely non-holomorphic part'' of a non-holomorphic theta function corresponding to an indefinite quadratic form (of signature $(1,1)$).
Its modularity now can be proven by using results of Zwegers \cite[Section 2.2]{Zw}.  The functions $\tau\mapsto F_{j,p}(p\tau)$, especially for $p=2$, have appeared in several studies of vertex algebras from different standpoints \cite{Adamovic-Milas,CM2014,GR,KW}.

In this paper we investigate higher-dimensional analogues. For this we consider certain $q$-series appearing in representation theory
of vertex algebras and $W$-algebras. They are sometimes called  {\it higher rank false theta functions} and are thoroughly studied in \cite{BM15,CM}. They appear from extracting the constant term of certain multivariable Jacobi forms \cite{BM15}. The constant term can be interpreted as the character of  the zero weight space of the corresponding Lie algebra representation. In the case of the simple Lie algebra $\frak{sl}_3$, the false theta function takes the following shape ($p \in \mathbb{N}$, $p \geq 2$) 
\begin{equation}\label{thetaLie}
F(q):= \!\!\!\! \sum_{m_1,m_2 \geq 1 \atop{ m_1 \equiv m_2 \pmod{3}}} \!\!\!\!\!\!\!\! {\rm min}(m_1,m_2) q^{\frac{p}{3}\left(m_1^2+m_2^2+m_1m_2\right)-m_1-m_2+\frac1p} \left(1-q^{m_1}\right)\left(1-q^{m_2}\right)\left(1-q^{m_1 + m_2}\right).
\end{equation}
Below we decompose this function as $F(q)=\frac{2}{p}F_1(q^p)+2F_2(q^p)$ with $F_1$ and $F_2$ defined in \eqref{defineF1} and \eqref{defineF2}, respectively.  The function $F_1$ and $F_2$ turn out to have generalized quantum modular properties. This connection goes via an analouge of \eqref{Eichler1}. For instance, we show that $F_1$ asymptotically agrees with an integral of the shape 
$$
\int_{-\overline{\tau}}^{i\infty}\int_{w_1}^{i\infty} \frac{f(w_1,w_2)}{\sqrt{-i(w_1+\tau)}\sqrt{-i(w_2+\tau)}}dw_2dw_1
$$
where $f\in S_{\frac{3}{2}}(\chi_1,\Gamma)\otimes S_{\frac32}(\chi_2,\Gamma)$ ($\chi_j$ are multipliers and $\Gamma\subset\SL_2(\Z)$).
Modular properties follow from the modularity of $f$ which in turn gives quantum modular properties of $F_1$.
The idea is that here the error of modularity \eqref{errormod} is less complicated than the original function.  We call the resulting functions higher depth
quantum modular forms (see Definition \ref{defgen} for a precise definition). Roughly speaking  (see Definition \ref{defgen} for a precise definition), depth two quantum modular forms of weight $k\in \frac12\Z$ satisfy, in the simplest case, the modular transformation property $(M=\left(\begin{smallmatrix} a & b \\ c & d
\end{smallmatrix}\right) \in \SL_2(\Z))$
%For example, depth two quantum modular satisfy the modular transformation property ($M=\left(\begin{smallmatrix} a & b \\ c & d
%\end{smallmatrix}\right) \in \Gamma\subset\SL_2(\Z)$)
$$
f(\tau)-(c\tau+d)^{-k}f(M\tau) \in  \mathcal{Q}_{\kappa} (\Gamma)\mathcal O(R) + \mathcal O(R)
$$
for some $\kappa\in\frac12\mathbb Z$, where  $\mathcal{Q}_{\kappa} (\Gamma)$ is the space of quantum modular forms of weight $\kappa$ and $\mathcal O(R)$ the space of real analytic functions on $R\subset\R$.
Clearly, we can construct examples of depth two simply by multiplying two (depth one) quantum modular forms.
Non-trivial examples arise from $F$ (see Theorem \ref{maintheorem} for precise statement). 

\begin{theorem}\label{maintheorem}	
	For $p \geq 2$, the higher rank false theta function $F$ can be written as the sum of two depth two quantum modular forms (with quantum set $\mathbb{Q}$) of weight one and two.
\end{theorem}
It is worth noting that all of our examples of quantum modular forms, including those studied in \cite{BM15}, have  $\mathbb{Q}$ as quantum set. Even though this feature is rare, a possible explanation is that vertex algebra characters are
generally better behaved functions and are expected to combine into vector-valued families under the full modular group. Thus in our future work \cite{BKM} we explore a vector-valued generalization of this theorem and its consequences to representation theory.

Zwegers \cite{Zw} found an important connection between the error term of the Eichler integral (as in Lemma \ref{lquant}) and classical Mordell integrals. This result applied to the case of $F^*_{j,p}$ leads to an elegant expression for the error term as a Mordell integral
$$\int_{\R} \cot\left(\pi i w+\frac{\pi j}{2p}\right)e^{2\pi i p w^2\tau}dw.$$
In this work we encounter error terms for iterated (double) Eichler integrals, so it is natural  to attempt to extend Zwegers' result to two dimensions. In \cite{BKM} we solve this problem in several special cases. In particular, we find that relevant integrals for the weight one component $\mathcal{E}_1$ (cf. Lemma \ref{quantE}) take the form
$$\int_{\R^2}
\cot\left(\pi iw_1+\pi\alpha_1\right)
\cot\left(\pi iw_2+\pi\alpha_2\right)e^{2\pi i\left(3w_1^2+3w_1 w_2 +w_2^2\right) \tau}
dw_1dw_2,$$
for some scalars $\alpha_1, \alpha_2$. This is what we call  a {\it double Mordell} integral.
We next turn to the modular completion of these Eichler integrals (see Propostiton \ref{Prop:CompletionOne} for a more precise version). For theta functions associated to indefinite quadratic forms, the reader is referred to \cite{ABMP, Ku, Na, WR}.
\begin{theorem}\label{indef}
	There exists an indefinite theta function, defined via \eqref{defineTheta}, of signature $(2,2)$ with ``purely non-holomorphic'' part $\Theta(\tau)\mathcal E_{1}(\tau)$ where $\Theta$ is a theta function of signature $(2,0)$ and the Eichler integral $\mathcal E_1$ is defined in \eqref{defineE1}.
\end{theorem}
The paper is organized as follows.
In Section 2, we review basic results on special functions, non-holomorphic Eichler integrals, and ``double error'' functions. We also recall the notion of quantum modular forms and introduce higher depth quantum modular forms. In Section 3,  the $\frak{sl}_3$ higher rank false theta function $F(q)=\frac{2}{p}F_1(q^p)+2F_2(q^p)$ is introduced. In Section 4,  we determine the asymptotic behavior of $F_1$ and $F_2$ at roots of unity. In Section 5, we introduce multiple Eichler integrals and prove modular transformation formulas for the double Eichler integrals. We also  study certain linear combinations of double Eichler integrals associated to $F_j$. In Section 6, we express special double Eichler integrals as pieces of indefinite theta series. Based on results in this section, in Section 7, we prove the main result, Theorem 1.1, on the quantum modularity of $F$. Section 8 deals with the completion of certain indefinite theta functions of signature $(2,2)$ associated to the companions of $F_j$ proving Theorem \ref{indef}. We conclude in Section 9 with several questions.

\section*{Acknowledgments} The research of the first author is supported by the Alfried Krupp Prize for Young University Teachers of the Krupp foundation
and the research leading to these results receives funding from the European Research Council under the European
Union's Seventh Framework Programme (FP/2007-2013) / ERC Grant agreement n. 335220 - AQSER. The research of the second author is partially supported by the European Research Council under the European
Union's Seventh Framework Programme (FP/2007-2013) / ERC Grant agreement n. 335220 - AQSER.
The research of the third author was partially supported by the Simons Foundation Collaboration Grant for Mathematicians ($\#$ 317908),
NSF grant DMS-1601070, and a stipend from the Max Planck Institute for Mathematics, Bonn.
The authors thank Chris Jennings-Shaffer, Josh Males, Boris Pioline, and Larry Rolen for helpful comments. Finally we thank the referee for providing useful comments.
\section{Preliminaries}

\subsection{Special functions}
Define, for $u\in\R$, 
$$
E(u):=2\int_{0}^{u}e^{-\pi w^2}dw.
$$
This function is essentially the error function and
its derivative is
$
E'(u)=2e^{-\pi u^2}.
$
We have the representation
\begin{equation}\label{EG}
E(u)=\sgn(u)\left(1-\frac{1}{\sqrt{\pi }}\Gamma\left(\frac12,\pi u^2\right)\right),
\end{equation}
where $\Gamma({\alpha},u):=\int_u^\infty e^{-w}w^{{\alpha}-1}dw$ is the {\it incomplete gamma function} and where for $u\in\R$, we set
$$\sgn(u):=\begin{cases}
1 & \text{if $u>0$,}\\
-1 & \text{if $u<0$,}\\
0 & \text{if $u=0$.}
\end{cases}
$$
We also require the functional equation of the incomplete $\Gamma$-function {with $\alpha=\frac12$}
\begin{equation}\label{G12}
\Gamma\left(\frac12,u\right) = -\frac12 \Gamma\left(-\frac12,u\right) + \frac{1}{\sqrt{u}} e^{-u}.
\end{equation}
%\left\{ \begin{array}{c}  \ \ 1, \ \ x>0  \\   -1, \ \ x<0. \end{array}\right.$$
Moreover, for $u \neq 0$, set
$$
M(u) := \frac{i}{\pi}\int_{\R-iu}\frac{e^{-\pi w^2-2\pi i uw}}{w}dw.
$$
We have
\begin{equation*}\label{defineM}
M(u)=E(u)-\sgn(u).
\end{equation*}
Thus, by \eqref{EG}
\begin{equation}\label{MG}
M(u) = -\frac{\sgn(u)}{\sqrt{\pi}} \Gamma\left(\frac12,\pi u^2\right).
\end{equation}
This implies that the following bound holds
\begin{equation*}\label{M1bound}
\lvert M(u)\rvert \leq 2e^{-\pi u^2}.
\end{equation*}

We next turn to two-dimensional analogues, following \cite{ABMP} (using slightly different notation).
%{\bf AM: do we explain? KB: What do you want to explain? One can easily see when looking at the paper. I just want to alert readers.}).
Define $E_2:\R\times\R^2\rightarrow\R$ by (throughout we use bold letters for vectors and denote their components using subscripts)
$$
E_2(\kappa;\boldsymbol{u}):=\int_{\R^2}\sgn\left(w_1\right)\sgn\left(w_2+\kappa w_1\right)e^{-\pi \left(\left(w_1-u_1\right)^2+\left(w_2-u_2\right)^2\right)}dw_1dw_2.
$$
Note that
\[
E_2(\kappa;-\boldsymbol{u})=E_2(\kappa;\boldsymbol{u}).
\]
Moreover, also following \cite{ABMP}, for $u_2, u_1- \kappa u_2\neq 0$ we set 

$$M_2(\kappa; \boldsymbol{u}):=-\frac{1}{\pi^2} \int_{\mathbb{R} - i u_2}  \int_{\mathbb{R}-iu_1} \frac{e^{-\pi w_1^2-\pi w_2^2-2 \pi i (u_1 w_1+u_2 w_2)}}{w_2(w_1-\kappa w_2)}dw_1dw_2. $$
Then we have 
\begin{equation}\label{M2E2}\begin{split}
M_2\left(\kappa;\boldsymbol{u}\right)&=E_2\left(\kappa;\boldsymbol{u}\right)-\sgn\left(u_2\right)M\left(u_1\right)\\
&\quad-\sgn\left(u_1-\kappa u_2\right)M\left(\frac{u_2+\kappa u_1}{\sqrt{1+\kappa^2}}\right)-\sgn\left(u_1\right)\sgn\left(u_2+\kappa u_1\right).
\end{split}\end{equation}
Note that \eqref{M2E2} extends the definition of $M_2$ to $u_2=0$ or $u_1=\kappa u_2$.
With $x_1:=u_1-\kappa u_2$, $x_2 := u_2$, a direct calculation shows that
\begin{align*}
M_2\left(\kappa; \boldsymbol{u}\right)&=E_2\left(\kappa; x_1+\kappa x_2, x_2\right)+\sgn(x_1)\sgn(x_2)\\
&\qquad-\sgn(x_2)E(x_1+\kappa x_2)-\sgn(x_1)E\left(\frac{\kappa x_1}{\sqrt{1+\kappa^2}}+\sqrt{1+\kappa^2}x_2\right).
%\label{M2sgn}
\end{align*}

We have the first partial derivatives
\begin{align}\label{diffM2}
M_2^{(0,1)}\left(\kappa;\boldsymbol{u}\right)&=\frac{2}{\sqrt{1+\kappa^2}} e^{-\frac{\pi(u_2+\kappa u_1)^2}{1+\kappa^2}} M\left(\frac{u_1-\kappa u_2}{\sqrt{1+\kappa^2}}\right),\\
\label{middM1}
M_2^{(1,0)}\left(\kappa;\boldsymbol{u}\right)&=2e^{-\pi u_1^2}M(u_2)+\frac{2\kappa}{\sqrt{1+\kappa^2}} e^{-\frac{\pi(u_2+\kappa u_1)^2}{1+\kappa^2}} M\left(\frac{u_1-\kappa u_2}{\sqrt{1+\kappa^2}}\right),
\end{align}
and the limiting behavior (cf.  \cite[Proposition 3.3, iii]{ABMP})

\begin{align}\label{limitsM2}
M_2\left(\kappa; \lambda \boldsymbol{u}\right)&\sim -\frac{e^{-\pi\lambda^2 \left(u_1^2+u_2^2\right)}}{\lambda^2\pi^2 u_2(u_1-\kappa u_2) } \qquad (\text{as} \quad \lambda\rightarrow \infty).
\end{align}
\begin{lemma}\label{limitsE2}
	For $u_3,u_4+\kappa u_3\neq 0$, we have the following limits
	\begin{align*}
	\lim_{\varepsilon\to 0^+}E_2\left(\varepsilon\kappa; u_1, \varepsilon u_2+\varepsilon^{-1}u_3 \right)
	&=\sgn(u_3)E(u_1),
	\\
	\lim_{\varepsilon\to 0^+}E_2\left(\kappa; \varepsilon u_1+\varepsilon^{-1} u_3, \varepsilon u_2+\varepsilon^{-1}u_4 \right)
	&=\sgn(u_3)\sgn(u_4+\kappa u_3).
	\end{align*}
\end{lemma}
\begin{proof}
	We only prove the first statement, the second follows analogously.
		We may compute the limit inside the integral due to the convergence of the dominating integral $\int_{\R^2}e^{-\pi( w_1^2+w_2^2) }dw=1$ to obtain
		\begin{align*}
		&\quad \lim_{\varepsilon\to 0^+}E_2\left(\varepsilon\kappa; u_1, \varepsilon u_2+\varepsilon^{-1}u_3 \right)
		\\&\qquad\quad=\int_{\R^2}e^{-\pi \left( w_1^2+w_2^2\right)}  \sgn\left(w_1+u_1\right)
		\lim_{\varepsilon\to 0^+}  \sgn\left(u_3+\varepsilon\left(w_2+\varepsilon\kappa w_1+\varepsilon u_2+\varepsilon\kappa u_1 \right) \right) dw_2dw_1
		\\&\qquad\quad=\int_{\R}e^{-\pi w_1^2} \sgn\left(w_1+u_1\right)\int_{\R} e^{-\pi w_2^2} \sgn\left(u_3 \right) dw_2 dw_1=\sgn(u_3) E(u_1).
		\end{align*}
\end{proof}

\subsection{Euler-Maclaurin summation formula}
We now state a special case of the Euler-Maclaurin summation formula. We only give it in the two-dimensional case; the one-dimensional case can be concluded by viewing the second variable as constant.

Let $B_m(x) $ be the $m$-th Bernoulli polynomial  defined by $\frac{t e^{xt}}{e^t - 1} =: \sum_{m \geq 0} B_m(x) \frac{t^m}{m!}$. We also require
\begin{equation*} \label{X1X}
B_m(1-x) = (-1)^m B_m(x).
\end{equation*}

\begin{comment}
The following result is stated as the first generalization of Proposition 3 in \cite{Zag06}.
\begin{proposition}
	\label{P:Zag}
	Suppose that $f$ has the asymptotic expansion \eqref{E:ft} and that $f$ and all of its derivatives are of rapid decay at infinity.  Suppose further that $I_f := \int_0^\infty f(u) du$ converges.  The function $g$ as defined in \eqref{E:gt} then has the asymptotic expansion
	\begin{equation*}
	g(t) = \frac{I_f}{t} - \sum_{n = 0}^N b_n \frac{B_{n+1}(a)}{n+1} t^n.
	\end{equation*}
\end{proposition}
\end{comment}
The Euler-Maclaurin summation formula implies that, for $\boldsymbol{\alpha}\in\R^2$, $F:\R^2\rightarrow\R$ a $C^\infty$-function which has rapid decay, we have (generalizing a result of \cite{Zag76} to include shifts by $\alpha$) 
\begin{equation}\label{EulerMcLaurin}
\begin{split}
&\sum_{\boldsymbol{n}\in\N_0^2}  F((\boldsymbol{n}+\boldsymbol{\alpha})t) \sim \frac{\mathcal I_F}{t^2} - \sum_{n_2 \geq 0} \frac{B_{n_2+1}(\alpha_2)}{(n_2+1)!} \int_{0}^{\infty}F^{(0,n_2)}(x_1,0)dx_1t^{n_2-1} \\ &\quad - \sum_{n_1 \geq 0} \frac{B_{n_1+1}(\alpha_1)}{(n_1+1)!}\int_{0}^{\infty}F^{(n_1,0)}(0,x_2)dx_2t^{n_1-1} + \sum_{n_1,n_2\geq 0}\frac{B_{n_1+1}(\alpha_1)}{(n_1+1)!}\frac{B_{n_2+1}(\alpha_2)}{(n_2+1)!} F^{(n_1,n_2)}(0,0)t^{n_1+n_2},
\end{split}
\end{equation}
where $\mathcal I_F:=\int_{0}^{\infty}\int_{0}^{\infty} \allowbreak F(\boldsymbol{x})dx_1dx_2$. 
Here by $\sim$ we mean that the difference between the left- and the right-hand side is $O(t^N)$ for any $N\in \N$.

\subsection{Shimura's theta functions}
We require transformation laws of certain theta functions studied, for example, by Shimura \cite{Sh}.
For $\nu\in\{0,1\}$, $h\in\Z$, $N,A\in\N$, with $A|N$, $N|hA$, define
\begin{equation}\label{shimura}
\Theta_\nu (A,h,N;\tau):=\sum_{\substack{m\in\Z \\ m\equiv h\pmod{N}}}m^\nu  q^{\frac{Am^2}{2N^2}}.
\end{equation}
%Then
%\begin{equation}\label{Thetainv}
%\Theta_\nu\left(A,h,N;-\frac{1}{\tau}\right)=(-i)^\nu (-i\tau)^{\frac12+\nu }A^{-\frac12}\sum_{\substack{k\pmod{N} \\ Ak\equiv 0 \pmod{N}}}e\left(\frac{Akh}{N^2}\right)\Theta_\nu (A,k,N;\tau).
%\end{equation}
%Moreover, we require the splitting ($c\in\N$)
%\begin{equation}\label{split}
%\Theta_\nu(A,h,N;\tau)=\sum_{\substack{g\equiv h\pmod{N}\\g\pmod{cN}}}\Theta_\nu(cA,g,cN;c\tau)\
%\end{equation}
Recall the following modular transformation 
\begin{equation}\label{Shimura1}
\Theta_\nu (A,h,N;M\tau) = e\left(\frac{abAh^2}{2N^2}\right)\left(\frac{2Ac}{d}\right)\varepsilon_d^{-1} (c\tau+d)^{\frac12+\nu}\Theta_\nu(A,ah,N;\tau)
\end{equation}
for $M=\left(\begin{smallmatrix} a&b\\c&d \end{smallmatrix}\right) \in \Gamma_0(2N)$ with $2|b$. Here $e(x):=e^{2\pi ix}$, for odd $d$, $\varepsilon_d=1$ or $i$, depending on whether $d\equiv 1\pmod{4}$ or $d\equiv 3\pmod{4}$. Also note that if $h_1\equiv h_2\pmod{N}$, then we have
\begin{equation*}\label{Shimura2}
\Theta_\nu(A,h_1,N;\tau)=\Theta_\nu(A,h_2,N;\tau),\qquad \Theta_\nu(A,-h,N;\tau)=(-1)^\nu\Theta_\nu(A,h,N;\tau).
\end{equation*}

\subsection{Indefinite theta functions}
	We begin by defining (possibly indefinite) theta functions.
	\begin{definition}
		Let $A\in M_m(\Z)$ be a non-singular symmetric $m\times m$ matrix, $P:\R^{m}\rightarrow \C$ and $\boldsymbol{a}\in\Q^m$. We define the associated theta function by ($\tau= u+iv$) 
		\begin{equation*}
		\Theta_{A,P,\boldsymbol{a}}(\tau):= \sum_{\boldsymbol{n}\in \boldsymbol{a}+\Z^m}P\left(\sqrt{v}\boldsymbol{n}\right)q^{\frac12 \boldsymbol{n}^TA\boldsymbol{n}}.
		\end{equation*}
	\end{definition}
	The following theorem shows that under certain conditions $\Theta_{A,P,\boldsymbol{a}}$ is modular.
	
	\begin{theorem}[Vign\'eras, \cite{Vi}]\label{Vigneras}
		Suppose that $A\in M_m(\Z)$ is non-singular and that $P$ satisfies the following conditions:  {\normalfont
			\begin{enumerate}[leftmargin=*]
				\item
				{\it For any differential operator $D$ of order two and any polynomial $R$ of degree at most two, we have that $D(\boldsymbol{w})(P(\boldsymbol{w})e^{{\pi}Q(\boldsymbol{w})})$ and $R(\boldsymbol{w})P(\boldsymbol{w})e^{{\pi}Q(\boldsymbol{w})}$ belong to $L^2(\R^m)\cap L^1(\R^m)$.}
				
				\item
				{\it 
					For some $\lambda\in\Z$ the Vign\'eras differential equation holds:
					$$
					\left(\mathcal{D}-\frac{1}{4\pi}\Delta\right)P=\lambda P.
					$$
				Here we define the Euler and Laplace operators ($\boldsymbol{w}:=(w_1,\dots,w_m),\ \partial_{\boldsymbol{w}}:=(\frac{\partial}{\partial w_1},\dots \frac{\partial}{\partial w_m})^T$)

				\[
				\mathcal{D}:= \boldsymbol{w}\partial_{\boldsymbol{w}}\qquad\text{and}\qquad
				\Delta=\Delta_{A^{-1}}:=\partial_{\boldsymbol{w}}^T A^{-1}\partial_{\boldsymbol{w}}
				.
				\]}
			\end{enumerate}
		}
		\noindent Then, assuming that $\Theta_{A,P,\boldsymbol{a}}$ is absolutely locally convergent, $\Theta_{A,P,\boldsymbol{a}}$ is modular of weight $\lambda+\frac{m}{2}$ for some subgroup of $\SL_2 (\Z)$.
	\end{theorem}

\subsection{Quantum modular forms}
We already motivated quantum modular forms in the introduction. The formal definition is as follows \cite{ZagierQuantum}.

\begin{definition}
	A function $f:\mathcal Q\to \C$ (here $\mathcal{Q} \subseteq \mathbb{Q}$) is called a {\it quantum modular form of weight $k \in \frac12\Z$ and multiplier $\chi$ for a subgroup $\Gamma$ of $\SL_2(\Z)$ and quantum set $\mathcal Q$} if for $M=\left(\begin{smallmatrix}
	a & b\\ c & d
	\end{smallmatrix}\right)\in\Gamma$, the function
	$$
	f(\tau)-\chi(M)^{-1}(c\tau+d)^{-k}f(M\tau)
	$$
	can be extended to an open subset of $\R$ and is real-analytic there. We denote the vector space of such forms by $\mathcal Q_k(\Gamma,\chi)$.
\end{definition}

\begin{remark}
	Zagier also considered {\it strong quantum modular forms}. Here one is looking at asymptotic expansions instead of just values.
\end{remark}

The introduction already gives examples of quantum modular forms. As mentioned there, the functions $F_{j,p}$ satisfy modular type transformations making them quantum modular forms. More generally, for $f\in S_k(\Gamma,\chi)$, the space of cusp forms of weight $k$ transforming as
\[
f(M\tau)=(c\tau+d)^k \chi(M) f(\tau)
\]
for $M=\left(\begin{smallmatrix} a&b\\c&d\end{smallmatrix}\right)\in\Gamma\subset\SL_2(\Z)$
and $\chi$ some multiplier, we set, for $\frac{d}{c}\in\Q$,

\begin{align}\label{Eichler}
I_f(\tau)&:=\int_{-\overline{\tau}}^{i\infty}\frac{f(w)}{\left(-i(w+\tau)\right)^{2-k}}dw,\qquad
r_{f, \frac dc}(\tau):=\int_{\frac dc}^{i\infty}\frac{f(w)}{\left(-i(w+\tau)\right)^{2-k}}dw.
\end{align}
For weight $k=\frac12$, we allow $f\in M_{\frac12}(\Gamma,\chi)$, the space of holomorphic modular forms of weight $\frac12$.
To state the modularity properties of $I_f$, we let $\Gamma^\ast :=P \Gamma P^{-1}$,
where $P:=\left(\begin{smallmatrix} -1 & 0 \\ 0 & 1 \end{smallmatrix} \right)$.
The proof of the following lemma follows along the same lines as the proof of Theorem \ref{quantheorem} below.
\begin{lemma}\label{lquant}
We have the transformation, for $M\in\Gamma^\ast$,
\[
I_f(\tau)-\chi^{-1}\left(M^\ast\right)(c\tau+d)^{k-2} I_f\left(M \tau\right)=r_{f, \frac{d}{c}}(\tau).
\]
The function $I_f$ is defined on $\H\cup\Q$ whereas $r_{f, \frac dc}$ exists on all of $\R\setminus \{-\frac{d}{c}\}$ and is real-analytic there. If $f\in S_k(\Gamma,\chi)$, then $r_{f,\frac{d}{c}}$ exists on $\R$.
\end{lemma}

\subsection{Higher Depth Quantum modular forms}
We next turn to generalizations of quantum modular forms.

%\begin{definition} We say that ordinary (resp. vector-valued) quantum modular forms are (resp. vector-valued) quantum modular forms  of depth one. For $m \geq 2$, a collection $f_1,\dots,f_N$ ($N\in\N$) is called a {\it depth $m$ vector-valued quantum modular form of weight $k$}  if there exist $\chi_{j,\ell}\in\C$ ($1\le j,\ell\le N$) such that for all $1\le j\le N$ {\bf AM: why only $S$-transformation? KB: same reply.}
%$$
%f_j\left(-\frac{1}{\tau}\right)-\tau^k \sum_{1\le\ell\le N} \chi_{j,\ell}f_\ell(\tau)\in\mathcal{O}(\R)\mathcal{Q}^{m-1},
%$$
%where $\mathcal{Q}^{m-1}$ is the space of (vector-valued) depth $m-1$ quantum modular forms. {\bf AM: Something wrong with this definition. We should first define non-vector valued. KB: what is wrong? AM: same as before "collection"} \end{definition}

\begin{definition}\label{defgen}
	A function $f:\mathcal Q\to\C$ ($\mathcal Q\subset \Q$) is called a {\it quantum modular form of depth $N\in\N$, weight $k \in \frac12\Z$, multiplier $\chi$, and quantum set $\mathcal Q$ for $\Gamma$} if for $M=\left(\begin{smallmatrix}
	a & b\\ c & d
	\end{smallmatrix}\right)\in\Gamma$
	
	$$
	f(\tau)-\chi(M)^{-1}(c\tau+d)^{-k}f(M\tau)\in  \bigoplus_j \mathcal Q_{\kappa_j}^{N_j}(\Gamma,\chi_j)\mathcal O(R),
	$$
	where $j$ runs through a finite set, $\kappa_j\in \frac12\Z$, $N_j \in \N$ with $\max_j(N_j)=N-1$, the $\chi_j$ are characters, $\mathcal O(R)$ is the space of real-analytic functions on $R \subset \R$ which contains an open subset of $\R$, $\mathcal Q_k^1(\Gamma,\chi) := \mathcal Q_k(\Gamma,\chi)$, $\mathcal Q_k^0(\Gamma,\chi):=1$, and $\mathcal Q_{k}^{N}(\Gamma,\chi)$ denotes the space of quantum modular forms of weight $k$, depth $N$, multiplier $\chi$ for $\Gamma$.
\end{definition}

\begin{remark}
	Again one can consider {\it higher depth strong quantum modular forms} by looking at asymptotic expansions instead of values. The examples of this paper satisfiy this stronger property.
\end{remark}

\begin{example} For $f_1 \in \mathcal Q_{k_1}^1(\Gamma_1,\chi_1)$ and $f_2 \in  \mathcal Q_{k_2}  ^1(\Gamma_2,\chi_2)$, we have that  $f_1  f_2 \in Q_{k_1+k_2}^2(\Gamma_1 \cap \Gamma_2,\chi_1 \chi_2)$.
\end{example}

\section{A rank two false theta function}

We briefly recall a construction from \cite{BM16,CM, FT}. For $p \in \mathbb{N}_{\geq 2}$, there is a vertex operator algebra $W(p)_{A_2}$ associated to the simple Lie algebra $\frak{sl}_3$ (more precisely, to its root lattice of type $A_2$). The  
character formula of $W(p)_Q$, where $Q$ is any ADE root lattice, was proposed in \cite{FT} (note that some arguments in \cite{FT} 
are not completely rigorous) and further studied in  \cite{BM16,CM,FT}; see also \cite{Adamovic}. Letting $\zeta_j:=e^{2\pi iz_j}$, we have \cite{BM16,CM}
%proposed full character  ${\rm ch}[{W(p)_{A_2}}](\tau,z)$ \cite{BM16,CM,FT} (see also \cite{Adamovic})  satisfies %%($z=(z_1,z_2)$, $\zeta_j:=e^{2\pi iz_j}$) 
\begin{align*}
&\eta(\tau)^2 {\rm ch}[{W(p)_{A_2}}](\tau,\boldsymbol{z})
%\sum_{m,n  \in \mathbb{Z}}  \frac{q^{p (m+1)^2+ p (n+1)^2- p (m+1)(n+1)-m-n-2+\frac{1}{p}}}%{\left(1-z_1^{-1}\right)\left(1-z_2^{-1}\right)\left(1-z_1^{-1}z_2^{-1}\right)}  \\
%& \times \left(z_1^m z_2^n -z_1^{-m+n-1} z_2^{n}-z_1^{m} z_2^{-n+m-1}+z_1^{-n-2}z_2^{-n+m-1}+z_1^{-m+n-1}z_2^{-m-2}-z_1^{-n-2}z_2^{-%m-2}\right)
%\\
=\!\!\!\!\!\sum_{m_1, m_2\in\Z}\!\!\!\!\!\frac{q^{p\left(\left(m_1-\frac{1}{p}\right)^2+\left(m_2-\frac{1}{p}\right)^2-\left(m_1-\frac{1}{p}\right)\left(m_2-\frac{1}{p}\right)\right)}}{\left(1-\zeta_1^{-1}\right)\left(1-\zeta_2^{-1}\right)\left(1-\zeta_1^{-1}\zeta_2^{-1}\right)}\left(\zeta_1^{m_1-1} \zeta_2^{m_2-1} \!\!-\zeta_1^{-m_1+m_2-1}\! \right.\\
&\qquad\quad\times \zeta_2^{m_2-1}-\zeta_1^{m_1-1}\zeta_2^{-m_2+m_1-1}\!+\left.\zeta_1^{-m_2-1} \zeta_2^{-m_2+m_1-1}\!+\zeta_1^{-m_1+m_2-1} \zeta_2^{-m_1-1}\!-\zeta_1^{-m_2-1}\zeta_2^{-m_1-1}\right).
\end{align*}
The six term expression in the numerator comes from the summation over the Weyl group $W $ of $\frak{sl}_3$ which is isomorphic to $S_3$.
Thanks to Weyl's character formula, the rational $\boldsymbol{z}$-part is in fact a Laurent polynomial.
There are two
important operations on this character:
\begin{itemize}
\item[(1)] taking the limit $\boldsymbol{z}=(z_1,z_2) \to (0,0)$, yielding a modular form \cite{BM16};
\item[(2)]  taking the constant term
$${\rm ch}[W^0(p)_{A_2}](\tau):= {\rm CT}_{\zeta_1,\zeta_2} {\rm ch}[{W(p)_{A_2}}](\tau,\boldsymbol{z}),$$ which computes the character of another vertex algebra. It was shown in \cite{BM16} that
\begin{equation*}
{\rm ch}[{W}^0(p)_{A_2}](\tau)=\frac{F(q)}{\eta(\tau)^2} .
\end{equation*}
\end{itemize}

Note that formulas like  $\eta(\tau)^{{\rm rank}(Q)}{\rm ch}[{W}^0(p)_{Q}](\tau)$, where $Q$ is any root lattice, are of interest
beyond vertex algebra theory \cite{BM16,CM}. The coefficients appearing in the $q$-expansion are essentially dimensions of the zero
weight spaces of finite-dimensional irreducible representations of simple Lie algebras (for the recent progress in understanding these numbers see  \cite{KP}).

%Motivated by the rank one, it was conjectured in \cite{CM2} and that there is a connection between mock modular forms and this character.
\begin{remark} Modular-type properties of regularized (or Jacobi) characters, in particular ${\rm ch}[W^0(p)^\varepsilon_{A_2}](\tau)$, were investigated in \cite{CM} (see also \cite{CM2014}). There are two important differences between the current work and \cite{CM}. In this paper, the value of the Jacobi parameter $\varepsilon$  is always zero whereas in \cite{CM} it is necessarily non-zero. Secondly, there seems to be no clear connection between transformation formulas appearing in \cite{CM} and mock modular forms. On the other hand, here we make this connection quite explicit by virtue of generalized Eichler integrals (see Section 5).  \end{remark}

Let $n_1=m_1-m_2, n_2=m_2$ in \eqref{thetaLie} and then change $n_1\mapsto 3n_1$. Then we have, with $F$ given in \eqref{thetaLie}, 
\begin{align*}
\frac12F(q)=f_1(q)+f_2(q)+f_3(q),
\end{align*}
%where
%\begin{align*}
%&g(q):= \sum_{n\geq1} n q^{pn^2 - 2n + \frac{1}{p}}\left(1-q^n\right)^2 \left(1 - q^{2n}\right),
%\\& f(q):= \sum_{m_1>m_2\geq 1 \atop{m_1 \equiv m_2 \pmod{3}}}m_2 q^{t_p(m_1,m_2)}\left(1-q^{m_1}\right)\left(1-q^{m_2}\right)\left(1-q^{m_1 + m_2}\right).
%\end{align*}
%We first consider $g(q)$. A direct calculation shows that
%\[
%g(q)=g_1\left(q^p\right)+g_2\left(q^p\right)
%\]
%with
%\begin{align*}
%g_1(q)&:= -\frac{q^{\frac{3}{4p^2}}}{p}\sum_{n\in\Z}q^{\left(n+\frac{1}{2p}\right)^2}+\frac1p \sum_{n\in\Z}q^{\left(n+\frac1p\right)^2},\\
%g_2(q)&:= 2q^{\frac{3}{4p^2}}\sum_{n \in\Z}\left(n+\frac{1}{2p}\right)q^{p\left(n+\frac{1}{2p}\right)^2}- \sum_{n \in\Z}\left(n+\frac{1}{p}\right)q^{p\left(n+\frac{1}{p}\right)^2}.
%\end{align*}
%
%We next investigate $f(q)$. We let $n_1 = m_1 - m_2$, $n_2=m_2$. Then $n_1\geq 1$ and  $n_1 \equiv  0 \pmod{3}$. We get
%\begin{align*}
%f(q) &=\sum_{n_1, n_2 \geq 1 \atop{n_1 \equiv 0 \pmod{3}}}n_2 q^{p\left(\frac{n_1^2}{3}+ n_1 n_2+ n_2^2\right)+ \frac{1}{p}}\left(q^{-n_1-2n_2}-q^{-n_2}-q^{-n_1-n_2}+q^{n_2}+q^{n_1+n_2}-q^{n_1 +2n_2}\right)\\
%&= f_1(q)+f_2(q)+f_3(q).
%\end{align*}
where, with  $Q(\boldsymbol{x}):=3x_1^2 +3x_1x_2+x_2^2$, we define
\begin{align*}
&f_1(q):= q^{\frac{1}{p}}\sideset{}{^\ast}\sum_{n_1,n_2\geq 0} n_2q^{pQ(\boldsymbol{n})}\left(q^{-3n_1-2n_2}-q^{3n_1+2n_2}\right),\\
&f_2(q):=q^{\frac{1}{p}}\sideset{}{^\ast}\sum_{n_1,n_2\geq 0} n_2q^{pQ(\boldsymbol{n})}\left(q^{n_2}-q^{-n_2}\right), \\
&f_3(q):=q^{\frac{1}{p}}\sideset{}{^\ast}\sum_{n_1,n_2\geq 0} n_2q^{pQ(\boldsymbol{n})}\left(q^{3n_1+n_2}-q^{-3n_1-n_2}\right).
\end{align*}
Here $\sum^\ast$ means that the $n_1=0$ term is weighted by $\frac12$.
We then rewrite
\begin{align*}
f_1(q)
&=-\sum_{n_1, n_2\geq 0}\left(n_2+\frac1p\right) q^{p Q\left(n_1+1, n_2+\frac1p\right)}
+\sum_{n_1, n_2\geq 0}\left(n_2+1-\frac1p\right) q^{p Q\left(n_1, n_2+1-\frac1p\right)}\\
&\quad+\frac1p \sum_{n_1, n_2\geq 0}q^{pQ\left(n_1+1, n_2+\frac1p\right)}+\frac1p  \sum_{n_1, n_2\geq 0}q^{pQ\left(n_1, n_2+1-\frac1p\right)}-\frac12\sum_{m\geq 0}\left(m+\frac1p\right)q^{p\left(m+\frac1p\right)^2}\\
&\quad-\frac12\sum_{m\geq 0}\left(m+1-\frac1p\right)q^{p\left(m+1-\frac1p\right)^2}+\frac1{2p}\sum_{m\ge 0} q^{p\left(m+\frac1p\right)^2}-\frac1{2p}\sum_{m\ge 0} q^{p\left(m+1-\frac1p\right)^2},\\
f_2(q)
&=\sum_{n_1, n_2\geq 0}\left(n_2+\frac2p\right) q^{p Q\left(n_1+1-\frac1p, n_2+\frac2p\right)}
-\sum_{n_1, n_2\geq 0}\left(n_2+1-\frac2p\right) q^{p Q\left(n_1+\frac1p, n_2+1-\frac2p\right)}\\
&\quad-\frac2p\sum_{n_1, n_2\geq 0}q^{p Q\left(n_1+1-\frac1p, n_2+\frac2p\right)}
-\frac2p\sum_{n_1, n_2\geq 0}q^{p Q\left(n_1+\frac1p, n_2+1-\frac2p\right)}\\
&\quad+\frac{q^{\frac{3}{4p}}}{2}\sum_{m\geq 1} m q^{p\left(m-\frac1{2p}\right)^2}+\frac{q^{\frac{3}{4p}}}{2}\sum_{m\geq 1} m q^{p\left(m+\frac1{2p}\right)^2},\\
f_3(q)
&=\sum_{n_1, n_2\geq 0}\left(n_2+1-\frac1p\right) q^{pQ\left(n_1+\frac1p, n_2+1-\frac1p\right)}
-\sum_{n_1, n_2\geq 0}\left(n_2+\frac1p\right) q^{pQ\left(n_1+1-\frac1p, n_2+\frac1p\right)}\\
&\quad+\frac1p \sum_{n_1, n_2\geq 0} q^{pQ\left(n_1+\frac1p, n_2+1-\frac1p\right)}
+\frac1p \sum_{n_1, n_2\geq 0} q^{pQ\left(n_1+1-\frac1p, n_2+\frac1p\right)}\\
&\quad-\frac{q^{\frac{3}{4p}}}{2}\sum_{m\geq 1} m q^{p\left(m+\frac1{2p}\right)^2}-\frac{q^{\frac{3}{4p}}}{2}\sum_{m\geq 1}m q^{p\left(m-\frac1{2p}\right)^2}.
\end{align*}
We thus obtain 
\[
F(q)=\frac2pF_1\left(q^p\right)+2F_2\left(q^p\right)
\]
with
\begin{align}\label{defineF1}
F_1(q):=\sum_{\boldsymbol{\alpha}\in\mathscr S}\varepsilon(\boldsymbol{\alpha})\sum_{\boldsymbol{n}\in\boldsymbol{\alpha} +\N_0^2}q^{Q(\boldsymbol{n})}+\frac12 \sum_{m\in\Z}\sgn\left(m+\frac1p\right)q^{\left(m+\frac1p\right)^2},
\end{align}
where
\begin{align*}
\mathscr{S}:=&\left\{\left(1-\frac{1}{p},\frac{2}{p}\right),\left(\frac{1}{p},1-\frac{2}{p}\right),\left(1,\frac{1}{p}\right)\left(0,1-\frac{1}{p}\right),\left(\frac{1}{p},1-\frac{1}{p}\right),\left(1-\frac{1}{p},\frac{1}{p}\right)\right\},
\end{align*}
and for $\boldsymbol{\alpha}\pmod{\Z^2}$, we set
\begin{align*}
\varepsilon(\boldsymbol{\alpha}):=&\begin{cases} -2 \qquad&\text{if } \boldsymbol{\alpha} \in\left\{\left(1-\frac{1}{p},\frac{2}{p}\right),\left(\frac{1}{p},1-\frac{2}{p}\right)\right\},\\ 1 &\text{otherwise}.\end{cases}
\end{align*}
Moreover
\begin{align}\label{defineF2}
F_2(q)&:=\sum_{\boldsymbol{\alpha}\in\mathscr{S}}\eta(\boldsymbol{\alpha})\sum_{\boldsymbol{n}\in\boldsymbol{\alpha}+\N_0^2} n_2 q^{Q(\boldsymbol{n})}-\frac12\sum_{m\in\Z}\left|m+\frac1p\right|q^{\left(m+\frac1p\right)^2},
\end{align}
where for $\boldsymbol{\alpha}\pmod{\Z^2}$, we let
\begin{align*}
\eta(\boldsymbol{\alpha})&:=
\begin{cases}
1&\quad\text{ if } \boldsymbol{\alpha} \in\left\{\left(1-\frac1p, \frac2p\right), \left(0, 1-\frac1p\right), \left(\frac1p, 1-\frac1p\right)\right\},\\
-1&\quad\text{ otherwise.}
\end{cases}
\end{align*}
%Finally
%\begin{align*}
%F_{3}(q)&:=-q^{\frac{3}{4p^2}}\sum_{n\in\Z} \left(n+\frac1{2p}\right) q^{\left(n+\frac1{2p}\right)^2}+\sum_{n\in\Z}\left(n+\frac1p\right) q^{\left(n+\frac1{p}\right)^2}-\frac12\sum_{n\in\Z}\left|n+\frac1p\right|q^{\left(n+\frac1p\right)^2},\\
%F_{4}(q)&:=\frac{q^{\frac{3}{4p^2}}}{2p}\sum_{n\in\Z}q^{\left(n+\frac1{2p}\right)^2}-\frac1p\sum_{n\geq 1} q^{\left(n-\frac1{p}\right)^2}-\frac{1}{2p}\sum_{n\in \Z}q^{\left(n+\frac1p\right)^2}.
%\end{align*}

\section{Asymptotic behavior of $F_1$ and $F_2$}\label{sec:asympF1}

In this section we determine the asymptotic behavior of $F(e^{2\pi i\frac{h}{k}-t})$ ($h,k\in\Z$ with $k>0$ and $\gcd(h,k)=1$) as $t\to0^+$ and in particular show that the limit exists.

\subsection{The function $F_1$}\label{functionf1}
We decompose
$$
F_1(q) = F_{1,1}(q) + F_{1,2}(q),
$$
where
\begin{align*}
F_{1,1}(q)&:= \sum_{\boldsymbol{\alpha}\in\mathscr S}\varepsilon(\boldsymbol{\alpha}) \sum_{\boldsymbol{n}\in\boldsymbol{\alpha}+\N_0^2}q^{Q(\boldsymbol{n})},\qquad
F_{1,2}(q):= \frac12\sum_{m\in\frac1p+\Z}\sgn(m)q^{m^2}.
\end{align*}
We first study the asymptotic behavior of $F_{1,1}$, rewriting it in a shape in which we can apply the Euler-Maclaurin formula \eqref{EulerMcLaurin}. For this, let $\boldsymbol{n}\mapsto \boldsymbol{\ell} + \boldsymbol{n}\frac{kp}{\delta}$ with $\boldsymbol{n}\in\N_0^2$, $0\leq \boldsymbol{\ell} \leq \frac{kp}{\delta}-1$, where $\delta:=\gcd(h,p)$. Here by the inequality we mean that it should hold componentwise. It is not hard to see that, with $\mathcal F_1(\boldsymbol{x}):= e^{-Q(\boldsymbol{x})}$,
\begin{equation*}
F_{1,1}\left(e^{2\pi i\frac{h}{k}-t}\right)=\sum_{\boldsymbol{\alpha} \in \mathscr{S}} \varepsilon(\boldsymbol{\alpha}) \sum_{0 \leq \boldsymbol{\ell} \leq \frac{kp}{\delta}-1}e^{2\pi i \frac{h}{k} Q \left(\boldsymbol{\ell}+\boldsymbol{\alpha} \right)}\sum_{\boldsymbol{n}\in\frac{\delta}{kp}(\boldsymbol{\ell}+\boldsymbol{\alpha})+\N_0^2}\mathcal F_1\left(\frac{kp}{\delta}\sqrt{t}\boldsymbol{n}\right).
\end{equation*}

The main term in \eqref{EulerMcLaurin} is then
\begin{equation}\label{main}
\frac{\delta^2}{k^2p^2 t}\mathcal I_{\mathcal F_1} \sum_{\boldsymbol{\alpha} \in \mathscr{S}}\varepsilon(\boldsymbol{\alpha}) \sum_{0 \leq \boldsymbol{\ell} \leq \frac{kp}{\delta}-1}e^{2\pi i \frac{h}{k}Q\left(\boldsymbol{\ell}+\boldsymbol{\alpha}\right)}.
\end{equation}
It is not hard to see that one may let $\boldsymbol{\ell}$ run modulo $\frac{kp}{\delta}$ (again meant componentwise). We write $\boldsymbol{\ell}=\boldsymbol{N}+k\boldsymbol{\nu}$ with $\boldsymbol{N}$ running modulo $k$, $\boldsymbol{\nu}$ modulo $\frac{p}{\delta}$, and $\boldsymbol{a}\in\{\left(-1,2\right),\left(1,-2\right),\left(0,1\right),\left(0,-1\right),\left(1,-1\right),$ $\left(-1,1\right)\}$ such that $\boldsymbol{\alpha}-\frac{\boldsymbol{a}}{p}\in\Z^2$. We then compute that the sum over $\boldsymbol{\ell}$ in \eqref{main} equals (since $Q(\boldsymbol{a})=1$)
\begin{align*}
e^{2\pi i\frac{h}{p^2k}}\sum_{\boldsymbol{N}\pmod{k}}e^{\frac{2\pi ih}{pk}
	\left(3\left(pN_1^2+2a_1N_1\right)+3\left(pN_1N_2+a_2N_1+a_1N_2\right)+pN_2^2+2a_2N_2\right)}\\
\qquad\times\sum_{\boldsymbol{\nu} \pmod{\frac{p}{\delta}}}e^{\frac{2\pi ih/\delta}{p/\delta}
	\left((6a_1+3a_2)\nu_1+(2a_2+3a_1)\nu_2\right)}.
\end{align*}
Since $\text{gcd}(\frac{h}{\delta},\frac{p}{\delta})=1$, the inner sum vanishes unless $\frac{p}{\delta} |\, 3(2a_1+a_2)$ and $\frac{p}{\delta} |\ (2a_2+3a_1).$ If $3|\frac{p}{\delta}$, then in particular $3\mid a_2$. This is however not satisfied for elements in $\mathscr{S}$.
If $3\nmid\frac{p}{\delta}$, then we easily obtain that $a_1\equiv a_2\equiv 0\pmod{\frac{p}{\delta}}$, implying that $\frac{p}{\delta} =1$.
We are thus left to show that $(\frac{p}{\delta}=1)$
\begin{equation}\label{sums}
\sum_{\boldsymbol{\alpha}\in\mathscr{S}} \varepsilon\left(\boldsymbol{\alpha}\right) \sum_{\boldsymbol{N}\pmod{k}}e^{\frac{2\pi ih/\delta}{k}
	\left(3\left(pN_1^2+2a_1N_1\right)+3\left(pN_1N_2+a_2N_1+a_1N_2\right)+pN_2^2+2a_2N_2\right)}=0.
\end{equation}
Changing $\boldsymbol{N}\mapsto \boldsymbol{N}-\boldsymbol{a}\overline{p}$, with $\overline{p}$ the inverse of $p$ modulo $k$ (note that $\frac{p}{\delta}=1$ implies that $\gcd(p,k)=1$), the sum on $\boldsymbol{N}$ equals
$$
e^{-\frac{2\pi i \overline{p} h\slash\delta }{k}}\sum_{\boldsymbol{N}\pmod{k}}e^{\frac{2\pi ih}{k}
	Q(\boldsymbol{N})},
$$
which is independent of $\boldsymbol{a}$.
Thus \eqref{sums} holds.
%\begin{equation}\label{mainvanish}
%\sum_{({\alpha_1}, {\alpha_2})\in\mathscr{S}}\varepsilon({\alpha_1}, {\alpha_2})\sum_{0\le\ell_1, \ell_2\leq\frac{k p}{\delta}-1} e^{2\pi i \frac{h}{k}Q\left(\ell_1+{\alpha_1}, \ell_2+{\alpha_2}\right)}=0.
%\end{equation}

The second term in \eqref{EulerMcLaurin} is
\begin{equation}\label{second2}
-\sum_{\boldsymbol{\alpha}\in\mathscr{S}}\varepsilon (\boldsymbol{\alpha})\sum_{0\leq\boldsymbol{\ell}\leq\frac{kp}{\delta}-1}e^{2\pi i\frac{h}{k}Q\left(\boldsymbol{\ell}+\boldsymbol{\alpha}\right)}\sum_{n_2\geq 0}\frac{B_{n_2+1}\left(\frac{\delta\left(\ell_2+{\alpha_2}\right)}{kp}\right)}{(n_2+1)!}\int_0^\infty \mathcal F_1^{(0, n_2)}(x_1, 0)dx_1\left(\frac{kp\sqrt{t}}{\delta}\right)^{n_2-1}.
\end{equation}
We claim that the contribution from those $n_2$ which are even vanishes. This follows, once we show that, for $\boldsymbol{\alpha}\in\mathscr{S}$,
\begin{equation*}
\sum_{0\leq\boldsymbol{\ell}\leq\frac{kp}{\delta}-1}\bigg(e^{2\pi i\frac{h}{k}Q\left(\boldsymbol{\ell}+\boldsymbol{\alpha}\right)}
B_{2n_2+1}\left(\frac{\delta\left(\ell_2+{\alpha_2}\right)}{kp}\right)\\+ e^{2\pi i\frac{h}{k}Q\left(\boldsymbol{\ell+1-\alpha}\right)}
B_{2n_2+1}\left(\frac{\delta\left(\ell_2+1-{\alpha_2}\right)}{kp}\right)\bigg)=0.
\end{equation*}
This is seen to be true by the change of variables $\boldsymbol{\ell} \mapsto -\boldsymbol{\ell} +(-1+\frac{kp}{\delta})\boldsymbol{1}$ for the second term.

Arguing in the same way for the contribution from $n_2$ odd, we obtain that \eqref{second2} equals
\begin{equation*}
-2\sum_{\boldsymbol{\alpha}\in\mathscr{S}^\ast}\varepsilon (\boldsymbol{\alpha})\sum_{0\leq\boldsymbol{\ell}\leq\frac{kp}{\delta}-1}e^{2\pi i\frac{h}{k}Q\left(\boldsymbol{\ell+\alpha}\right)}\sum_{n_2\ge 0}\frac{B_{2n_2+2}\left(\frac{\delta\left(\ell_2+{\alpha_2}\right)}{kp}\right)}{(2n_2+2)!}\int_0^\infty \mathcal F_1^{(0, 2n_2+1)}(x_1, 0)dx_1\left(\frac{k^2p^2}{\delta^2}t\right)^{n_2},
\end{equation*}
where
\[ \mathscr{S}^\ast := \left\{ \left(1-\frac1p,\frac2p\right),\left(0,1-\frac1p\right),\left(\frac1p,1-\frac1p\right)\right\}. \]
The third term in \eqref{EulerMcLaurin} is treated in the same way, yielding the contribution
\begin{equation*}
-2 \sum_{\boldsymbol{\alpha}\in\mathscr{S}^\ast}\varepsilon (\boldsymbol{\alpha})\sum_{0\leq\boldsymbol{\ell} \leq\frac{kp}{\delta}-1}e^{2\pi i\frac{h}{k}Q\left(\boldsymbol{\ell+\alpha}\right)}\sum_{n_1\geq 0}\frac{B_{2n_1+2}\left(\frac{\delta\left(\ell_1+{\alpha_1}\right)}{kp}\right)}{(2n_1+2)!}\int_0^\infty \mathcal F_1^{(2n_1+1, 0)}(0, x_2)dx_2\left(\frac{k^2p^2}{\delta^2}t\right)^{n_1}.
\end{equation*}

The final term in \eqref{EulerMcLaurin} equals
\begin{multline*}
\sum_{\boldsymbol{\alpha}\in\mathscr{S}}\varepsilon (\boldsymbol{\alpha})\sum_{0\leq\boldsymbol{\ell}\leq\frac{kp}{\delta}-1}e^{2\pi i\frac{h}{k}Q\left(\boldsymbol{\ell+\alpha}\right)}\\
\times\sum_{n_1, n_2\geq 0} \frac{B_{n_1+1}\left(\frac{\delta\left(\ell_1+{\alpha_1}\right)}{kp}\right)}{(n_1+1)!} \frac{B_{n_2+1}\left(\frac{\delta\left(\ell_2+{\alpha_2}\right)}{kp}\right)}{(n_2+1)!}
\mathcal F_1^{(n_1, n_2)}(0, 0)\left(\frac{kp\sqrt{t}}{\delta}\right)^{n_1+n_2}.
\end{multline*}
Arguing in the same way as before this equals
\begin{align*}
& 2\sum_{\boldsymbol{\alpha}\in\mathscr{S}^\ast}\varepsilon (\boldsymbol{\alpha})\sum_{0\leq\ell\leq\frac{kp}{\delta}-1}e^{2\pi i\frac{h}{k}Q\left(\boldsymbol{\ell+{\alpha}}\right)}\\
%\label{tmp1}
&\times\sum_{n_1, n_2\geq 0\atop{n_1\equiv n_2\pmod{2}}} \frac{B_{n_1+1}\left(\frac{\delta\left(\ell_1+{\alpha_1}\right)}{kp}\right)}{(n_1+1)!} \frac{B_{n_2+1}\left(\frac{\delta\left(\ell_2+{\alpha_2}\right)}{kp}\right)}{(n_2+1)!}
\mathcal F_1^{(n_1, n_2)}(0, 0)\left(\frac{kp\sqrt{t}}{\delta}\right)^{n_1+n_2}.
\end{align*}

The function $F_{1,2}$ is treated similarly, yielding, with $\mathcal F_2(x):= e^{-x^2}$,
\begin{align*}
& -\sum_{0\leq r \leq \frac{kp}{\delta}-1} e^{2\pi i \frac{h}{k}\left(r+\frac1p\right)^2}\sum_{m\geq 0}\frac{ B_{2m+1}\left(\frac{\delta\left(r+\frac1p\right)}{kp}\right)}{(2m+1)!} \mathcal F_2^{(2m)}(0) \left(\frac{k^2p^2}{\delta^2}t\right)^m.
\end{align*}

\subsection{The function $F_2$}\label{functionf2}

Since the calculations are similar to those for $F_1$, we skip some of the details.
Decompose
$$
F_2(q) = F_{2,1}(q) + F_{2,2}(q),
$$
with
\begin{align*}
F_{2,1}(q)&:= \sum_{\boldsymbol{\alpha}\in\mathscr S} \eta(\boldsymbol{\alpha}) \sum_{\boldsymbol{n}\in\boldsymbol{\alpha}+\N_0^2}n_2q^{Q(\boldsymbol{n})},\qquad F_{2,2}(q):= -\frac12 \sum_{m\in\frac1p+\Z} |m|q^{m^2}.
\end{align*}

We first study the asymptotic behavior of $F_{2,1}$.
Arguing as for $F_{1,1}$, we have 
\[
F_2\left(e^{2\pi i\frac{h}{k}-t}\right)=\frac{1}{\sqrt{t}}\sum_{\boldsymbol{\alpha}\in\mathscr{S}}\!\!\eta(\boldsymbol{\alpha})\!\!\!\!\sum_{0\leq\boldsymbol{\ell}\leq \frac{kp}{\delta}-1} \!\!\!\!\!e^{2\pi i\frac{h}{k}Q\left(\boldsymbol{\ell+{\alpha}}\right)}
\!\!\sum_{n\in\frac{\delta}{kp}(\boldsymbol{\ell+\alpha})+\N_0^2}\mathcal G_1\left(\frac{kp}{\delta}\sqrt{t}\boldsymbol{n}\right),
\]
with $\mathcal G_1(\boldsymbol{x}):= x_2 \mathcal F_1(\boldsymbol{x})$. 
The Euler-Maclaurin main term is 
\[
\frac{1}{t^\frac32}\left(\frac{\delta}{kp}\right)^{2}\mathcal I_{\mathcal G_1}\sum_{\boldsymbol{\alpha}\in\mathscr{S}}\eta(\boldsymbol{\alpha})\sum_{\boldsymbol{\ell}\pmod{\frac{kp}{\delta}}}e^{2\pi i\frac{h}{k}Q(\boldsymbol{\ell+\alpha})}.
\]
As in Subsection \ref{functionf1}, one can show that this vanishes.

The second term in the Euler-Maclaurin summation formula is
\begin{equation}\label{second}
-2\sum_{\boldsymbol{\alpha}\in\mathscr{S}^\ast}\sum_{0\leq\boldsymbol{\ell}\leq \frac{kp}{\delta}-1}e^{2\pi i\frac{h}{k}Q\left(\boldsymbol{\ell+{\alpha}}\right)} \sum_{n_2\geq 1}\frac{B_{2n_2+1}\left(\frac{\delta(\ell_2+{\alpha_2})}{kp}\right)}{(2n_2+1)!}\int_0^\infty \mathcal G_1^{(0, 2n_2)}(x_1, 0)dx_1\left(\frac{kp}{\delta}\right)^{2n_2-1} t^{n_2-1},
\end{equation}
again pairing $\boldsymbol{\alpha}$ and $\boldsymbol{1-\alpha}$ and using that $\mathcal G_1(x_1,0)=0$.

In the same way we obtain that the third term in the Euler-Maclaurin summation formula is
\begin{equation}\label{thirdterm}
-2\sum_{\boldsymbol{\alpha}\in\mathscr S^*} \sum_{0\leq\boldsymbol{\ell}\leq \frac{kp}{\delta}-1}e^{2\pi i\frac{h}{k}Q(\boldsymbol{\ell+\alpha})}\sum_{n_1\geq 0}\frac{B_{2n_1+1}\left(\frac{\delta(\ell_1+\alpha_1)}{kp}\right)}{(2n_1+1)!}\int_0^\infty \mathcal G_1^{(2n_1,0)}(0,x_2)dx_2\left(\frac{kp}{\delta}\right)^{2n_1-1}t^{n_1-1}.
\end{equation}
The final term in Euler-Maclaurin evaluates as
\begin{multline*}
2\sum_{\boldsymbol{\alpha}\in\mathscr{S}^*} \sum_{0\leq \boldsymbol{\ell}\leq  \frac{kp}{\delta}-1} e^{2\pi i\frac{h}{k}Q\left(\boldsymbol{\ell+\alpha}\right)}\\
\times\sum_{n_1, n_2\ge 0\atop{n_1\not\equiv n_2\pmod{2}}}\!\!\!\frac{B_{n_1+1}\left(\frac{\delta(\ell_1+{\alpha_1})}{kp}\right)}{(n_1+1)!}\frac{B_{n_2+1}\left(\frac{\delta(\ell_2+{\alpha_2})}{kp}\right)}{(n_2+1)!}\mathcal G_1^{(n_1, n_2)}(0,0)\left(\frac{kp}{\delta}\right)^{n_1+n_2} t^{\frac{n_1+n_2-1}{2}},
\end{multline*}
again pairing $\boldsymbol{\alpha}$ with $\boldsymbol{1-\alpha}$.

We next determine those terms of $F_{2,1}$ that grow as $t\to 0^+$. Inspecting the terms above we see that this comes from the $n_1=0$ term of \eqref{thirdterm} and is given by
\begin{equation}\label{grow1}
-\frac{2\delta}{kpt}\sum_{\boldsymbol{\alpha}\in \mathscr S^\ast} \sum_{0\leq \boldsymbol{\ell}\leq \frac{kp}{\delta}-1} B_1\left(\frac{\delta(\ell_1+\alpha_1)}{kp}\right) e^{2\pi i\frac{h}{k} Q(\boldsymbol{\ell+\alpha})} \int_0^\infty \mathcal G_1(0,x_2)dx_2.
\end{equation}
Using that $\mathcal G_1 (0,x_2)=x_2 e^{-x_2^2}=:\mathcal G_2(x_2)$, we obtain that \eqref{grow1} equals
$$
-\frac{2\delta}{kpt} \mathcal{I}_{\mathcal{G}_2}\sum_{\boldsymbol{\alpha}\in\mathscr S^\ast}\sum_{0\leq\boldsymbol{\ell}\leq \frac{kp}{\delta}-1} B_1\left(\frac{\delta(\ell_1+\alpha_1)}{kp}\right)e^{2\pi i\frac{h}{k}Q(\boldsymbol{\ell+\alpha})}.
$$

Turning to $F_{2,2}$, its Euler-Maclaurin main term is
\begin{equation}\label{F2main}
-\frac{\delta}{kpt} \mathcal I_{\mathcal G_2} \sum_{r\pmod{\frac{kp}{\delta}}} e^{2\pi i \frac{h}{k}\left(r+\frac1p\right)^2}.
\end{equation}

Arguing as before, the second term in the Euler-Maclaurin summation formula equals
$$
\sum_{0\leq r \leq \frac{kp}{\delta}-1}e^{2\pi i\frac{h}{k}\left(r+\frac1p\right)^2} \sum_{m\geq 0} \frac{B_{2m+2}\left(\frac{\delta(\ell+\frac1p)}{kp}\right)}{(2m+2)!}\mathcal G_2^{(2m+1)}(0)\left(\frac{kp}{\delta}\right)^{2m+1} t^{m}.
$$
To see that all terms that grow as $t\to 0^+$ cancel, we need to prove that
\begin{equation}\label{sumsmatch}
\sum_{r\pmod{\frac{kp}{\delta}}}e^{2\pi i\frac{h}{k}\left(r+\frac{1}{p}\right)^2}=-2\sum_{\boldsymbol{\alpha}\in\mathscr S^\ast}\sum_{0\leq \boldsymbol{\ell}\leq \frac{kp}{\delta}-1}B_1\left(\frac{\delta(\ell_1+\alpha_1)}{kp}\right)e^{2\pi i\frac{h}{k}Q(\boldsymbol{\ell+\alpha})}.
\end{equation}
To show \eqref{sumsmatch}, we first assume that $\frac{p}{\delta}\not\in\{1,2\}$. Writing $\boldsymbol{\ell}=\boldsymbol{N}+k\boldsymbol{\nu}$, $0\leq \boldsymbol{N}< k$, $0\leq \boldsymbol{\nu}< \frac{p}{\delta}$ and $\boldsymbol{a}=p\boldsymbol{\alpha}$, we obtain that the sum on $\boldsymbol{\ell}$ equals
\begin{align}\label{splitsum}
e^{2\pi i\frac{h}{p^2k}Q(\boldsymbol{a})}&\sum_{0\leq \boldsymbol{N}< k}e^{\frac{2\pi ih}{pk}\left(3\left(pN_1^2+2a_1N_1\right)+3\left(pN_1N_2+a_2N_1+a_1N_2\right)+pN_2^2+2a_2N_2\right)}\notag\\&\qquad\times\sum_{0\leq \boldsymbol{\nu}< \frac{p}{\delta}} B_1\left(\frac{\delta\left(N_1+k\nu_1+\frac{a_1}{p}\right)}{kp}\right)e^{2\pi i \frac{h/\delta}{p/\delta} \left(\left(6a_1+3a_2\right)\nu_1+\left(2a_2+3a_1\right)\nu_2\right)}.
\end{align}
The sum on $\nu_2$ vanishes unless $\frac{p}{\delta}|(2a_2+3a_1)$. It is not hard to see that (under the assumption that $\frac{p}{\delta}\not\in\{1,2\}$) this is not satisfied for elements in $p\mathscr S^\ast$.

We next assume that $\frac{p}{\delta}=1$. It is not hard to see that
\begin{equation}\label{claimquad}
e^{2\pi i\frac{h}{k}Q\left(k-\ell_1-1+1-\frac1p,\ell_2+3\ell_1+1+\frac2p\right)} = e^{2\pi i\frac{h}{k}Q\left(\ell_1+\frac1p,\ell_2+1-\frac1p\right)}.
\end{equation}
This then implies that the contribution of the first and third element in $\mathscr S^*$ cancel due to a negative sign from the Bernoulli polynomial and we can shift the sum in $\ell_2$ by integers. Thus the right-hand side of \eqref{sumsmatch} becomes
\begin{equation}\label{onlysecondterm}
-2\sum_{0\leq\boldsymbol{\ell}<k}B_1\left(\frac{\ell_1}{k}\right)e^{2\pi i\frac{h}{k}Q\left(\ell_1,\ell_2+1-\frac1p\right)}.
\end{equation}
Now one can show that
\begin{equation}\label{secondtermcomp}
e^{2\pi i\frac{h}{k}Q\left(k-\ell_1,\ell_2+3\ell_1+1-\frac1p\right)}= e^{2\pi i\frac{h}{k}Q\left(\ell_1,\ell_2+1-\frac1p\right)}.
\end{equation}
To finish the claim \eqref{sumsmatch}, we assume, without loss of generality, that $k$ is odd. We split the sum in \eqref{onlysecondterm}, substitute $(\ell_1,\ell_2)\mapsto (k-\ell_1,\ell_2+3\ell_1)$ in the second part and use \eqref{secondtermcomp} to obtain
\begin{align*}
-2\sum_{0\leq\boldsymbol{\ell}<k}&B_1\left(\frac{\ell_1}{k}\right)
e^{2\pi i\frac{h}{k}Q\left(\ell_1,\ell_2+1-\frac1p\right)}
=-2\left(\sum_{\substack{0\leq \ell_1\leq \frac12(k-1)\\ \ell_2\pmod k}} + \sum_{\substack{\frac12(k+1)\leq \ell_1<k\\ \ell_2 \pmod k}}\right)
B_1\left(\frac{\ell_1}{k}\right)e^{2\pi i\frac{h}{k}Q\left(\ell_1,\ell_2+1-\frac1p\right)} \\
&=-2\sum_{\substack{0\leq \ell_1\leq \frac12(k-1)\\ \ell_2\pmod k}} B_1\left(\frac{\ell_1}{k}\right)
e^{2\pi i\frac{h}{k}Q\left(\ell_1,\ell_2+1-\frac1p\right)}
-2\sum_{\substack{0< \ell_1\leq\frac12(k-1)\\ \ell_2\pmod k}}
B_1\left(1-\frac{\ell_1}{k}\right)
e^{2\pi i\frac{h}{k}Q\left(\ell_1,\ell_2+1-\frac1p\right)}\\
&=-2B_1(0) \sum_{\ell_2\pmod k}e^{2\pi i\frac{h}{k}Q\left(0,\ell_2+1-\frac1p\right)}=\sum_{\ell_2\pmod k}e^{2\pi i\frac{h}{k}\left(\ell_2+1-\frac1p\right)^2}.
\end{align*}

The case $\frac{p}{\delta}=2$ is done similarly.

\section{Companions in the lower half plane}

In this section we investigate multivariable Eichler integrals.
\subsection{Multiple Eichler integrals}
Let $f_j\in S_{k_j}(\Gamma,\chi_j)$; if $k_j = \frac12$ we also allow $f_j\in M_{\frac12}(\Gamma,\chi_j)$. Define the {\it double Eichler integral}
\begin{align*}
\quad I_{f_1,f_2}(\tau) := \int_{-\overline{\tau}}^{i\infty} \int_{w_1}^{i\infty} \frac{f_1(w_1)f_2(w_2)}{(-i(w_1+\tau))^{2-k_1}(-i(w_2+\tau))^{2-k_2}} dw_2 dw_1,
\end{align*}
and the {\it multiple error of modularity}
\begin{align*}
r_{f_1,f_2,\frac{d}{c}}(\tau) := \int_{\frac{d}{c}}^{i\infty} \int_{w_1}^{\frac{d}{c}} \frac{f_1(w_1) f_2(w_2)}{(-i(w_1+\tau))^{2-k_1} (-i(w_2+\tau))^{2-k_2}} dw_2 dw_1.
%\label{double-eichler}
\end{align*}

\begin{theorem}\label{quantheorem}
	We have, for $M=\left(\begin{smallmatrix}a&b\\c&d\end{smallmatrix}\right)\in\Gamma^\ast$,
	\begin{equation}\label{modtrans}
	I_{f_1,f_2}(\tau)-\chi_1^{-1}\left(M^\ast\right)\chi_2^{-1}\left(M^\ast\right)(c\tau+d)^{k_1+k_2-4}I_{f_1,f_2}(M\tau) = r_{f_1,f_2,\frac{d}{c}}(\tau) + I_{f_1}(\tau)r_{f_2,\frac{d}{c}}(\tau).
	\end{equation}
	Moreover $r_{f_1,f_2,\frac{d}{c}}\in\mathcal O(\R\backslash\{-\frac{d}{c}\})$.
	If $f_j\in S_{k_j}(\Gamma,\chi_j)$ (for $j=1,2$), then $r_{f_1,f_2, \frac{d}{c}}\in \mathcal O(\R)$.
\end{theorem}
	
\begin{proof}[Proof of Theorem \ref{quantheorem}] For simplicity, we assume that $\frac12\leq k_j\leq2$ and that $f_1,f_2$ are cuspidal. The proof in the case that $f_1$ or $f_2$ are not cuspidal and of weight $\frac12$ is basically the same;
we then require the bound
\[
f_j\left(iw_j+\frac{d}{c}\right)\ll 1+w_j^{-\frac12}.
\] 
A direct calculation gives that, for  $M\in \Gamma^\ast$,
%\[
%I_{f_j,g_\ell}\left(-\frac1\tau\right) = -i\tau \sum_{\substack{1\le k\le N\\ 1\le m\le M}} \chi_{j,k} \: \psi_{\ell,m} \int_{-\overline{\tau}}^{0} \int_{w}^{0} \frac{f_k(w)g_m(s)}{\sqrt{-i(w+\tau)}\sqrt{-i(s+\tau)}}dsdw.
%\]
%Thus
%\[
%-\frac{i}{\tau} I_{f_j,g_\ell}\left(-\frac1\tau\right) +\sum_{\substack{1\le k\le N\\1\le m\le M}} \chi_{j,k} \: \psi_{\ell,m} I_{f_k,g_m}(\tau) = \sum_{\substack{1\le k\le N\\1\le m\le M}}\chi_{j,k} \: \psi_{\ell,m} \int_{0}^{i\infty} \int_w^{i\infty} \frac{f_k(w)g_m(s)}{\sqrt{-i(w+\tau)}\sqrt{-i(s+\tau)}} dsdw \]
%\[+ \sum_{\substack{1\le k\le N\\1\le m\le M}} \chi_{j,k} \: \psi_{\ell,m} \int_0^{-\overline{\tau}} \frac{f_k(w)}{\sqrt{-i(w+\tau)}}dw\int_0^{\infty} \frac{g_m(s)}{\sqrt{-i(s+\tau)}}ds.
%\]

\begin{align*}
I_{f_1,f_2}(M\tau) &= \chi_1\left(M^*\right)\chi_2\left(M^*\right)(c\tau+d)^{4 -k_1 -k_2} \int_{-\overline{\tau}}^{\frac{d}{c}}\int_{w_1}^{\frac{d}{c}} \frac{f_1(w_1)f_2(w_2)}{(-i(w_1+\tau))^{2-k_1}(-i(w_2+\tau))^{2-k_2}}dw_2dw_1.
%&\quad -\chi_1(M^\ast)\chi_2(M^\ast)(c\tau+d) \int_{-\overline{\tau}}^{\frac{d}{c}}\int_w^{\frac{d}{c}} \frac{f_1(w)f_2(s)}{\sqrt{-i(w+\tau)}\sqrt{-i(s+\tau)}}dsdw,
\end{align*}
The transformation \eqref{modtrans} now follows by splitting
\begin{align*}
\int_{-\overline{\tau}}^{\frac{d}{c}}\int_{w_1}^{\frac{d}{c}}
=\int_{-\overline{\tau}}^{i\infty}\int_{w_1}^{i\infty}
+\int_{\frac{d}{c}}^{i\infty}\int_{\frac{d}{c}}^{w_1}
-\int_{-\overline{\tau}}^{i\infty}\int_{\frac{d}{c}}^{i\infty}.
\end{align*}
%that $r_{f_1,f_2,M}(\tau)$ equals
%\begin{equation}\label{modtrans}
%-\int_{\frac{d}{c}}^{i\infty} \int_{\frac{d}{c}}^{w}\frac{f_1(w)f_2(s)}{(-i(w+\tau))^{2-k_1}(-i(s+\tau))^{2-k_2}}dsdw+I_{f_1}(\tau)\Gamma_{f_2,\frac{d}{c}}(\tau).
%\end{equation}

Using Lemma \ref{lquant}, we are left to show that $r_{f_1,f_2,\frac{d}{c}}$ is real-analytic on $\R$ which follows once we prove that the following function is real-analytic
\begin{equation}\label{errorint}
\int_0^\infty \int_0^{w_1} \frac{f_1\left(iw_1+\frac{d}{c}\right)f_2\left(iw_2+\frac{d}{c}\right)}{\left(w_1-i\left(\tau+\frac{d}{c}\right)\right)^{2-k_1}\left(w_2-i\left(\tau+\frac{d}{c}\right)\right)^{2-k_2}}dw_2dw_1.
\end{equation}
We use that for $w_j\geq 1$
\begin{equation}\label{boundbig}
f_j\left(iw_j+\frac{d}{c}\right)\ll e^{-a_jw_j}\qquad a_j\in\R^+,
\end{equation}
and for $0< w_j\leq 1$ (the implied constant and $b_j$ may depend on $c$)
\begin{equation}\label{boundsmall}
f_j\left(iw_j+\frac{d}{c}\right)\ll w_j^{-k_j} e^{-\frac{b_j}{w_j}}\qquad b_j\in\R^+.
\end{equation}

To show real-analycity of \eqref{errorint} on $\R$, we split it into 3 pieces. Firstly, set
$$
I_1:=\int_1^\infty \int_1^{w_1} \frac{f_1\left(iw_1+\frac{d}{c}\right)f_2\left(iw_2+\frac{d}{c}\right)}{\left(w_1-i\left(\tau+\frac{d}{c}\right)\right)^{2-k_1}\left(w_2-i\left(\tau+\frac{d}{c}\right)\right)^{2-k_2}}dw_2dw_1.
$$
Using \eqref{boundbig} and that $w_1\geq 1$ easily gives the locally uniform bound
$$
I_1\ll \int_1^\infty \frac{e^{-a_1w_1}}{w_1^{2-k_1}} dw_1 \int_1^\infty \frac{e^{-a_2w_2}}{w_2^{2-k_2}}dw_2 \ll 1.
$$

Next consider
$$
I_2:=\int_0^1\int_0^{w_1}\frac{f_1\left(iw_1+\frac{d}{c}\right)f_2\left(iw_2+\frac{d}{c}\right)}{\left(w_1-i\left(\tau+\frac{d}{c}\right)\right)^{2-k_1}\left(w_2-i\left(\tau+\frac{d}{c}\right)\right)^{2-k_2}}dw_2dw_1.
$$
Using \eqref{boundsmall} gives that 
$$
I_2\ll \int_0^1\frac{e^{-\frac{b_1}{w_1}}}{w_1^{2}}dw_1\int_0^1\frac{e^{-\frac{b_2}{w_2}}}{w_2^{2}}dw_2 \ll 1.
$$

Finally, we set
$$
I_3:=\int_1^\infty \frac{f_1\left(iw_1+\frac{d}{c}\right)}{\left(w_1-i\left(\tau+\frac{d}{c}\right)\right)^{2-k_1}}dw_1\int_0^1 \frac{f_2\left(iw_2+\frac{d}{c}\right)}{\left(w_2-i\left(\tau+\frac{d}{c}\right)\right)^{2-k_2}}dw_2.
$$
Combining the above bounds gives again $I_3\ll 1$.

\end{proof}

\subsection{Special multiple Eichler integrals of weight one}

Define for $\boldsymbol{\alpha} \in \mathscr{S}^\ast$
\begin{align*}
\mathcal E_{1,\boldsymbol{\alpha}}(\tau):=-\frac{\sqrt{3}}{4}
\int_{-\overline{\tau}}^{i\infty}
\int_{w_1}^{i\infty}
\frac{\theta_1(\boldsymbol{\alpha};\boldsymbol{w})+\theta_2(\boldsymbol{\alpha};\boldsymbol{w})}{\sqrt{-i(w_1+\tau)}\sqrt{-i(w_2+\tau)}}
dw_2 dw_1
\end{align*}
with
\begin{align*}
\theta_1(\boldsymbol{\alpha};\boldsymbol{w})&:=\sum_{\boldsymbol{n}\in\boldsymbol{\alpha}+\Z^2}(2n_1+n_2)n_2
e^{\frac{3\pi i}{2}(2n_1+n_2)^2w_1
	+ \frac{\pi in_2^2w_2}{2}},
\\
\theta_2(\boldsymbol{\alpha};\boldsymbol{w})&:=\sum_{\boldsymbol{n}\in \boldsymbol{\alpha}+\Z^2}(3n_1+2n_2)n_1
e^{\frac{\pi i}{2}(3n_1+2n_2)^2w_1
	+ \frac{3\pi in_1^2w_2}{2}}.
\end{align*}
Moreover set 
\begin{align}\label{defineE1}
\mathcal{E}_1(\tau)&:= \sum_{\boldsymbol{\alpha}\in \mathscr{S}^\ast}\varepsilon(\boldsymbol{\alpha})\mathcal{E}_{1,\boldsymbol{\alpha}}(p\tau),\\\notag
\Gamma_p&:=\left\{\left(\begin{matrix}
a & b \\ c & d
\end{matrix}\right)\in\Gamma_0(12p):b\equiv 0\pmod{4p}, d\equiv \pm 1\pmod{2p}\right\}.
\end{align}
\begin{remark}
	Note that $\Gamma_p^\ast = \Gamma_p.$
\end{remark}
\begin{remark}
One can show that
	\begin{equation*}
	\mathcal E_1(\tau) = -\frac{\sqrt{3}}{4p} \sum_{\delta\in\{0,1\}} I_{\Theta_1(2p,1+p\delta,2p;\,\cdot\,),\Theta_1(6p,3+3p\delta,6p;\,\cdot\,)}(\tau).
	\end{equation*}
	However, as this representation is not required for the remainder of the paper, we do not provide a proof of this identity.
\end{remark}
\begin{proposition}\label{quantE}
	We have, for $M=\left(\begin{smallmatrix}a&b\\c&d\end{smallmatrix}\right)\in\Gamma_p$,
	$$
	\mathcal E_1(\tau) - \left(\frac{-3}{d}\right) (c\tau+d)^{-1}\mathcal E_1(M\tau) = \sum_{j=1}^{12} \left(r_{f_j,g_j,\frac{d}{c}}(\tau)+I_{f_j}(\tau) r_{g_j,\frac{d}{c}}(\tau)\right),
	$$
	where $f_j,g_j$ are cusp forms of weight $\frac32$ (with some multiplier).
\end{proposition}
\begin{proof}
To use Theorem \ref{quantheorem}, we write $\theta_j$ in terms of Shimura's theta functions \eqref{shimura}.
For $\theta_1$, we set $\nu_1:=2n_1+n_2$, $\nu_2:=n_2$. Then $\nu_1\in2{\alpha_1}+{\alpha_2}+\Z$, $\nu_2\in{\alpha_2}+\Z$,  and $\nu_1-\nu_2\in2{\alpha_1}+2\Z$ and we obtain 
\begin{align*}
\theta_1(\boldsymbol{\alpha};\boldsymbol{w})&=\sum_{\substack{\boldsymbol{\nu}\in(2{\alpha_1}+{\alpha_2},{\alpha_2})+\Z^2\\\nu_1-\nu_2\in2{\alpha_1}+2\Z}}\nu_1 \nu_2e^{\frac{3\pi i\nu_1^2w_1}{2}+\frac{\pi i\nu_2^2w_2}{2}}\\
&=\sum_{\varrho\in\{0,1\}}\sum_{\nu_1\in2{\alpha_1}+{\alpha_2}+\varrho+2\Z}\nu_1 e^{\frac{3\pi i\nu_1^2w_1}{2}}\sum_{\nu_2\in{\alpha_2}+\varrho+2\Z} \nu_2 e^{\frac{\pi i\nu_2^2w_2}{2}}.
\end{align*}
Summing then easily gives
\begin{align*}
\sum_{\boldsymbol{\alpha}\in\mathscr S^*} \varepsilon\left(\boldsymbol{\alpha}\right) \theta_1(\boldsymbol{\alpha};\boldsymbol{w})
&=\frac{1}{p^2}\sum_{\boldsymbol{A}\in\mathcal A}\varepsilon_1\left(\boldsymbol{A}\right) \quad\sum_{\nu_1\equiv A_1\pmod{2p}}\nu_1 e^{\frac{3\pi i\nu_1^2w_1}{2p^2}}\sum_{\nu_2\equiv A_2\pmod{2p}} \nu_2 e^{\frac{\pi i\nu_2^2w_2}{2p^2}}
\\&=\frac{1}{p^2}\sum_{\boldsymbol{A}\in\mathcal A}\varepsilon_1\left(\boldsymbol{A}\right)\Theta_1\left(2p,A_1,2p;\frac{3w_1}{p}\right)\Theta_1\left(2p,A_2,2p;\frac{w_2}{p}\right)
\end{align*}
with
\begin{align*}
\mathcal A&:=\!\left\{\left(0,2\right),\left(p,p+2\right),\left(p-1,p-1\right),\left(-1,-1\right),\left(p+1,p-1\right),\left(1,-1\right)\right\}\!, \varepsilon_1(\boldsymbol{A}):= \varepsilon\left(\frac{A_1-A_2}{2p},
\frac{A_2}{p}\right).
\end{align*}

For $\theta_2$, we proceed similarly. Set $\nu_1=3n_1+2n_2$, $\nu_2=n_1$. Then $\nu_1\in3{\alpha_1}+2{\alpha_2}+\Z$, $\nu_2\in{\alpha_1}+\Z$, and $\nu_1-3\nu_2\in2{\alpha_2}+2\Z$ and we obtain 
\begin{align*}
\theta_2(\boldsymbol{\alpha};\boldsymbol{w})&=\!\!\!\sum_{\substack{\boldsymbol{\nu}\in(3{\alpha_1}+2{\alpha_2},{\alpha_1})+\Z^2 \\ \nu_1-3\nu_2\in 2{\alpha_2}+2\Z}}\!\!\! \nu_1 \nu_2 e^{\frac{\pi i\nu_1^2w_1}{2}+\frac{3\pi i\nu_2^2 w_2}{2}}\\
&=\sum_{\varrho\in\{0,1\}}\sum_{\nu_1\in3{\alpha_1}+2{\alpha_2}+\varrho+2\Z}\nu_1 e^{\frac{\pi i\nu_1^2w_1}{2}}\!\!\!\sum_{\nu_2\in{\alpha_1}+\varrho+2\Z}\!\!\!\nu_2 e^{\frac{3\pi i\nu_2^2w_2}{2}}.
\end{align*}
Summing gives
\begin{align*}
\sum_{\boldsymbol{\alpha}\in\mathscr S^*}
\varepsilon\left(\boldsymbol{\alpha}\right)
\theta_2(\boldsymbol{\alpha};\boldsymbol{w})&=\frac{1}{p^2}\sum_{\boldsymbol{B}\in \mathcal{B}}
\varepsilon_2\left(\boldsymbol{B}\right)
\sum_{\nu_1\equiv B_1\pmod{2p}} \nu_1 e^{\frac{\pi i\nu_1^2w_1}{2p^2}}\!\!\!\!\!\!\!\sum_{\nu_2\equiv B_2\pmod{2p}}\!\!\!\!\!\nu_2e^{\frac{3\pi i\nu_2^2w_2}{2p^2}}
\\
&=\frac{1}{p^2}\sum_{\boldsymbol{B}\in \mathcal B}\varepsilon_2\left(\boldsymbol{B}\right)
\Theta_1\left(2p,B_1,2p;\frac{w_1}{p}\right)\Theta_1\left(2p,B_2,2p;\frac{3w_2}{p}\right)
\end{align*}
with
\begin{align*}
\mathcal{B}&:=\!\left\{(p+1,p-1),(1,-1),(p+2,p),(2,0),(1,1),(p+1,p+1)\right\},\
\varepsilon_2(\boldsymbol{B}):=\varepsilon\left( \frac{B_2-3B_1}{2p},\frac{B_1}{p}\right).
\end{align*}
Combining the above yields that
\begin{align*}
\mathcal E_1\left(\tau\right)&=-\frac{\sqrt{3}}{4p}\sum_{\boldsymbol{A}\in \mathcal{A}} \varepsilon_1(\boldsymbol{A})\int_{-\overline{\tau}}^{i\infty}\int_{w_1}^{i\infty}\frac{\Theta_1(2p,A_1,2p;3w_1)\Theta_1\left(2p,A_2,2p;w_2\right)}{{\sqrt{-i(w_1+\tau)}\sqrt{-i(w_2+\tau)}}}dw_2dw_1\\
&\quad-\frac{\sqrt{3}}{4p}\sum_{\boldsymbol{B}\in \mathcal{B}} \varepsilon_2\left(\boldsymbol{B}\right)\int_{-\overline{\tau}}^{i\infty}\int_{w_1}^{i\infty}\frac{\Theta_1\left(2p,B_1,2p;w_1\right)\Theta_1\left(2p,B_2,2p;3w_2\right)}{{\sqrt{-i(w_1+\tau)}\sqrt{-i(w_2+\tau)}}}dw_2dw_1.
\end{align*}
%We use \eqref{Thetainv} to get
%$$
%\Theta_1\left(2p,a,2p;-\frac{1}{\tau}\right)=-i(-i\tau)^\frac32 (2p)^{-\frac12}\sum_{k\pmod{2p}}e\left(\frac{ka}{2p}\right)\Theta_1(2p,k,2p;\tau).
%$$
%For $\Theta_1\left(2p,a,2p;\frac{\tau}{3}\right)$, we use \eqref{split} and then again \eqref{Thetainv} to obtain that
%$$
%\Theta_1\left(2p,a,2p;-\frac{1}{3\tau}\right)=-i(-i\tau)^\frac32(6p)^{-\frac12}\sum_{k\pmod{6p}}e\left(\frac{ka}{6p}\right)\Theta_1(6p,k,6p;\tau).
%$$

For $M\in\Gamma_p$, we have, using \eqref{shimura} and \eqref{Shimura1}, 
$$
\Theta_1\left(2p,A,2p;\ell M\tau\right)=\pm\left(\frac{\ell pc}{d}\right)\varepsilon^{-1}_d(c\tau+d)^\frac32\Theta_1(2p,A,2p;\ell\tau).
$$
Theorem \ref{quantheorem} then finishes the claim using that $\varepsilon_d^2 = (\frac{-1}{d})$.

\end{proof}

\subsection{Special multiple Eichler integrals of weight two}

Define for $\boldsymbol{\alpha}\in\mathscr S^\ast$
\begin{align*}
\mathcal E_{2,\boldsymbol{\alpha}}(\tau):=&\frac{\sqrt{3}}{8\pi} \int_{-\overline{\tau}}^{i\infty} \int_{w_1}^{i\infty} \frac{2\theta_3(\boldsymbol{\alpha}; \boldsymbol{w}) -\theta_4(\boldsymbol{\alpha}; \boldsymbol{w})}{\sqrt{-i(w_1+\tau)}(-i(w_2+\tau))^{\frac32}}dw_2dw_1 \\
&+\frac{\sqrt{3}}{8\pi} \int_{-\overline{\tau}}^{i\infty} \int_{w_1}^{i\infty} \frac{\theta_5(\boldsymbol{\alpha};\boldsymbol{w})}{(-i(w_1+\tau))^{\frac32}\sqrt{-i(w_2+\tau)}} dw_2 dw_1
\end{align*}
with
\begin{align*}
\theta_3(\boldsymbol{\alpha}; \boldsymbol{w})&:= \sum_{n\in\alpha+\Z^2}(2n_1+n_2) e^{\frac{3\pi i}{2}(2n_1+n_2)^2w_1+\frac{\pi i n_2^2 w_2}{2}},\\
\theta_4(\boldsymbol{\alpha}; \boldsymbol{w})&:= \sum_{n\in\alpha+\Z^2}(3n_1+2n_2) e^{\frac{\pi i}{2}(3n_1+2n_2)^2w_1+\frac{3\pi i n_1^2 w_2}{2}},\\
\theta_5(\boldsymbol{\alpha}; \boldsymbol{w})&:= \sum_{n\in\alpha+\Z^2}n_1 e^{\frac{\pi i}{2}(3n_1+2n_2)^2w_1+\frac{3\pi i n_1^2 w_2}{2}}.
\end{align*}
We then set
$$
\mathcal E_2(\tau) := \sum_{\boldsymbol{\alpha}\in\mathscr S^\ast}\mathcal E_{2,\boldsymbol{\alpha}}(p\tau).
$$
\begin{remark}
	Similarly as for $\mathcal E_1$, one can simplify $\mathcal E_2$ as
	\begin{equation*}
	\mathcal E_2(\tau) = -\frac{\sqrt{3}}{8\pi} \sum_{\boldsymbol{B}\in\mathcal B} I_{\Theta_1(2p,B_1,2p;\,\cdot\,),\Theta_0(6p,3B_2,6p;\,\cdot\,)}(\tau).	
	\end{equation*}
\end{remark}
This function again transforms as a depth two quantum modular.
\begin{proposition}\label{Derivativetrans}
	We have, for $M\in\Gamma_p$,
	$$
	\mathcal E_2(\tau) - \left(\frac{3}{d}\right) (c\tau+d)^{-2} \mathcal E_2(M\tau) = \sum_{j=1}^{18} \left(r_{f_j,g_j,\frac{d}{c}}(\tau)+I_{f_j}(\tau)r_{g_j,\frac{d}{c}}(\tau)\right),
	$$
	 where $f_j$ and $g_j$ are holomorphic modular forms of weight $\frac12$ or cusp forms of weight $\frac32$.
\end{proposition}
\begin{proof}
	As in the proof of Proposition \ref{quantE}, we obtain
	\begin{align*}
	\sum_{\substack{\boldsymbol{\alpha}\in\mathscr S^\ast \\ \boldsymbol{n}\in\boldsymbol{\alpha}+\Z^2}}(2n_1+n_2)e^{\frac{3\pi i}{2}(2n_1+n_2)^2w_1+\frac{\pi in_2^2w_2}{2}}&=\frac{1}{p}\sum_{\boldsymbol{A}\in\mathcal A}\Theta_1\left(2p,A_1,2p;\frac{3w_1}{p}\right)\Theta_0\left(2p,A_2,2p;\frac{w_2}{p}\right),\\	
	\sum_{\substack{\boldsymbol{\alpha}\in\mathscr S^\ast \\ \boldsymbol{n}\in\boldsymbol{\alpha}+\Z^2}} (3n_1+n_2)e^{\frac{\pi i}{2}(3n_1+2n_2)^2w_1+\frac{3\pi in_1^2w_2}{2}}&=\frac1p \sum_{\boldsymbol{B}\in\mathcal B}\Theta_1\left(2p,B_1,2p;\frac{w_1}{p}\right)\Theta_0\left(2p,B_2,2p;\frac{3w_2}{p}\right),\\
	\sum_{\substack{\boldsymbol{\alpha}\in\mathscr S^\ast \\ \boldsymbol{n}\in\boldsymbol{\alpha}+\Z^2}}n_1e^{\frac{\pi i}{2}(3n_1+2n_2)^2w_1+\frac{3\pi i}{2}n_1^2w_2}&=\frac1p \sum_{\boldsymbol{B}\in\mathcal B}\Theta_0\left(2p,B_1,2p;\frac{w_1}{p}\right)\Theta_1\left(2p,B_2,2p;\frac{3w_2}{p}\right).
	\end{align*}
%	\begin{align*}
%	\mathcal C&:= \left\{(0,2),(p,p+2),(p-1,p-1),(-1,-1),(p+1,p-1),(1,-1)\right\},\\
%	\mathcal D&:= \left\{ (p+1,p-1),(1,-1),(-2,0),(p-2,p),(1,1),(p+1,p+1) \right\}.
%	\end{align*}
	The claim now again follows from Theorem \ref{quantheorem} using \eqref{shimura} and \eqref{Shimura1}.
\end{proof}

\subsection{More on double Eichler integrals}

%{\bf KB: shouldn't this also be for more general weight? KB: I think in general, we want
%	$$
%	S_{k_1}(\Gamma_1,\chi_1) \times S_{k_2}(\Gamma_2,\chi_2) \quad \rightarrow \quad \mathcal Q_{2-k_1}\left(\Gamma_1^*,\chi_1^*\right)%\times Q_{2-k_2}\left(\Gamma_2^*,\chi_2^*\right) \ \  \subset \ \  \mathcal Q_{4-k_1-k_2}\left(\Gamma_1^*\cap \Gamma_2^*, \chi_1^*\chi_2^*%%\right).
%	$$
%	Antun, could you rewrite accordingly? AM: In this generality there is no symmetric/exterior square - you must take $k_1=k_2$. KB: But that %is not what we always have. We do have weight $(\frac12,\frac32)$.}
We have an obvious map $S_{k}(\Gamma,\chi) \to \mathcal Q_{2-k}(\Gamma^\ast,\chi^\ast)$, where $\chi^*(M):=\chi(M^*)$, which assigns to $f \in S_{k}(\Gamma,\chi)$ its Eichler integral $I_f$, defined in \eqref{Eichler}.
Clearly, we also have a map from $S_{k}(\Gamma,\chi) \otimes S_{k}(\Gamma,\chi)$, actually from its symmetric square, to $(\mathcal Q_{2-k}(\Gamma^\ast,\chi^\ast))^2$, by mapping $f_1 \otimes f_2$ to $I_{f_1} I_{f_2}$.  The double Eichler integral construction $I_{f_1,f_2}$ gives rise to a map
$$ \Lambda^2\left(S_{k}(\Gamma,\chi)\right) \longrightarrow \mathcal Q_{4-2k}^2 \left(\Gamma^\ast,{\chi^\ast}^2\right)\Big\slash \left(\mathcal Q_{2-k}\left(\Gamma^\ast,\chi^\ast\right)\right)^2,$$
where  $\Lambda^2(S_{2-k}(\Gamma,\chi))$ is the second exterior power of $S_{2-k}(\Gamma,\chi)$.
To see this, it suffices to observe the simplest {\em shuffle} relation for iterated integrals
\begin{equation*} \label{shuffle}
I_{f_1,f_2}+I_{f_2,f_1}=I_{f_1} I_{f_2}.
\end{equation*}

\begin{remark}
It is now straightforward to consider even more general iterated Eichler integrals ($r \in \N$):
$$I_{f_1,...,f_r} := \int_{-\overline{\tau}}^{i \infty} \int^{i \infty}_{w_{r-1}}   \cdots \int^{i \infty}_{w_{2}} \prod_{j=1}^r \frac{f_j(w_j)}{(-i(w_j+\tau))^{2-k_j}} dw_1 \cdots dw_r,$$
where the $f_j$ are cusp forms of weight $k_j\geq\frac12$ (or possibly holomorphic forms for weight $\frac12$). We do not pursue their (mock/quantum) modular properties here -- we will address this in our future work \cite{BKM} (see also Section 9 for related comments).
\end{remark}

\section{Indefinite theta functions}

We next realize the double Eichler integrals studied in Section 5 as pieces of indefinite theta functions.

\subsection{The function $\mathcal E_1$ as an indefinite theta function}

The next lemma rewrites $\mathbb E_1(\tau):=\mathcal E_1(\frac{\tau}{p})$ in a shape to which one can apply the Euler-Maclaurin summation formula.
\begin{lemma}\label{thm:thetasum}
	We have
	\begin{equation*}\label{thetasum}
	\mathbb E_1(\tau)=\frac12\sum_{\boldsymbol{\alpha}\in\mathscr S^*}\varepsilon(\boldsymbol{\alpha})\sum_{\boldsymbol{n}\in\boldsymbol{\alpha}+\Z^2}M_2\left(\sqrt 3;\sqrt v\left(2\sqrt 3n_1+\sqrt 3n_2,n_2\right)\right)q^{-Q(\boldsymbol{n})}.
	\end{equation*}
\end{lemma}
\begin{proof}The claim follows, once we prove that
\begin{align}\notag
	&M_2\left(\sqrt{3};\sqrt{3v}(2n_1+n_2),\sqrt{v}n_2\right)\\
	\notag
	&\qquad=-\frac{\sqrt{3}}{2}(2n_1+n_2)n_2q^{Q(\boldsymbol{n})}\!\! \int_{-\overline{\tau}}^{i \infty}
	\frac{e^{\frac{3\pi i}{2}(2n_1+n_2)^2w_1}}{\sqrt{-i(w_1+\tau)}}\int_{w_1}^{i \infty}\!\!\!
	\frac{e^{\frac{\pi i n_2^2 w_2}{2}}}{\sqrt{-i(w_2+\tau)}}dw_2 dw_1
	\\
	&\qquad\quad
	-\frac{\sqrt{3}}{2}(3n_1+2n_2)n_1q^{Q(\boldsymbol{n})} \int_{-\overline{\tau}}^{i \infty}
	\frac{e^{\frac{\pi i}{2}(3n_1+2n_2)^2w_1}}{\sqrt{-i(w_1+\tau)}}
	\int_{w_1}^{i \infty}
	\frac{e^{\frac{3\pi i n_1^2 w_2}{2}}}{\sqrt{-i(w_2+\tau)}}dw_2 dw_1.	\label{M2id}
\end{align}
For simplicity we only show \eqref{M2id} for $n_1\neq 0$. Since, by \eqref{limitsM2},
$$
\lim_{\lambda\to\infty}M_2(\kappa; \lambda u_1, \lambda u_2)=0,
$$
we obtain, using \eqref{diffM2} and \eqref{middM1},
\begin{align}
\nonumber&\quad M_2(\kappa; u_1,u_2)=
-\int_{1}^{\infty} \frac{\partial}{\partial w_1}
M_2(\kappa; w_1u_1,w_1u_2) dw_1
\\
\nonumber&=
-\int_{1}^{\infty} \left(u_1
M_2^{(1,0)}(\kappa; w_1u_1,w_1u_2)
+u_2
M_2^{(0,1)}(\kappa; w_1u_1,w_1u_2) \right)dw_1\\
\nonumber&=
-2\int_{1}^{\infty}
\left(u_1 e^{-\pi u_1^2 w_1^2} M(u_2w_1)  + \frac{u_2+\kappa u_1}{\sqrt{1+\kappa^2}} e^{-\frac{\pi(u_2+\kappa u_1)^2w_1^2}{1+\kappa^2}}
M\left(w_1\frac{u_1-\kappa u_2}{\sqrt{1+\kappa^2}}\right)\right) dw_1
\nonumber\\&=
-\int_{1}^{\infty}\left(
u_1 e^{-\pi u_1^2w_1} M\left( u_2\sqrt{w_1}\right)
+ \frac{u_2+\kappa u_1}{\sqrt{1+\kappa^2}}
e^{-\frac{\pi(u_2+\kappa u_1)^2w_1}{1+\kappa^2}}
M\left(\sqrt{w_1}\frac{u_1-\kappa u_2}{\sqrt{1+\kappa^2}}\right)\right) \frac{dw_1}{\sqrt{w_1}}
\label{6.2}
\\&=\frac{i}{\sqrt{2}}
\int_{-\overline{\tau}}^{i \infty}\Bigg(
\frac{u_1}{\sqrt{v}} e^{\frac{\pi i u_1^2 w_1}{2v} } q^{\frac{u_1^2}{4v}}
M\left(\sqrt{\frac{-i(w_1+\tau)}{2v}}u_2\right)  \label{rewriteM2}
\nonumber\\ &\nonumber \qquad\quad + \frac{u_2+\kappa u_1}{\sqrt{(1+\kappa^2)v}}
e^{\frac{\pi i(u_2+\kappa u_1)^2w_1}{2\left(1+\kappa^2\right)v} } q^{\frac{(u_2+\kappa u_1)^2}{4\left(1+\kappa^2\right)v}}
M\left(\sqrt{\frac{-i(w_1+\tau)}{2}} \frac{u_1-\kappa u_2}{\sqrt{(1+\kappa^2)v}}\right)\Bigg) \frac{dw_1}{\sqrt{-i(w_1+\tau)}}.
\end{align}
Now write for $N\in\R^+$
\begin{align*}
\quad M\left(\sqrt{\frac{-i(w_1+\tau)}{2}}N\right)= \frac{iN}{\sqrt{2}}q^{\frac{N^2 }{4}}
\int_{w_1 }^{i\infty} e^{\frac{\pi i N^2w_2}{2}}  \frac{dw_2}{\sqrt{-i(w_2+\tau)}}.
\end{align*}
Plugging this into \eqref{6.2} easily yields that
\begin{align*}
&M_2(\kappa;u_1,u_2)=
-\frac{u_1}{2\sqrt{v}}
\frac{ u_2}{\sqrt{v}}
q^{\frac{u_1^2}{4v}+\frac{u_2^2}{4v}}
\int_{-\overline{\tau}}^{i \infty}
\frac{e^{\frac{\pi i  u_1^2 w_1}{2v}} }{\sqrt{-i(w_1+\tau)}}
\int_{w_1}^{i\infty}
\frac{e^{\frac{\pi i u_2^2 w_2}{2v}}}{\sqrt{-i(w_2+\tau)}}dw_2
dw_1
\\ &\quad -
\frac{u_2+\kappa u_1}{2\sqrt{(1+\kappa^2)v}}
\frac{u_1-\kappa u_2}{\sqrt{(1+\kappa^2)v}}
q^{\frac{(u_2+\kappa u_1)^2}{4\left(1+\kappa^2\right)v}
	+\frac{(u_1 - \kappa u_2)^2}{4\left(1+\kappa^2\right)v}}
\int_{-\overline{\tau}}^{i \infty}
\frac{e^{\frac{\pi i(u_2+\kappa u_1)^2w_1}{2\left(1+\kappa^2\right)v}} }{\sqrt{-i(w_1+\tau)}}
\int_{w_1}^{i\infty}
\frac{e^{\frac{\pi i(u_1-\kappa u_2)^2w_2}{2\left(1+\kappa^2\right)v} }}{\sqrt{-i(w_2+\tau)}}
dw_2
dw_1.
\end{align*}
From this it is not hard to conclude \eqref{M2id}.
\end{proof}

\subsection{The function $\mathcal E_2$ as an indefinite theta function}

We next write $\mathbb E_2(\tau):=\mathcal E_2(\frac{\tau}{p})$ as a piece of a derivative of an indefinite theta function, having an extra Jacobi variable.
%{\bf AM: I think adding the rank one case (as in DerivativeEichler file) might be a good idea - it would further elucidate our method. KB: Well, %the paper is already really long. So unless you can do it just in a few lines I would vote to not do it.}

\begin{lemma}\label{Eichler2}
	We have
	\begin{align*}
	\mathbb E_{2}(\tau) &=\frac{1}{4\pi i} \sum_{\boldsymbol{\alpha}\in\mathscr S^\ast}\\
	&\qquad\times\sum_{\boldsymbol{n}\in\boldsymbol{\alpha}+\Z^2}
	\left[\frac{\partial}{\partial z}\left(M_2\left(\sqrt{3};\sqrt{3v}(2n_1+n_2),\sqrt{v}\left(n_2-\frac{2\im(z)}{v}\right)\right)e^{2\pi in_2z} \right)\right]_{z=0}q^{-Q(\boldsymbol{n})}.
	\end{align*}
\end{lemma}
\begin{proof}
We first compute
\begin{equation}\label{defineH}\begin{split}
\frac{1}{2\pi i}&\left[\frac{\partial}{\partial z}\left(M_2\left(\sqrt 3;\sqrt{3v}\left(2n_1+n_2\right),\sqrt{v}\left(n_2-\frac{2\im(z)}{v}\right)\right)e^{2\pi in_2z}\right)\right]_{z=0}\\
&\hspace{3.2cm}=n_2M_2\left(\sqrt 3;\sqrt{3v}\left(2n_1+n_2\right),\sqrt{v}n_2\right)+\frac{1}{2\pi \sqrt{v}}e^{-\pi \left(3n_1+2n_2\right)^2v}M\left(\sqrt{3v}n_1\right).
\end{split}
\end{equation}
We show below that
\begin{align}
\nonumber&n_2M_2\left(\sqrt{3};\sqrt{3v}(2n_1+n_2),\sqrt{v}n_2\right)=-\frac{\sqrt{3}}{2\pi}(2n_1+n_2)\int_{2v}^\infty \frac{e^{-\frac{3\pi}{2}(2n_1+n_2)^2w_1}}{\sqrt{w_1}}\int_{w_1}^\infty \frac{e^{-\frac{\pi n_2^2w_2}{2}}}{w_2^{\frac32}}dw_2dw_1\\
\nonumber&+\frac{\sqrt{3}}{4\pi}(3n_1+2n_2)\int_{2v}^\infty \frac{e^{-\frac{\pi}{2}(3n_1+2n_2)^2w_1}}{\sqrt{w_1}}\int_{w_1}^\infty \frac{e^{-\frac{3\pi n_1^2w_2}{2}}}{w_2^{\frac32}}dw_2dw_1-\frac{1}{2\pi\sqrt{v}}e^{-\pi (3n_1+2n_2)^2v}M\left(\sqrt{3v}n_1\right)\\
&-\frac{\sqrt{3}n_1}{4\pi}\int_{2v}^\infty \frac{e^{-\frac{\pi}{2}(3n_1+2n_2)^2w_1}}{w_1^\frac32}\int_{w_1}^\infty\frac{e^{-\frac{3\pi n_1^2w_2}{2}}}{w_2^\frac12}dw_2dw_1.\label{showed}
\end{align}
Since the third term cancels the second term on the right-hand side of \eqref{defineH} this then implies the claim,
using that
\begin{align*}
\int_{2v}^\infty \frac{e^{-2\pi M^2w_1}}{w_1^\frac12} \int_{w_1}^\infty \frac{e^{-2\pi N^2w_2}}{w_2^\frac32}dw_2dw_1&=-q^{M^2+N^2}\int_{-\overline{\tau}}^{i\infty}\frac{e^{2\pi iM^2w_1}}{\left(-i(w_1+\tau)\right)^\frac12} \int_{w_1}^{i\infty} \frac{e^{2\pi iN^2w_2}}{\left(-i(w_2+\tau)\right)^\frac32}dw_2dw_1,\\
\int_{2v}^\infty \frac{e^{-2\pi M^2w_1}}{w_1^\frac32} \int_{w_1}^\infty \frac{e^{-2\pi N^2w_2}}{w_2^\frac12} dw_2dw_1&=-q^{N^2+M^2}\int_{-\overline{\tau}}^{i\infty}\frac{e^{2\pi iM^2w_1}}{(-i(w_1+\tau))^\frac32} \int_{w_1}^{i\infty} \frac{e^{2\pi iN^2w_2}}{(-i(w_2+\tau))^\frac12}dw_2dw_1.	
\end{align*}
To prove \eqref{showed}, we again, for simplicity, restrict to $n_1\neq0$.

Plugging in \eqref{6.2} yields
\begin{equation}\label{M2int}\begin{split}
M_2\left(\sqrt{3};\sqrt{3v}(2n_1+n_2),\sqrt{v}n_2\right)&=-\int_1^\infty\left(\sqrt{3v}(2n_1+n_2)e^{-3\pi v(2n_1+n_2)^2w_1}M\left(\sqrt{vw_1}n_2\right)\right.\\
&\left.\qquad\qquad+\sqrt{v}(3n_1+2n_2)e^{-\pi v(3n_1+2n_2)^2w_1}M\left(\sqrt{3vw_1}n_1\right)\right)\frac{dw_1}{\sqrt{w_1}}.
\end{split}
\end{equation}
Using \eqref{MG} and \eqref{G12} the first term in \eqref{M2int} multiplied by $n_2$ gives
\begin{multline}\label{t1}
-\frac{\sqrt{3v}}{2\sqrt{\pi}}|n_2|(2n_1+n_2)\int_{1}^{\infty} e^{-3\pi v(2n_1+n_2)^2w_1}\Gamma\left(-\frac12,\pi vn_2^2w_1\right)\frac{dw_1}{\sqrt{w_1}}\\
+\frac{\sqrt{3}}{\pi}(2n_1+n_2)\int_1^\infty e^{-4\pi vQ(n_1,n_2)w_1}\frac{dw_1}{w_1}.
\end{multline}
For the second term in \eqref{M2int}, we split
\begin{equation}\label{splitn2}
n_2=\frac12 (3n_1+2n_2)-\frac32 n_1.
\end{equation}
The $n_1$-term contributes to $n_2M_2$ as
\begin{multline}\label{t2}
\frac34\sqrt{\frac{v}{\pi}}|n_1|(3n_1+2n_2)\int_1^\infty e^{-\pi v(3n_1+2n_2)^2w_1}\Gamma\left(-\frac12,3\pi vn_1^2w_1\right)\frac{dw_1}{\sqrt{w_1}}\\
-\frac{\sqrt{3}}{2\pi}(3n_1+2n_2)\int_1^\infty e^{-4\pi vQ(n_1,n_2)w_1}\frac{dw_1}{w_1}.
\end{multline}
We next use that for $N\in\N_0$, $M\in\N$
\begin{equation*}
\int_1^\infty e^{-4\pi N^2vw_1}\Gamma\left(-\frac12,4\pi vM^2w_1\right)\frac{dw_1}{\sqrt{w_1}}=\frac{1}{2\sqrt{\pi v}|M|}\int_{2v}^\infty \frac{e^{-2\pi N^2w_1}}{\sqrt{w_1}}\int_{w_1}^\infty \frac{e^{-2\pi M^2w_2}}{w_2^\frac32}dw_2dw_1.
\end{equation*}
We use this to rewrite the first terms in \eqref{t1} and \eqref{t2}.
The first term in \eqref{t1} is the first term on the right-hand side of \eqref{showed}.
%\begin{align*}
%-\frac{\sqrt{3}}{2\pi}(2n_1+n_2)\int_{2v}^\infty \frac{e^{-\frac{3\pi}{2}(2n_1+n_2)^2w_1}}{\sqrt{w_1}}\int_{w_1}^\infty \frac{e^{-\frac{\pi n_2^2w_2}{2}}}{w_2^{\frac32}}dw_2dw_1.
%\end{align*}
Similarly, since $n_1\neq 0$, the first term in \eqref{t2} equals the second term in \eqref{showed}.
%\begin{align*}
%\frac{\sqrt{3}}{4\pi} (3n_1+2n_2)\int_{2v}^\infty \frac{e^{-\frac{\pi}{2}(3n_1+2n_2)^2w_1}}{\sqrt{w_1}}\int_{w_1}^\infty \frac{e^{-\frac{3\pi n_1^2w_2}{2}}}{w_2^{\frac32}}dw_2dw_1.
%\end{align*}
Now we combine the second terms in \eqref{t1} and \eqref{t2}, to get
\begin{align}
\frac{\sqrt{3} n_1}{2\pi}\int_1^\infty e^{-4\pi vQ(\boldsymbol{n})w_1}\frac{dw_1}{w_1}. \label{term1}
\end{align}
Next we compute the contribution from the first term in \eqref{splitn2},
\begin{multline*}
-\frac{\sqrt{v}}{2}\left(3n_1+2n_2\right)^2 \int_1^\infty e^{-\pi v(3n_1+2n_2)^2w_1}M\left(\sqrt{3vw_1}n_1\right)\frac{dw_1}{\sqrt{w_1}}\\
= \frac{1}{2\pi\sqrt{v}} \int_1^\infty \frac{\partial}{\partial w_1} \left(e^{-\pi v(3n_1+2n_2)^2w_1}\right) \frac{M\left(\sqrt{3vw_1}n_1\right)}{\sqrt{w_1}}dw_1.
\end{multline*}
Using integration by parts, this becomes
\begin{multline}
-\frac{1}{2\pi\sqrt{v}} e^{-\pi v(3n_1+2n_2)^2}M\left(\sqrt{3v}n_1\right) -\frac{\sqrt{3}n_1}{2\pi} \int_1^\infty e^{-4\pi vQ(n_1,n_2)w_1}\frac{dw_1}{w_1} \\
+ \frac{1}{4\pi\sqrt{v}} \int_1^\infty e^{-\pi v(3n_1+2n_2)^2w_1} \frac{M\left(\sqrt{3vw_1}n_1\right)}{w_1^\frac32}dw_1\label{intpart}.
\end{multline}
The second term now cancels \eqref{term1} and the first term equals the third term in \eqref{showed}.

To rewrite the final term in \eqref{intpart}, we use that for $M, N\in\Z$ with $N\neq 0$
\begin{align*}
\int_1^\infty e^{-4\pi vM^2w_1}\frac{M\left(2\sqrt{vw_1}N\right)}{w_1^\frac32}dw_1=-2\sqrt{v}N\int_{2v}^\infty \frac{e^{-2\pi M^2w_1}}{w_1^\frac32}\int_{w_1}^\infty \frac{e^{-2\pi N^2 w_2}}{w_2^\frac12}dw_2dw_1.
\end{align*}
Thus the last term in \eqref{intpart} gives the final term in \eqref{showed}.

\end{proof}

\section{Asymptotic behavior of multiple Eichler integrals and proof of Theorem \ref{maintheorem}}
In this section, we asymptotically relate $F_j$ and $\mathbb E_j$.
\subsection{Asymptotic behavior of $\mathbb E_1$}
Write
\[
F_1\left(e^{2\pi i\frac{h}{k}-t}\right)\sim\sum_{m\geq 0} a_{h, k}(m) t^m\quad \left(t\to 0^+\right).
\]

The goal of this subsection is to prove the following.
\begin{theorem}\label{asagree}
We have, for $h,k\in\Z$ with $k>0$ and $\gcd(h,k)=1,$
\[
\mathbb{E}_1\left(\frac{h}{k}+\frac{it}{2\pi}\right){\sim } \sum_{m\geq 0} a_{-h, k}(m) (-t)^m\quad \left(t\to 0^+\right).
\]
\end{theorem}
\begin{proof}
We use Lemma \ref{thm:thetasum} and the fact that $M_2$ is an even function, to rewrite
\begin{align*}
\mathbb E_1(\tau)& = \frac12 \sum_{\boldsymbol{\alpha}\in\mathscr S} \varepsilon(\boldsymbol{\alpha}) \sum_{\boldsymbol{n}\in\boldsymbol{\alpha}+\N_0^2} M_2\left(\sqrt{3};\sqrt{v}\left(2\sqrt{3}n_1+\sqrt{3}n_2,n_2\right)\right)q^{-Q(\boldsymbol{n})}\\
&\quad+\frac12 \sum_{\boldsymbol{\alpha}\in\widetilde{\mathscr{S}}} \widetilde{\varepsilon}(\boldsymbol{\alpha})\sum_{\boldsymbol{n}\in\boldsymbol{\alpha}+\N_0^2}M_2\left(\sqrt{3};\sqrt{v}\left(-2\sqrt{3}n_1+\sqrt{3}n_2,n_2\right)\right)q^{-Q(-n_1,n_2)},
\end{align*}
where
\begin{align*}
\widetilde{\mathscr S}:=\left\{\left(1-\alpha_1,\alpha_2\right):\boldsymbol{\alpha}\in \mathscr S\right\},\quad
\widetilde{\varepsilon}(\boldsymbol{\alpha}):=\varepsilon(1-\alpha_1,\alpha_2).
\end{align*}
To apply the Euler-Maclaurin summation formula directly, we turn every sgn into $\sgn^*$, where $\sgn^\ast(x):=\sgn(x)$ for $x\neq0$ and $\sgn^\ast(0):=1$. To be more precise, we set
\begin{align}\notag
&M_2^\ast\left(\sqrt{3};\sqrt{3}(2x_1+x_2),x_2\right):=\sgn^\ast(x_1)\sgn^\ast(x_2)+E_2\left(\sqrt{3};\sqrt{3}(2x_1+x_2),x_2\right)\\
\label{M2star}
&\hspace{6cm}-\sgn^\ast(x_2)E\left(\sqrt{3}(2x_1+x_2)\right)-\sgn^\ast(x_1)E(3x_1+2x_2).
\end{align}
Using that
\begin{equation}\label{limit}
M_2\left(\sqrt{3};\sqrt{3}x_2,x_2\right)-\lim_{x_1\to 0^+} M_2^\ast \left(\sqrt{3};\sqrt{3} \left(\pm 2x_1+x_2\right),x_2\right) = \pm M(2x_2),
\end{equation}
we then split
$$
\mathbb E_1(\tau)=\mathcal E_1^\ast(\tau) +H_1(\tau)
$$
with
\begin{align*}
\mathcal E_1^\ast(\tau)&:=\frac12 \sum_{\boldsymbol{\alpha}\in\mathscr{S}}\varepsilon (\boldsymbol{\alpha})\sum_{\boldsymbol{n}\in\boldsymbol{\alpha}+\N_0^2}M_2^\ast\left(\sqrt{3};\sqrt{v}\left(2\sqrt{3}n_1+\sqrt{3}n_2,n_2\right)\right)q^{-Q(\boldsymbol{n})}\\
&\qquad+\frac12 \sum_{\boldsymbol{\alpha}\in\widetilde{\mathscr{S}}}\widetilde{\varepsilon} (\boldsymbol{\alpha})\sum_{\boldsymbol{n}\in\boldsymbol{\alpha}+\N_0^2}M_2^\ast\left(\sqrt{3};\sqrt{v}\left(-2\sqrt{3}n_1+\sqrt{3}n_2,n_2\right)\right)q^{-Q(-n_1,n_2)},\\
H_1(\tau)&:=-\frac12 \sum_{m\in\frac1p+\N_0}M\left(2\sqrt{v}m\right)q^{-m^2}+\frac12\sum_{m\in1-\frac1p +\N_0}M\left(2\sqrt{v}m\right)q^{-m^2}.
\end{align*}
Note that for $n_1=0$ we take the limit $n_1\to0$ in the $M_2^*$-functions.

We proceed as in Subsection \ref{functionf1} to determine the asymptotic behavior of $\mathcal E_1^\ast$ and $H_1$. Firstly we rewrite
\begin{align*}
\mathcal{E}_1^\ast\left(\frac{h}{k}+\frac{it}{2\pi}\right)&=\sum_{\boldsymbol{\alpha}\in\mathscr S}\varepsilon(\boldsymbol{\alpha})\!\!\!\sum_{0\leq\boldsymbol{\ell}\leq\frac{kp}{\delta}-1}\!\!\!\!\!\! e^{-2\pi i\frac{h}{k}Q(\boldsymbol{\ell}+\boldsymbol{\alpha})}\sum_{n\in\frac{\delta}{kp}(\boldsymbol{\ell}+\boldsymbol{\alpha})+\N_0^2}\mathcal F_3\left(\frac{kp}{\delta}\sqrt{t}\boldsymbol{n}\right)\\
&\quad+\!\!\!\sum_{\boldsymbol{\alpha}\in\widetilde{\mathscr S}}\widetilde{\varepsilon}(\alpha)\!\!\!\sum_{0\leq\boldsymbol{\ell}\leq\frac{kp}{\delta}-1}\!\!\!\!\!\!e^{-2\pi i\frac{h}{k}Q(-(\ell_1+{\alpha_1}),\ell_2+{\alpha_2})}\sum_{n\in \frac{\delta}{kp}(\boldsymbol{\ell}+\boldsymbol{\alpha})+\N_0^2}\widetilde{\mathcal F}_3\left(\frac{kp}{\delta}\sqrt{t}\boldsymbol{n}\right),
\end{align*}
where
\begin{align*}
\mathcal F_3(\boldsymbol{x})&:=\frac12M_2^\ast\left(\sqrt{3};\frac1{\sqrt{2\pi}}\left(\sqrt{3}\left(2x_1+x_2\right),x_2\right)\right)e^{Q(\boldsymbol{x})}, \qquad\widetilde{\mathcal F}_3(\boldsymbol{x}):=\mathcal F_3(-x_1,x_2).
\end{align*}

The contribution from the $\mathcal F_3$ term to the first term in \eqref{EulerMcLaurin} is
$$
\frac{\delta^2}{k^2p^2t} \mathcal I_{\mathcal F_3} \sum_{\boldsymbol{\alpha}\in\mathscr S}\varepsilon(\boldsymbol{\alpha})\sum_{0\leq \boldsymbol{\ell}\leq\frac{kp}{\delta}-1}e^{-2\pi i\frac{h}{k}Q(\boldsymbol{\ell}+\boldsymbol{\alpha})}=0,
$$
conjugating \eqref{sums}. In the same way the main term coming from $\widetilde{\mathcal F_3}$ is shown to vanish.

The contribution to the second term of Euler-Maclaurin is

\begin{multline*}
%&- \frac1{\sqrt{t}}\sum_{r\geq 0}\frac{\left(\frac{kp\sqrt{t}}{\sqrt{\delta}}\right)^{r}}{(r+1)!}\int_0^\infty g^{(0, r)}(x_1, 0)dx_1
%\sum_{({\alpha_1}, {\alpha_2})\in\mathscr{S}}\varepsilon({\alpha_1}, {\alpha_2})\\ &\hspace{5cm}\times\sum_{0\leq\ell_1, \ell_2\leq\frac{kp}{\delta}-1} e^{-2\pi i\frac{h}{k}Q\left(\ell_1+{\alpha_1}, \ell_2+{\alpha_2}\right)}
%B_{r+1}\left(\frac{\delta\left(\ell_2+{\alpha_2}\right)}{kp}\right)\\
-2\sum_{\boldsymbol{\alpha}\in\mathscr{S}^\ast}\varepsilon(\boldsymbol{\alpha})\sum_{0\leq\boldsymbol{\ell}\leq\frac{kp}{\delta}-1} e^{-2\pi i\frac{h}{k}Q\left(\boldsymbol{\ell}+\boldsymbol{\alpha}\right)} \sum_{n_2\geq 0}\frac{B_{2n_2+2}\left(\frac{\delta\left(\ell_2+{\alpha_2}\right)}{kp}\right)}{(2n_2+2)!} \\
\times \int_0^\infty\left( \mathcal F_3^{(0, 2n_2+1)}(x_1, 0)+\widetilde{\mathcal{F}}_3^{(0,2n_2+1)}(x_1,0)\right)dx_1\left(\frac{k^2p^2t}{\delta^2}\right)^{n_2}.
\end{multline*}
We now claim that
\begin{equation}\label{intagru}
\int_0^\infty \left(\mathcal F_3^{(0,2n_2+1)}(x_1,0) +\widetilde{\mathcal F}_3^{(0,2n_2+1)}(x_1,0)\right)dx_1
 = (-1)^{n_2}\int_0^\infty \mathcal F_1^{(0,2n_2+1)}(x_1,0)dx_1.
\end{equation}
Firstly the right-hand side of \eqref{intagru} equals
\begin{equation}\label{intagree}
\left[\frac{\partial^{2n_2+1}}{\partial x_2^{2n_2+1}}\int_0^\infty \mathcal F_1(x_1,x_2)dx_1\right]_{x_2=0}=\left[\frac{\partial^{2n_2+1}}{\partial x_2^{2n_2+1}}\left(e^{-\frac{x_2^2}{4}}\int_0^\infty e^{-3\left(x_1+\frac{x_2}{2}\right)^2}dx_1\right)\right]_{x_2=0}.
\end{equation}
Now the integral in \eqref{intagree} evaluates as
$$
\sqrt{\frac{\pi}{3}}\int_{\frac{\sqrt{3} x_2}{2\sqrt{\pi}}}^{\infty}e^{-\pi x_1^2}dx_1=\frac{\sqrt{\pi}}{2\sqrt{3}}\left(1-E\left(\frac{\sqrt{3}x_2}{2\sqrt{\pi}}\right)\right).
$$
Thus \eqref{intagree} becomes
\begin{equation}\label{mainold}
\frac{\sqrt{\pi}}{2\sqrt{3}}\left[\frac{\partial^{2n_2+1}}{\partial x_2^{2n_2+1}}\left(e^{-\frac{x_2^2}{4}}\left(1-E\left(\frac{\sqrt{3}x_2}{2\sqrt{\pi}}\right)\right)\right)\right]_{x_2=0} \!\!\!\!\!\!\!\!= -\frac{\sqrt{\pi}}{2\sqrt{3}}\left[\frac{\partial^{2n_2+1}}{\partial x_2^{2n_2+1}}\left(e^{-\frac{x_2^2}{4}}E\left(\frac{\sqrt{3}x_2}{2\sqrt{\pi}}\right)\right)\right]_{x_2=0}.
\end{equation}
To compute the left-hand side of \eqref{intagru}, we decompose, according to \eqref{M2star},
\begin{align*}\label{splitE2}
M_2^\ast\left(\sqrt{3},\sqrt{3}(2x_1+x_2),x_2\right) =\sgn^*(x_1)\sgn^*(x_2)+h_1(\boldsymbol{x})
-\sgn^*(x_2)h_2(\boldsymbol{x})-\sgn^*(x_1)h_3(\boldsymbol{x}),
\end{align*}
where
\begin{align*}
h_1(\boldsymbol{x}):=E_2\left(\sqrt{3};\sqrt{3}\left(2x_1+x_2\right),x_2\right),\quad h_2(\boldsymbol{x}):=E\left(\sqrt{3}\left(2x_1+x_2\right)\right)\!,\quad h_3(\boldsymbol{x}):=E\left(3x_1+2x_2\right).
\end{align*}
Setting
\[
a_0(\boldsymbol{x}):=e^{Q(\boldsymbol{x})},\qquad a_j(\boldsymbol{x}):=h_j\left(\frac{1}{\sqrt{2\pi}}(\boldsymbol{x})\right)e^{Q(\boldsymbol{x})},
\]
we then obtain
\begin{align*}
&\mathcal F_3^{(0,2n_2+1)}(x_1,0)+\widetilde{\mathcal F}_3^{(0,2n_2+1)}(x_1,0)\\
&\quad=\frac12\left(a_0^{(0,2n_2+1)}(x_1,0)+a_1^{(0,2n_2+1)}(x_1,0)-a_2^{(0,2n_2+1)}(x_1,0)-a_3^{(0,2n_2+1)}(x_1,0)\right)\\
&\qquad+\frac12\left(-a_0^{(0,2n_2+1)}(-x_1,0)+a_1^{(0,2n_2+1)}(-x_1,0)-a_2^{(0,2n_2+1)}(-x_1,0)+a_3^{(0,2n_2+1)}(-x_1,0)\right)\\
&\quad=a_0^{(0,2n_2+1)}(x_1,0)-a_2^{(0,2n_2+1)}(x_1,0),
\end{align*}
using that $a_0$ and $a_1$ are even and $a_2$ and $a_3$ are odd. Plugging in the definition of $a_0$ and $a_2$, we need to consider
\begin{equation}\label{diff}
-\left[\frac{\partial^{2n_2+1}}{\partial x_2^{2n_2+1}}\left(e^{x_2^2}\int_0^\infty e^{3x_1^2+3x_1x_2}M\left(\sqrt{\frac{3}{2\pi}}(2x_1+x_2)\right)dx_1\right)\right]_{x_2=0}.
\end{equation}
Changing variables $w:=\sqrt{\frac{3}{2\pi}}(2x_1+x_2)$, the function in \eqref{diff} before differentiation is
$$
-\sqrt{\frac{\pi}{6}}e^{\frac{x_2^2}{4}}\int_{\sqrt{\frac{3}{2\pi}}x_2}^{\infty}M(w)e^{\frac{\pi w^2}{2}}dw = -\sqrt{\frac{\pi}{6}}e^{\frac{x_2^2}{4}}\left(\int_0^\infty M(w)e^{\frac{\pi w^2}{2}}dw-\int_0^{\sqrt{\frac{3}{2\pi}}x_2}M(w)e^{\frac{\pi w^2}{2}}dw\right).
$$
The first integral vanishes upon differentiating an odd number of times and then setting $x_2=0$. In the second integral we decompose $M(w) = E(w) - 1$.
The contribution of the $E$-function vanishes, since $E$ is an odd function. We are left with
$$
-\sqrt{\frac{\pi}{6}}\!\!\left[\frac{\partial^{2n_2+1}}{\partial x_2^{2n_2+1}}\left(e^{\frac{x_2^2}{4}}\int_0^{\sqrt{\frac{3}{2\pi}}x_2}e^{\frac{\pi w^2}{2}}dw\right)\right]_{x_2=0} \!\!\!\!\!\!\!\!\!= -\sqrt{\frac{\pi}{6}}i^{-2n_2-1}\!\!\left[\frac{\partial^{2n_2+1}}{\partial x_2^{2n_2+1}}\left(e^{-\frac{x_2^2}{4}}\int_0^{\sqrt{\frac{3}{2\pi}}x_2i}e^{\frac{\pi w^2}{2}}dw\right)\right]_{x_2=0}\!\!\!.
$$
The integral equals
$$
i\sqrt 2 \int_0^{\frac{\sqrt{3}x_2}{2\sqrt{\pi}}}e^{-\pi w^2}dw = \frac{i}{\sqrt 2}E\left(\frac{\sqrt{3}x_2}{2\sqrt{\pi}}\right).
$$
Thus we obtain
$$
\frac{\sqrt{\pi}}{2\sqrt{3}} (-1)^{n_2+1}\left[\frac{\partial^{2n_2+1}}{\partial x_2^{2n_2+1}}\left(e^{-\frac{x_2^2}{4}}E\left(\frac{\sqrt{3}x_2}{2\sqrt{\pi}}\right)\right)\right]_{x_2=0},
$$
as claimed, by comparing with \eqref{mainold}.

In the same way one can show that the third term in Euler-Maclaurin equals 
\begin{multline*}
- 2\sum_{\boldsymbol{\alpha}\in\mathscr{S}^\ast}\varepsilon(\boldsymbol{\alpha})\sum_{0\leq\boldsymbol{\ell}\leq\frac{kp}{\delta}-1} e^{-2\pi i\frac{h}{k}Q\left(\boldsymbol{\ell}+\boldsymbol{\alpha}\right)}\sum_{n_1\geq 0}\frac{B_{2n_1+2}\left(\frac{\delta\left(\ell_1+{\alpha_1}\right)}{kp}\right)}{(2n_1+2)!}\int_0^\infty \mathcal F_1^{(2n_1+1,0)}(0,x_2)dx_2\left(-\frac{k^2p^2t}{\delta^2}\right)^{n_1}.
\end{multline*}

The contribution to the final term is, pairing as in Section 4
\begin{multline*}
2\sum_{\boldsymbol{\alpha}\in\mathscr{S}^\ast}\varepsilon(\boldsymbol{\alpha}) \sum_{0\leq\boldsymbol{\ell}\leq\frac{kp}{\delta}-1} e^{-2\pi i\frac{h}{k}Q\left(\boldsymbol{\ell}+\boldsymbol{\alpha}\right)}\sum_{n_1, n_2\geq 0\atop{n_1\equiv n_2\pmod{2}}}\frac{B_{n_1+1}\left(\frac{\delta\left(\ell_1+{\alpha_1}\right)}{kp}\right)}{(n_1+1)!}\frac{B_{n_2+1}\left(\frac{\delta\left(\ell_2+{\alpha_2}\right)}{kp}\right)}{(n_1+1)!}\\
\times\left(\mathcal F_3^{(\boldsymbol{n})}(\boldsymbol{0})-(-1)^{n_1}\widetilde{\mathcal F}_3^{(\boldsymbol{n})}(\boldsymbol{0})\right)\left(\frac{kp\sqrt{t}}{\delta}\right)^{n_1+n_2}
.
\end{multline*}
We next show that
\begin{equation*}
\mathcal F_3^{(n_1, n_2)}(\boldsymbol{0})-(-1)^{n_1}\widetilde{\mathcal F}_3^{(n_1, n_2)}(\boldsymbol{0})=i^{n_1+n_2}\mathcal F_1^{(n_1,n_2)}(\boldsymbol{0}).
\end{equation*}

\noindent For this, we compute
\begin{align*}
&\mathcal F_3^{(n_1,n_2)}(\boldsymbol{0})-(-1)^{n_1}\widetilde{\mathcal F}_3^{(n_1,n_2)}(\boldsymbol{0})=a_0^{(n_1,n_2)}(\boldsymbol{0})-a_3^{(n_1,n_2)}(\boldsymbol{0}).
\end{align*}
%Differentiating gives
%\begin{multline}\label{a3vanish}
%a_3^{(r,s)}(0,0)=\sum_{1\le j_1\le r\atop{1\le j_2\le s}} \binom{r}{j_1}\binom{s}{j_2}\underbrace{\left[\frac{\partial^{j_1}}{\partial y^{j_1}}\circ\frac{\partial^{j_2}}{\partial x^{j_2}}\left(E\left(\frac1{\sqrt{2\pi}} \left(3x+2y\right)\right)\right)\right]_{x=y=0}}_{2^{j_1}3^{j_2}\sqrt{2\pi}^{-j_1-j_2} E^{(j_1+j_2)}(0)}\\
%\times\left[\frac{\partial^{r-j_1}}{\partial y^{r-j_1}}\circ\frac{\partial^{s-j_2}}{\partial x^{s-j_2}}e^{3x^2+3xy+y^2}\right]_{x=y=0}.
%\end{multline}
%Because the $E$ function is odd, only the contribution from $j_1+j_2$ odd survives. So we take a derivative
%\[
%\left[\frac{\partial^{m}}{\partial y^{m}}\frac{\partial^{n}}{\partial x^{n}}e^{3x^2+3xy+y^2}\right]_{x=y=0}
%\]
%with $m+n$ odd. But this is $0$.
Since $a_3(-x_1,-x_2)=-a_3(\boldsymbol{x})$, we obtain
$$
a_3^{(n_1,n_2)}(\boldsymbol{0})=(-1)^{n_1+n_2+1}a_3^{(n_1,n_2)}(\boldsymbol{0}).
$$
Because in the sums of interest $n_1\equiv n_2\pmod{2}$, the contribution of $a_3$ vanishes. As claimed, we are left with
\begin{align*}
a_0^{(n_1,n_2)}(\boldsymbol{0})&=i^{n_1+n_2}\left[\frac{\partial^{n_1}}{\partial x_1^{n_1}}\ \frac{\partial^{n_2}}{\partial x_2^{n_2}}e^{-Q(\boldsymbol{x})}\right]_{\boldsymbol{x}=0}= i^{n_1+n_2}\mathcal F_1^{(\boldsymbol{n})}(\boldsymbol{0}).
\end{align*}
%So the last term in Euler-Maclaurin becomes
%\begin{multline*}\label{tmp2}
%2 \sum_{\alpha\in\mathscr{S}^\ast}\varepsilon(\alpha) \!\!\!\!\!\!\sum_{0\leq\ell\leq\frac{kp}{\delta}-1} \!\!\!\!\!\!e^{-2\pi i\frac{h}{k}Q\left(\ell+\alpha\right)}\\
%\times \sum_{n_1, n_2\geq 0\atop{n_1\equiv n_2\pmod{2}}}\frac{B_{n_1+1}\left(\frac{\delta\left(\ell_1+{\alpha_1}\right)}{kp}\right)}{(n_1+1)!}\frac{B_{n_2+1}\left(\frac{\delta\left(\ell_2+{\alpha_2}\right)}{kp}\right)}{(n_2+1)!} \mathcal F_1^{(n_1,n_2)}(0,0)\left(-\frac{kp}{\delta}\sqrt{t}\right)^{n_1+n_2}
%.
%\end{multline*}

Finally, the contribution from $H_1$ gives, observing that the Euler-Maclaurin main term vanishes,
\begin{align*}
\sum_{0\leq r \leq \frac{kp}{\delta}-1} e^{-2\pi i \frac{h}{k}\left(r+\frac1p\right)^2} \sum_{m\geq 0} \frac{B_{2m+1}\left(\frac{\delta\left(r+\frac{1}{p}\right)}{kp}\right)}{(2m+1)!}\mathcal F_4^{(2m)}(0)\left(\frac{k^2p^2}{\delta^2}t\right)^m
\end{align*}
with $\mathcal F_4(x):= M(\sqrt{\frac{2}{\pi}}x)e^{x^2}$. The claim then follows, observing that
\begin{equation*}
\mathcal F_4^{(2m)}(0)=(-1)^{m+1}\left[\frac{\partial^{2m}}{\partial x^{2m}} e^{-x^2}\right]_{x=0} = (-1)^{m+1} \mathcal F_2^{(2m)}(0).
\end{equation*}
\end{proof}

\subsection{Asymptotics of $\mathcal E_2$}
We write
$$
F_2\left(e^{2\pi i\frac{h}{k}-t}\right)\sim \sum_{m\geq 0} b_{h,k}(m) t^m \qquad \left(t\to 0^+\right),
$$
\begin{theorem}\label{matchas2}
	We have, for $h,k\in\Z$ with $k>0$ and $\gcd(h,k)=1,$
	$$
	\mathbb E_2\left(\frac{h}{k} + \frac{it}{2\pi}\right) \sim  \sum_{m\geq 0} b_{-h,k}(m) (-t)^{m}\qquad \left(t\to 0^+\right).
	$$
\end{theorem}
\begin{proof}
	We write, using Lemma \ref{Eichler2} and \eqref{defineH}
	$$
	\mathbb E_2(\tau) = \mathcal E_{2,1}(\tau) +\mathcal E_{2,2}(\tau),
	$$
	where
	\begin{align*}
	\mathcal E_{2,1}(\tau)&:= \frac12\sum_{\boldsymbol{\alpha}\in\mathscr S} \eta(\boldsymbol{\alpha}) \sum_{\boldsymbol{n}\in\boldsymbol{\alpha}+\N_0^2} n_2M_2\left(\sqrt{3};\sqrt{3v}(2n_1+n_2),\sqrt{v}n_2\right)q^{-Q(\boldsymbol{n})}\\
	&\qquad + \frac12\sum_{\boldsymbol{\alpha}\in\widetilde{\mathscr S}} \widetilde{\eta}(\boldsymbol{\alpha}) \sum_{\boldsymbol{n}\in\boldsymbol{\alpha}+\N_0^2} n_2M_2\left(\sqrt{3};\sqrt{3v}(-2n_1+n_2),\sqrt{v}n_2\right)q^{-Q(-n_1,n_2)},\\
	\mathcal E_{2,2}(\tau)&:= \frac{1}{4\pi\sqrt{v}}\sum_{\boldsymbol{\alpha}\in\mathscr S} \eta(\boldsymbol{\alpha}) \sum_{\boldsymbol{n}\in\boldsymbol{\alpha}+\N_0^2} e^{-\pi(3n_1+2n_2)^2 v} M\left(\sqrt{3v}n_1\right)q^{-Q(\boldsymbol{n})}\\
	&\qquad + \frac{1}{4\pi\sqrt{v}}\sum_{\boldsymbol{\alpha}\in\widetilde{\mathscr S}} \widetilde{\eta}(\boldsymbol{\alpha}) \sum_{\boldsymbol{n}\in\boldsymbol{\alpha}+\N_0^2} e^{-\pi(-3n_1+2n_2)^2 v} M\left(\sqrt{3v}n_1\right)q^{-Q(-n_1,n_2)},
	\end{align*}
	where $\widetilde{\eta}(\boldsymbol{\alpha}):=\eta(1-\alpha_1,\alpha_2)$.
	We then again use \eqref{limit}, to split
	$$
	\mathcal E_{2,1}(\tau) = \mathcal E_{2}^*(\tau) + H_2(\tau),
	$$
	where
	\begin{align*}
	\mathcal E_2^*(\tau)&:=\frac12\sum_{\boldsymbol{\alpha}\in\mathscr S} \eta(\boldsymbol{\alpha}) \sum_{\boldsymbol{n}\in\boldsymbol{\alpha}+\N_0^2} n_2M_2^*\left(\sqrt{3};\sqrt{3v}(2n_1+n_2),\sqrt{v}n_2\right)q^{-Q(\boldsymbol{n})}\\
	&\qquad+\frac12\sum_{\boldsymbol{\alpha}\in\widetilde{\mathscr S}} \widetilde{\eta}(\boldsymbol{\alpha}) \sum_{\boldsymbol{n}\in\boldsymbol{\alpha}+\N_0^2} n_2M_2^*\left(\sqrt{3};\sqrt{3v}(-2n_1+n_2),\sqrt{v}n_2\right)q^{-Q(-n_1,n_2)},\\
	H_2(\tau)&:= \frac12\sum_{\beta\in\left\{\frac1p,1-\frac1p\right\}} \sum_{m\in\beta+\N_0} m M\left(2\sqrt{v}m\right) q^{-m^2}.
	\end{align*}
	Using that $\lim_{x\to 0^+}M^\ast(\pm x)=\mp 1$, where
	we let $M^*(x):=E(x)-{\rm sgn}^*(x)$, we split
	$$
	\mathcal E_{2,2}(\tau) = \mathcal E_{2,2}^*(\tau)+H_3(\tau),
	$$
	where
	\begin{align*}
	\mathcal E_{2,2}^*(\tau):=&\frac{1}{4\pi \sqrt{v}}\sum_{\boldsymbol{\alpha}\in\mathscr S}\eta(\boldsymbol{\alpha})\sum_{\boldsymbol{n}\in\boldsymbol{\alpha}+\N_0^2}e^{-\pi (3n_1+2n_2)^2v}M^\ast\left(\sqrt{3v}n_1\right)q^{-Q(\boldsymbol{n})}\\
	&+\frac{1}{4\pi \sqrt{v}}\sum_{\boldsymbol{\alpha}\in\widetilde{\mathscr S}}\widetilde{\eta}(\boldsymbol{\alpha})\sum_{\boldsymbol{n}\in\boldsymbol{\alpha}+\N_0^2}e^{-\pi (-3n_1+2n_2)^2v}M^\ast\left(-\sqrt{3v}n_1\right)q^{-Q(-n_1,n_2)},\\
	H_3(\tau):=&\frac{1}{4\pi \sqrt{v}}\sum_{\beta\in\left\{\frac1p,1-\frac1p\right\}}\sum_{m\in \beta+\N_0}e^{-4\pi m^2v}q^{-m^2}.
	\end{align*}
	
We first investigate asymptotic properties of $\mathcal E_{2}^\ast$. Writing $\mathcal G_3(\boldsymbol{x}):= x_2\mathcal F_3(\boldsymbol{x})$ and \\ $\widetilde{\mathcal G}_3(\boldsymbol{x}):=\mathcal G_3(-x_1,x_2)$, we have
\begin{align*}
\mathcal E_2^\ast\left(\frac{h}{k}+\frac{it}{2\pi}\right)&= \frac{1}{\sqrt{t}}\sum_{\boldsymbol{\alpha}\in\mathscr S} \eta(\boldsymbol{\alpha})\sum_{0\leq \boldsymbol{\ell}\leq\frac{kp}{\delta}-1}e^{-2\pi i\frac{h}{k}Q(\boldsymbol{\ell}+\boldsymbol{\alpha})}\sum_{\boldsymbol{n}\in\frac{\delta}{kp}(\boldsymbol{\ell}+\boldsymbol{\alpha})+\N_0^2}\mathcal G_3\left(\frac{kp}{\delta}\sqrt{t}\boldsymbol{n}\right)\\
&\qquad+\frac{1}{\sqrt{t}}\sum_{\boldsymbol{\alpha}\in\widetilde{\mathscr S}}\widetilde{\eta}(\boldsymbol{\alpha}) \sum_{0\leq \boldsymbol{\ell}\leq\frac{kp}{\delta}-1}e^{-2\pi i\frac{h}{k}Q(-\ell_1-\alpha_1,\ell_2+\alpha_2)}\sum_{\boldsymbol{n}\in\frac{\delta}{kp}(\boldsymbol{\ell}+\boldsymbol{\alpha})+\N_0^2}\widetilde{\mathcal G}_3\left(\frac{kp}{\delta}\sqrt{t}\boldsymbol{n}\right).
\end{align*}
The contribution from $\mathcal G_3$ to the Euler-Maclaurin main term is, as in Subsection \ref{functionf2}, 
\begin{align*}
\frac{\delta^2}{k^2p^2t^\frac{3}{2}}\mathcal I_{\mathcal G_3} \sum_{\boldsymbol{\alpha}\in \mathscr{S}}\eta(\boldsymbol{\alpha}) \sum_{0\leq \boldsymbol{\ell}\leq \frac{kp}{\delta}-1} e^{-2\pi i \frac{h}{k} Q(\boldsymbol{\ell}+\boldsymbol{\alpha})} =0.
\end{align*}
In the same way we see that the contribution from $\widetilde{\mathcal G}_3$ to the main term vanishes.

The contribution to the second term in Euler-Maclaurin is, as in Subsection \ref{functionf2},
\begin{align*}
-2\sum_{\boldsymbol{\alpha}\in\mathscr S^*} & \sum_{0\leq \boldsymbol{\ell}\leq\frac{kp}{\delta}-1}  e^{-2\pi i\frac{h}{k}Q(\boldsymbol{\ell}+\boldsymbol{\alpha})}\sum_{n_2\geq 1} \frac{B_{2n_2+1}\left(\frac{\delta(\ell_2+\alpha_2)}{kp}\right)}{(2n_2+1)!} \\
& \qquad\qquad\times \int_0^\infty \left(\mathcal G_3^{(0,2n_2)}(x_1,0)+\widetilde{\mathcal G}_3^{(0,2n_2)}(x_1,0)\right)dx_1\left(\frac{kp}{\delta}\right)^{2n_2-1}t^{n_2-1}.
\end{align*}

We claim that
\begin{equation}\label{matchGint}
\int_0^\infty \left(\mathcal G_3^{(0,2n_2)}(x_1,0) +\widetilde{\mathcal G}_3^{(0,2n_2)}(x_1,0)\right) dx_1=(-1)^{n_2+1}\int_0^\infty \mathcal G_1^{(0,2n_2)}(x_1,0)dx_1.
\end{equation}
Since we need to differentiate the $x_2$-factor exactly once,
we have
\begin{align*}
\mathcal G_3^{(0,2n_2)}(x_1,0)+\widetilde{\mathcal{G}}_3^{(0,2n_2)} (x_1,0)&=2n_2\left(\mathcal{F}_3^{(0,2n_2-1)}(x_1,0)+\widetilde{\mathcal{F}}_3^{(0,2n_1-1)}(x_1,0)\right),\\
\mathcal G_1^{(0,2n_2)} (x_1,0)&=2n_2 \mathcal F_1^{(0,2n_2-1)}(x_1,0).
\end{align*}
The claim \eqref{matchGint} then follows from \eqref{intagru}.
This gives the correspondence to \eqref{second}.

The third term in Euler-Maclaurin is, in the same way,
\begin{align*}
-2\sum_{\boldsymbol{\alpha}\in\mathscr S^\ast} & \sum_{0\leq \boldsymbol{\ell}\leq \frac{kp}{\delta}-1}e^{-2\pi i\frac{h}{k}Q(\boldsymbol{\ell}+\boldsymbol{\alpha})}\sum_{n_1\geq 0} \frac{B_{2n_1+1}\left(\frac{\delta\left(\ell_1+\alpha_1\right)}{kp}\right)}{(2n_1+1)!}\\ 
& \qquad\qquad\times\int_0^\infty \left(\mathcal G_3^{(2n_1,0)} (0,x_2)-\widetilde{\mathcal G}_3^{(2n_1,0)} (0,x_2)\right)dx_2\left(\frac{kp}{\delta}\right)^{2n_1-1}t^{n_1-1}.
\end{align*}
To relate this to \eqref{thirdterm} (skipping the $n_1=0$ term in both cases), we compute that
\begin{align*}
\mathcal G_3^{(2n_1,0)}(0,x_2)-\widetilde{\mathcal G}_3^{(2n_1,0)} (0,x_2)=x_2\left(a_0^{(2n_1,0)}(0,x_2)-a_3^{(2n_1,0)}(0,x_2)\right).
\end{align*}
Note that
\begin{align}\label{computediff}
a_0(\boldsymbol{x})-a_3(\boldsymbol{x})=-e^{Q(\boldsymbol{x})}M^\ast\left(\sqrt{\frac1{2\pi}}(3x_1+2x_2)\right).
\end{align}
We next show that
\begin{multline}\label{want}
 \int_0^\infty x_2\left(a_0^{(2n_1,0)}(0,x_2)-a_3^{(2n_1,0)}(0,x_2)\right)dx_2\\
=(-1)^{n_1+1}\int_0^\infty x_2\left[\frac{\partial^{2n_1}}{\partial x_1^{2n_1}}\left(e^{-x_2^2-3x_1x_2-3x_1^2}\right)\right]_{x_1=0}dx_2+\frac1{\sqrt{2}}\left[\frac{\partial^{2n_1}}{\partial x_1^{2n_1}}e^{\frac{3x_1^2}{4}}\right]_{x_1=0},
\end{multline}
where the first terms on the right-hand side corresponds to \eqref{thirdterm}. We write it as
\[
(-1)^{n_1+1}\left[\frac{\partial^{2n_1}}{\partial x_1^{2n_1}} \left(e^{-\frac{3x_1^2}{4}} \int_0^\infty x_2 e^{-\left(x_2+\frac{3x_1}{2}\right)^2}dx_2\right)\right]_{x_1=0}.
\]
Now we let
\begin{align*}
f(x_1)e^{\frac{3x_1^2}{4}}(-1)^{n_1+1}:=\int_{\frac{3x_1}{2}}^{\infty}\left(x_2-\frac{3x_1}{2}\right) e^{-x_2^2}dx_2=\frac12 e^{-\frac{9x_1^2}{4}}-\frac{3x_1}{2} \frac{\sqrt{\pi}}{2}\left(1-E\left(\frac{3x_1}{2\sqrt{\pi}}\right)\right),
\end{align*}
using integration by parts. We then compute (using $n_1>0$)
\begin{equation}\label{difff}
f^{(2n_1)}(0)=\frac{(-1)^{n_1+1}}2\left[\frac{\partial^{2n_1}}{\partial x_1^{2n_1}}e^{-3x_1^2}\right]_{x_1=0}+\frac{(-1)^{n_1+1}3\sqrt{\pi}}{4}
\left[\frac{\partial^{2n_1}}{\partial x_1^{2n_1}}\left(x_1 e^{-\frac{3x_1^2}{4}}E\left(\frac{3x_1}{2\sqrt{\pi}}\right)\right)\right]_{x_1=0}.
\end{equation}

For the left-hand side of \eqref{want} we use \eqref{computediff} and consider
\[
-\left[ \frac{\partial^{2n_1}}{\partial x_1^{2n_1}} \left(\int_0^\infty x_2 M^*\left(\frac{2x_2+3x_1}{\sqrt{2\pi}}\right)e^{3x_1^2+3x_1x_2+x_2^2}dx_2\right)\right]_{x_1=0}.
\]
Making the change of variables $u=\frac{2x_2+3x_1}{\sqrt{2\pi}}$, the integral before differentiation (including the minus sign) becomes
\begin{equation}\label{int}
-\sqrt{\frac{\pi}{2}} e^{\frac{3x_1^2}{4}} \frac12 \int_{\frac{3x_1}{\sqrt{2\pi}}}^\infty \left(\sqrt{2\pi} u-3x_1\right)M^*(u)e^{\frac{\pi u^2}{2}}du.
\end{equation}
Using integration by parts, the contribution from $\sqrt{2\pi }u$ equals
\[
\frac12 e^{3x_1^2} \left(E\left(\frac{3x_1}{\sqrt{2\pi}}\right)-1\right)+\frac1{\sqrt{2}}e^{\frac{3x_1^2}{4}}\left(1-E\left(\frac{3x_1}{2\sqrt{\pi}}\right)\right).
\]
Thus differentiating $2n_1$ times with respect to $x_1$ and then setting $x_1=0$ gives (using that $z\mapsto E(z)$ is odd)
\[
-\frac1{2}\left[\frac{\partial^{2n_1}}{\partial x_1^{2n_1}}e^{3x_1^2}\right]_{x_1=0}\!\!\!+\frac1{\sqrt{2}}
\left[\frac{\partial^{2n_1}}{\partial x_1^{2n_1}}e^{\frac{3x_1^2}{4}}\right]_{x_1=0}
\!\!\!=-\frac12(-1)^{n_1}\left[\frac{\partial^{2n_1}}{\partial x_1^{2n_1}}e^{-3x_1^2}\right]_{x_1=0}
\!\!\!+\frac{1}{\sqrt{2}}\left[\frac{\partial^{2n_1}}{\partial x_1^{2n_1}}e^{\frac{3x_1^2}{4}}\right]_{x_1=0}.
\]
The first term matches the first term in \eqref{difff}, the second term is the second term on the right-hand side of \eqref{want}. For the second term in \eqref{int}, we split
\[
\frac32\sqrt{\frac{\pi}{2}}e^{\frac{3x_1^2}{4}}x_1
\left(\int_0^\infty M^\ast(u) e^{\frac{\pi u^2}{2}}du-\int_0^{\frac{3x_1}{\sqrt{2\pi}}}E(u)e^{\frac{\pi u^2}{2}}du
+\int_0^{\frac{3x_1}{\sqrt{2\pi}}}e^{\frac{\pi u^2}{2}}du\right).
\]
Since we take an even number of derivatives only the last term survives, yielding the contribution
\begin{equation*}
\frac32\sqrt{\frac{\pi}{2}}
\left[\frac{\partial^{2n_1}}{\partial x_1^{2n_1}}\left(x_1e^{\frac{3x_1^2}{4}}\int_0^{\frac{3x_1}{\sqrt{2\pi}}}e^{\frac{\pi u^2}{2}}du\right)\right]_{x_1=0}
=
-(-1)^{n_1}\frac{3\sqrt{\pi}}{4}\left[\frac{\partial^{2n_1}}{\partial x_1^{2n_1}}\left(x_1e^{-\frac{3x_1^2}{4}}	E\left(\frac{3x_1}{2\sqrt{\pi}}\right)\right)\right]_{x_1=0}.
\end{equation*}
This is the second term in \eqref{difff}, which implies \eqref{want}.

The left-over term from \eqref{want} overall contributes as
\begin{equation*}
-\sqrt{2}\sum_{\boldsymbol{\alpha}\in\mathscr{S}^\ast}\sum_{0\leq \boldsymbol{\ell}\leq\frac{kp}{\delta}-1}e^{-2\pi i\frac{h}{k}Q(\boldsymbol{\ell}+\boldsymbol{\alpha})}\sum_{n_1\geq 1}\frac{B_{2n_1+1}\left(\frac{\delta\left(\ell_1+\alpha_1\right)}{kp}\right)}{(2n_1+1)!}\left[\frac{\partial^{2n_1}}{\partial x_1^{2n_1}}e^{\frac{3x_1^2}{4}}\right]_{x_1=0}\left(\frac{kp}{\delta}\right)^{2n_1-1}t^{n_1-1}.
\end{equation*}
%So we want
%\begin{align*}
%&\left(\frac12 \left[\frac{\partial^{2n_1}}{\partial x_1^{2n_1}}e^{-3x_1^2}\right]_{x_1=0}-\frac{3\sqrt{\pi}}{4}
%\left[\frac{\partial^{2n_1}}{\partial x_1^{2n_1}}\left(x_1e^{-\frac{3x_1^2}{4}}E\left(\frac{3x_1}{2\sqrt{\pi}}\right)\right)\right]_{x_1=0}\right)2(-1)^n\\
%&"="\frac{(-1)^n}{2}\left[\frac{\partial^{2n_1}}{\partial x_1^{2n_1}}e^{-\frac{3x_1^2}{2}}\right]_{x_1=0}
%+\frac{(-1)^n}{\sqrt{2}}\left[\frac{\partial^{2n_1}}{\partial x_1^{2n_1}}e^{-\frac{3x_1^2}{4}}\right]_{x_1=0}
%+(-1)^{n+1}\frac32\sqrt{\pi}\left[\frac{\partial^{2n_1}}{\partial x_1^{2n_1}}\left(x_1e^{-\frac{3x_1^2}{4}}E\left(\frac{3x_1}{2\sqrt{\pi}}\right)\right)\right]_{x_1=0}.
%\end{align*}
%There is one term matching, the rest is left-over.

The final term in Euler-Maclaurin is
\begin{multline*}
2\sum_{\boldsymbol{\alpha}\in\mathscr S^\ast}\sum_{0\leq \boldsymbol{\ell}\leq \frac{kp}{\delta}-1}e^{-2\pi i\frac{h}{k}Q(\boldsymbol{\ell}+\boldsymbol{\alpha})} \sum_{\substack{n_1,n_2\geq 0\\ n_1\not\equiv n_2\pmod{2}}}\frac{B_{n_1+1}\left(\frac{\delta(\ell_1+\alpha_1)}{kp}\right)}{(n_1+1)!} \frac{B_{n_2+1}\left(\frac{\delta(\ell_2+\alpha_2)}{kp}\right)}{(n_2+1)!} \\\times\left(\mathcal G_3^{(n_1,n_2)} (\boldsymbol{0})+(-1)^{n_1+1} \widetilde{\mathcal G}_3^{(n_1,n_2)}(\boldsymbol{0})\right)\left(\frac{kp}{\delta}\right)^{n_1+n_2} t^{\frac{n_1+n_2-1}{2}}.
\end{multline*}
Then
\begin{align*}
\mathcal G_3^{(n_1,n_2)}(\boldsymbol{0})+(-1)^{n_1+1}\widetilde{\mathcal G}_3^{(n_1,n_2)}(\boldsymbol{0}) = i^{n_1+n_2-1}\mathcal G_1^{(n_1,n_2)}(\boldsymbol{0})
\end{align*}
gives the relation to \eqref{F2main}.

We next consider $H_2$.   We have, with $\mathcal G_4(x):=x\mathcal F_4(x)$,
$$
H_2\left(\frac{h}{k}+\frac{it}{2\pi}\right)=\frac1{2\sqrt{t}}\sum_{\beta \in \left\{\frac1{p},1-\frac1{p}\right\}}\sum_{0\leq r\leq \frac{kp}{\delta}-1} e^{-2\pi i\frac{h}{k}(r+\beta)^2} \sum_{m\in \frac{(r+\beta)\delta}{kp}+\N_0} \mathcal G_4\left(\frac{kp}{\delta}\sqrt{t}m\right).
$$
The Euler-Maclaurin main term is
$$
\frac1{2\sqrt{t}} \frac{\delta}{kp\sqrt{t}} \mathcal I_{\mathcal G_4} \sum_{\beta\in\left\{\frac1{p},1-\frac1{p}\right\}}\sum_{0\leq r\leq \frac{kp}{\delta}-1}e^{-2\pi i\frac{h}{k}(r+\beta)^2}= \frac{\delta}{kpt}\mathcal I_{\mathcal G_4} \sum_{r\pmod{\frac{kp}{\delta}}} e^{-2\pi i\frac{h}{k}\left(r+\frac1p\right)^2}.
$$

The second term becomes
$$
-\sum_{0\leq r\leq \frac{kp}{\delta}-1} e^{-2\pi i\frac{h}{k}\left(r+\frac1{p}\right)^2} \sum_{m\geq 0} \frac{B_{2m+2}\left(\frac{\delta\left(r+\frac1{p}\right)}{kp}\right)}{(2m+2)!} \mathcal G_4^{(2m+1)}(0) \left(\frac{kp}{\delta}\right)^{2m+1}t^m.
$$
Then
$$
\mathcal G_4^{(2m+1)}(0)=(2m+1)\mathcal F_4^{(2m)}(0) = (2m+1)(-1)^{m+1} \mathcal F_2^{(2m)}(0) = (-1)^{m+1}\mathcal G_2^{(2m+1)}(0).
$$
gives the relation to \eqref{splitsum}.

Finally, we consider $\mathcal E_{2,2}$. We first study $\mathcal E_{2,2}^\ast$ and write
\begin{align*}
\mathcal E_{2,2}^\ast \left(\frac{h}{k}+\frac{it}{2\pi}\right) =& \frac1{\sqrt{\pi t}} \sum_{\boldsymbol{\alpha} \in \mathscr S}\eta (\boldsymbol{\alpha})\sum_{0\leq \boldsymbol{\ell}\leq \frac{kp}{\delta}-1}e^{-2\pi i\frac{h}{k}Q(\boldsymbol{\ell}+\boldsymbol{\alpha})} \sum_{n \in \frac{\delta}{kp}(\boldsymbol{\ell}+\boldsymbol{\alpha})+\N_0^2} \mathcal G_5\left(\frac{kp}{\delta} \sqrt{t}\boldsymbol{n}\right)\\&+\frac1{\sqrt{\pi t}}\sum_{\boldsymbol{\alpha}\in \widetilde{\mathscr S}} \widetilde \eta (\boldsymbol{\alpha})\sum_{0\leq \boldsymbol{\ell}\leq \frac{kp}{\delta}-1} e^{-2\pi i\frac{h}{k}Q(-\ell_1-\alpha_1,\ell_2+\alpha_2)} \sum_{n \in \frac{\delta}{kp}(\boldsymbol{\ell}+\boldsymbol{\alpha})+\N_0^2} \widetilde{\mathcal G}_5 \left(\frac{kp}{\delta}\sqrt{t} \boldsymbol{n}\right),
\end{align*}
where
\[
\mathcal G_5(\boldsymbol{x}):= \frac{1}{2\sqrt{2}}e^{-\frac{3x_1^2}{2}-3x_1x_2-x_2^2 }M^*\left(\sqrt{\frac{3}{2\pi}}x_1\right),\quad \widetilde{\mathcal G}_5(\boldsymbol{x}):= {\mathcal G}_5(-x_1,x_2).
\]
As before the main term in Euler-Maclaurin vanishes. The second term equals
\begin{multline*}
-\frac2{\sqrt{\pi t}} \sum_{\boldsymbol{\alpha}\in \mathscr S^\ast} \sum_{0\leq \boldsymbol{\ell}\leq \frac{kp}{\delta}-1} e^{-2\pi i\frac{h}{k}Q(\boldsymbol{\ell}+\boldsymbol{\alpha})} \sum_{n_2\geq 0} \frac{B_{2n_2+1}\left(\frac{\delta(\ell_2+\alpha_2)}{kp}\right)}{(2n_2+1)!}\\\times \int_0^\infty \left(\mathcal G_5^{(0,2n_2)}(x_1,0)+\widetilde{\mathcal G}_5 ^{(0,2n_2)} (x_1,0)\right)dx_1 \left(\frac{kp}{\delta}\right)^{2n_2-1}t^{n_2-\frac12}.
\end{multline*}
It is however not hard to see that
$$
\mathcal G_5^{(0,2n_2)} (x_1,0) +\widetilde{\mathcal G}_5 ^{(0,2n_2)} (x_1,0)=0.
$$
The third term in Euler-Maclaurin is
\begin{multline*}
-\frac{2}{\sqrt{\pi t}} \sum_{\boldsymbol{\alpha}\in \mathscr S^\ast} \sum_{0\leq \boldsymbol{\ell} \leq \frac{kp}{\delta}-1} e^{-2\pi i\frac{h}{k}Q(\boldsymbol{\ell}+\boldsymbol{\alpha})} \sum_{n_1\geq 0} \frac{B_{2n_1+1} \left(\frac{\delta(\ell_1+\alpha_1)}{kp}\right)}{(2n_1+1)!}\\\times \int_0^\infty \left(\mathcal G_5^{(2n_1,0)}(0,x_2)-\widetilde{\mathcal G}_5^{(2n_1,0)} (0,x_2)\right)dx_2 \left(\frac{kp}{\delta}\right)^{2n_1-1}t^{n_1-\frac12}.
\end{multline*}
Now
\begin{align*}
\mathcal{G}_5^{(2n_1, 0)}\left(0, x_2\right)-\widetilde{\mathcal{G}}_5^{(2n_1, 0)}\left(0, x_2\right)=2\mathcal{G}_{5, 1}^{(2n_1, 0)}\left(0, x_2\right),
\end{align*}
where $\mathcal G_{5,1}(\boldsymbol{x}) := -\frac{1}{2\sqrt{2}}e^{-\frac{3x_1^2}{2}-3x_1x_2-x_2^2}$.
We thus need to compute
\begin{align*}
&2\int_0^\infty \mathcal{G}_{5, 1}^{(2n_1, 0)}\left(0, x_2\right)dx_2=-\frac1{\sqrt{2}}\left[\frac{\partial^{2n_1}}{\partial x_1^{2n_1}}e^{\frac{3x_1^2}{4}}\int_0^\infty e^{-\left(x_2+\frac32 x_1\right)^2}dx_2\right]_{x_1=0}
\\
&=-\frac{\sqrt{\pi}}{2\sqrt{2}}\left[\frac{\partial^{2n_1}}{\partial x_1^{2n_1}}\left(e^{\frac{3x_1^2}{4}}\left(1-E\left(\frac{3x_1}{2\sqrt{\pi}}\right)\right)\right)\right]_{x_1=0}
=-\frac{\sqrt{\pi}}{2\sqrt{2}}\left[\frac{\partial^{2n_1}}{\partial x_1^{2n_1}}e^{\frac{3x_1^2}{4}}\right]_{x_1=0}.
\end{align*}
This term then contributes as
\begin{equation}\label{leftover}
\frac{1}{\sqrt{2}}\sum_{\boldsymbol{\alpha} \in \mathscr S^\ast}\sum_{0\leq \boldsymbol{\ell}\leq \frac{kp}{\delta}-1}e^{-2\pi i\frac{h}{k}Q(\boldsymbol{\ell}+\boldsymbol{\alpha})} \sum_{n_1\geq 0}
\frac{B_{2n_1+1}\left(\frac{\delta(\ell_1+\alpha_1)}{kp}\right)}{(2n_1+1)!}
\left[\frac{\partial^{2n_1}}{\partial x_1^{2n_1}} e^{\frac{3x_1^2}{4}}\right]_{x_1=0}\left(\frac{kp}{\delta}\right)^{2n_1-1} t^{n_1-1}.
\end{equation}

The final term in the Euler-Maclaurin summation formula is
\begin{multline*}
 \frac2{\sqrt{\pi t}}\sum_{\boldsymbol{\alpha}\in \mathscr S^\ast}\eta(\boldsymbol{\alpha})\sum_{0\leq\boldsymbol{\ell}\leq \frac{kp}{\delta}-1}e^{-2\pi i\frac{h}{k} Q(\boldsymbol{\ell}+\boldsymbol{\alpha})}
\sum\limits_{n_1, n_2\geq 0\atop{n_1\not\equiv n_2\pmod{2}}}\frac{B_{n+1}\left(\frac{\delta\left(\ell_1+\alpha_1\right)}{kp}\right)}{\left(n_1+1\right)!}
\frac{B_{n+2}\left(\frac{\delta\left(\ell_2+\alpha_2\right)}{kp}\right)}{\left(n_2+1\right)!}\\
\times\left(\mathcal{G}_5^{(n_1, n_2)}(\boldsymbol{0})+(-1)^{n_1+1}\widetilde{\mathcal{G}}_5^{(n_1, n_2)}(\boldsymbol{0})\right)\left(\frac{kp}{\delta}\sqrt{t}\right)^{n_1+n_2}.
\end{multline*}
It is easy to see that under the condition $n_1\not\equiv n_2\pmod{2}$ we have
\[
\mathcal{G}_5^{(n_1, n_2)}(\boldsymbol{0})+(-1)^{n_1+1}\widetilde{\mathcal{G}}_5^{(n_1, n_2)}(\boldsymbol{0})=0.
\]
Next, we consider
\[
H_3\left(\frac{h}{k}+\frac{it}{2\pi}\right)=\frac1{2\sqrt{2\pi t}}\sum_{\beta\in\left\{\frac1p, 1-\frac1p\right\}}\sum_{0\leq r\leq\frac{kp}{\delta}-1}
e^{-2\pi i\frac{h}{k}(r+\beta)^2}\sum_{m\in\frac{\delta(r+\beta)}{kp}+\N_0}\mathcal{F}_2\left(\frac{kp}{\delta}\sqrt{t}m\right).
\]
The Euler-Maclaurin main term is
\begin{align*}
\frac{\delta}{2k pt\sqrt{2\pi }}\mathcal{I}_{\mathcal{F}_2}\sum_{\beta\in\left\{\frac1p, 1-\frac1p\right\}}\sum_{r\pmod{\frac{kp}{\delta}}}e^{-2\pi i\frac{h}{k}(r+\beta)^2}
=\frac{\delta}{kpt\sqrt{2\pi }}\mathcal{I}_{\mathcal{F}_2}\sum_{r\pmod{\frac{kp}{\delta}}}e^{-2\pi i\frac{h}{k}\left(r+\frac1p\right)^2}.
\end{align*}
The final term is
\[
\frac1{\sqrt{2\pi t}}\sum_{0\leq r\leq\frac{kp}{\delta}-1}e^{-2\pi i\frac{h}{k}\left(r+\frac1p\right)^2}\sum_{m\geq 0}
\frac{B_{2m+2}\left(\frac{\delta\left(r+\frac1p\right)}{kp}\right)}{(2m+2)!}\mathcal{F}_2^{(2m+1)}(0)=0
\]
since $\mathcal{F}_2$ is an even function.

Collecting all growing terms gives
\begin{equation}\label{mainat}
\frac{\delta}{kpt}\sum_{r\pmod{\frac{kp}{\delta}}}e^{-2\pi i\frac{h}{k}\left(r+\frac1p\right)^2}
\left(\mathcal{I}_{\mathcal{G}_3(0, \cdot)-\widetilde{\mathcal{G}}_3(0, \cdot)}+\mathcal{I}_{\mathcal{G}_4}+\frac1{\sqrt{\pi}}
\mathcal{I}_{\mathcal{G}_5(0, \cdot)-\widetilde{\mathcal{G}}_5(0, \cdot)}+\frac{\mathcal{I}_{\mathcal{F}_2}}{\sqrt{2\pi}}\right).
\end{equation}
We compute $\mathcal I_{\mathcal{F}_2}=\frac{\sqrt{\pi}}{2}$ and, using integration by parts,
\begin{align*}
\mathcal{I}_{\mathcal{G}_4} &= \int_{0}^{\infty}xe^{x^2}M^*\left(\sqrt{\frac{2}{\pi}}x\right)dx=-\frac{M^*(0)}{2}-\sqrt{\frac{2}{\pi}}\int_{0}^{\infty}e^{-x^2}dx=\frac12-\frac1{\sqrt{2}}
\end{align*}
by conjugating \eqref{sumsmatch}. Moreover, \eqref{want} gives
\begin{align*}
\mathcal{I}_{\mathcal{G}_3(0, \cdot)-\widetilde{\mathcal{G}}_3(0, \cdot)}&-\int_0^\infty x_2e^{-x_2^2}dx_2+\frac1{\sqrt{2}}=\frac12\left[e^{-x_2^2}\right]_0^\infty+\frac1{\sqrt{2}}
=-\frac12+\frac1{\sqrt{2}},\\
\mathcal{I}_{\mathcal{G}_5(0, \cdot)-\widetilde{\mathcal{G}}_5(0, \cdot)}&=-\frac{\sqrt{\pi}}{2\sqrt{2}}.
\end{align*}
Thus the term inside the paranthesis in \eqref{mainat} vanishes.

We are left to show that the contributions from \eqref{computediff} and \eqref{leftover} vanish. For this it suffices to show
that, for all $n\in\N$, 
\begin{equation*}
\sum_{\boldsymbol{\alpha}\in\mathscr{S}^\ast}\sum_{0\leq \boldsymbol{\ell}\leq\frac{kp}{\delta}-1}e^{-2\pi i\frac{h}{k}Q(\boldsymbol{\ell}+\boldsymbol{\alpha})}B_{2n+1}\left(\frac{\delta\left(\ell_1+\alpha_1\right)}{kp}\right)=0.
\end{equation*}
%From the computation on page 3 we know this sum is zero for $\frac{p}{\delta} \notin \{ 1,2 \}$ since we got 0 for the sum on $\ell_2$.
%If $\frac{p}{\delta}=1$, then (almost verbatim) computation as on page 4 together with {\bf KB: it should be equal to 0}
%$$B_{2m+1}\left(\frac12\right)=B_{2m+1}(0)=1; \ \ m \geq 1$$
%gives (\ref{left-short}).
%For $\frac{p}{\delta}=2$, we use the same argument as on page 6 and above formula, and again we get zero.
%{\bf KB: Putting in the details here to be sure.}

As in (\ref{splitsum}) we get that this sum is zero for $\frac{p}{\delta} \notin \{ 1,2 \}$.
Next we consider $\frac{p}\delta =1$. We first combine the first and third element in $\mathscr S^\ast$. Using \eqref{claimquad} and
\begin{align}\label{oddminus}
B_{2m+1}(1-x) = -B_{2m+1}(x)
\end{align}
gives that these cancel. Thus we need to show that
\begin{align}\label{wantvanish}
\sum_{0\leq \ell<k} B_{2n+1}\left(\frac{\ell_1}k\right)e^{-2\pi i\frac{h}k Q\left(\ell_1,\ell_2+1-\frac1p\right)}=0.
\end{align}
We use \eqref{onlysecondterm} and distinguish again whether $k$ is even or odd. If $k$ is odd we do the same change of variables and use \eqref{oddminus} to obtain that \eqref{wantvanish} equals
$$
B_{2n+1}(0) \sum_{\ell_2\pmod k}e^{-2\pi i\frac{h}k\left(\ell_2+1-\frac1p\right)^2}=0
$$
since for $m\geq 3$ odd, $B_m(0) =0$.

If $k$ is even, then we obtain
$$
B_{2n+1}(0) \sum_{\ell_2\pmod k} e^{-2\pi i\frac{h}k \left(\ell_2+1-\frac1p\right)^2}+B_{2n+1}\left(\frac12\right) \sum_{\ell_2\pmod k} e^{-2\pi i\frac{h}k \left(\ell_2+1-\frac1p\right)^2} =0
$$
since for $m$ odd $B_m(\frac12) =0$.

We next turn to the case $\frac{p}\delta =2$. Then only the second element survives and we want
\begin{equation}\label{alsowant}
\sum_{0\leq \boldsymbol{\ell}\leq 2k-1} B_{2n+1}(0) e^{-2\pi i\frac{h}k \left(\ell_1,\ell_2+1-\frac1p\right)} = 0.
\end{equation}
We obtain for the left-hand side of \eqref{alsowant}
$$
\left(B_{2n+1}(0)+B_{2n+1}\left(\frac12\right)\right) \sum_{\ell_2 \pmod{k}} e^{-2\pi i\frac{h}k \left(\ell_2 +1-\frac1p\right)^2}=0.
$$
This finally proves the theorem.

%{\bf KB: What about $\frac{p}{\delta}\not\in\{1,2\}$?}

\end{proof}

\subsection{Proof of Theorem \ref{maintheorem}}
We are now ready to prove a refined version of Theorem \ref{maintheorem}.

\begin{theorem}\label{quantumref}
	\begin{enumerate}[leftmargin=*]
	
		\item[\textnormal{(1)}] The function $\widehat{F}_1:\Q\to\C$ defined by $\widehat{F}_1(\frac{h}{k}):=F_1(e^{2\pi i\frac{ph}{k}})$ is a depth two quantum modular form of weight one for $\Gamma_p$ with multiplier $(\frac{-3}{d})$.
		
		\item[\textnormal{(2)}] The function $\widehat{F}_2:\mathcal \Q \to \C$ defined by $\widehat{F}_2(\frac{h}{k}):=F_2(e^{2\pi i\frac{ph}{k}})$ is a depth two quantum modular form of weight two for $\Gamma_p$ with multiplier $(\frac{3}{d})$.
		
		% and quantum set
		%$$
		%\mathcal Q_1:=\left\{ \frac{h}{k}\in\Q : \gcd(h,k) = 1, p|k \right\} \subset \Q.
		%$$
	\end{enumerate}
\end{theorem}

\begin{proof}
	%\begin{enumerate}[leftmargin=*]
		(1) We have, by Theorem \ref{asagree},
		$$
		\widehat{F}_1\left(\frac{h}{k}\right) = \lim_{t\to 0^+} F_1\left(e^{2\pi i\frac{ph}{k}-t}\right)=a_{hp_1,\frac{k}{p_2}}(0) = \lim_{t\to 0^+}\mathbb E_1\left(-\frac{h}{k}+\frac{it}{2\pi}\right),
		$$
		where $p_1:= p/\gcd(k,p)$, $p_2:=\gcd(k,p)$. Proposition \ref{quantE} then gives the claim.

		(2)  Theorem \ref{matchas2} gives 		$$
		\widehat{F}_2\left(\frac{h}{k}\right)= \lim_{t\to 0^+} {F}_2\left(e^{2\pi i\frac{ph}{k}-t}\right) = b_{hp_1,\frac{k}{p_2}}(0) = \lim_{t\to0^+}\mathbb E_2\left(-\frac{h}{k}+\frac{it}{2\pi}\right).
		$$
		Proposition \ref{Derivativetrans} then gives the claim.
	%\end{enumerate}
\end{proof}

\begin{remark} For odd $d$, we have that $(\frac{3}{d})=(\frac{-3}{d})=1$ if and only if $d \equiv 1 \pmod{12}$ so that both $F_1$ and $F_2$ can be viewed as quantum modular forms with the trivial character under a suitable subgroup of $ \Gamma_p$ (e.g. the principal congruence subgroup $\Gamma(12p)$).
\end{remark}

\section{Completed indefinite theta functions}

In this section, we embed the double Eichler integrals in a modular context by viewing them as ``purely non-holomorphic'' parts of indefinite theta series.

\subsection{Weight one	} \label{Completion:weightone}

The functions $E_2$ and $M_2$ were introduced in \cite{ABMP}, where they played a crucial role in understanding modular indefinite theta functions of signature $(j,2)$ ($j\in\N_0$). We consider the quadratic form $Q_1(\boldsymbol{n}):=\frac12\boldsymbol{n}^TA_1\boldsymbol{n}$ and the bilinear form $B_1(\boldsymbol{n},\boldsymbol{m}):=\boldsymbol{n}^TA_1\boldsymbol{m}$ given by
$
A_1:=\left(\begin{smallmatrix}
6 & 3 & 6 & 3 \\
3 & 2 & 3 & 2 \\
6 & 3 & 0 & 0 \\
3 & 2 & 0 & 0
\end{smallmatrix}\right),
$
and define $A_0:=\left(\begin{smallmatrix}
6&3\\3&2
\end{smallmatrix}\right)$, $P_0(\boldsymbol{n}):=M_2(\sqrt{3};\sqrt{3}\left(2n_1+n_2\right),n_2)$ and, for $\boldsymbol{n}\in\R^4$, set
\begin{align*}
P(\boldsymbol{n})&:=M_2\left(\sqrt{3}; \sqrt{3}(2n_3+n_4), n_4\right)
+\left(\sgn(2n_3+n_4)+\sgn(n_1)\right)
\left(\sgn(3n_3+2n_4)+\sgn(n_2)\right)
\\&\quad +
\left(\sgn\left(n_4\right)+\sgn\left(n_2\right)\right)M_1\left(\sqrt{3}(2n_3+n_4)\right)
+
\left(\sgn\left(n_3\right)+ \sgn\left(n_1\right)\right)M_1\left(3n_3+2n_4\right).
\end{align*}
Note that, for $\boldsymbol{\alpha}\in\mathscr{S}^\ast$,
\begin{align*}
2\mathcal E_{1,\boldsymbol{\alpha}}(\tau)=\Theta_{-A_0,P_0,\boldsymbol{\alpha}}(\tau).
\end{align*}
We view this function as ``purely non-holomorphic" part of the indefinite theta function
\begin{equation}\label{defineTheta}
\Theta_{A_1,P,\boldsymbol{a}}(\tau)= \sum_{\boldsymbol{n}\in \boldsymbol{a}+\Z^4}
P\left(\sqrt{v}\boldsymbol{n}\right)q^{Q_1\left(\boldsymbol{n}\right)},
\end{equation}
where $\boldsymbol{a}\in \frac1pA_1^{-1}\Z^4$ with $(a_3,a_4)=(\alpha_1,\alpha_2)$. One can either employ Section 4.3 of \cite{ABMP} or proceed directly (as we do here) to prove the following proposition.

\begin{proposition}\label{Prop:CompletionOne}
	Assume that $\boldsymbol{a}\in \frac1p A_1^{-1}\Z^4$ with $a_1,a_2,a_4\not\in \Z$.
	{\normalfont
		\begin{enumerate}[leftmargin=*]
			\item
			{\it We have
				\begin{align*}
					\Theta_{A_1,P^-,\boldsymbol{a}}(\tau)
					=2\mathcal E_{1,(a_3,a_4)}(\tau)
					\Theta_{A_0,1,(a_1-a_3,a_2-a_4)}(\tau)
					,
				\end{align*}
				where
				\begin{equation*}
					P^-(\boldsymbol{m}):= M_2\left(\sqrt{3};\sqrt{3}\left(2n_3+n_4\right),n_4\right).
				\end{equation*}
			}
			\item
			{\it The functions $\Theta_{A_1,P,\boldsymbol{a}}$ and $\Theta_{A_0,P_0,(a_3,a_4)}$ converge absolutely and locally uniformly.}
			
			\item
			{\it The function $\tau\mapsto\Theta_{A_1,P,a}(p\tau)$ transforms like a modular form of weight two for some subgroup of $\SL_2(\Z)$ and some character.}
			
		\end{enumerate}
	}
	
\end{proposition}
\begin{remark}
	When considering indefinite theta functions of signature $(j,2)$, one usually obtains four $M_2$-terms as the purely ``non-holomorphic" part. The arguments of these four $M_2$-functions are dictated by the holomorphic part. The fact that $(1,0,0,0)^T$ and $(0,1,0,0)^T$ (which correspond to $n_1$ and $n_2$ occuring in $P$) have norm zero with respect to $A_1^{-1}$ causes the ``missing" $M_2$-terms to vanish. Therefore we refer to this situation as a \emph{double null limit} (see \cite{ABMP}).
\end{remark}
\begin{proof}[Proof of Proposition \ref{Prop:CompletionOne}]
	(1)
	Shifting $(n_1,n_2,n_3,n_4)\mapsto (n_1-n_3,n_2-n_4,n_3,n_4)$ on the left hand side of the identity gives the claim.

	(2) For $\Theta_{-A_0,P_0,(a_3,a_4)}$ we employ the
	asymptotic given in \eqref{limitsM2}, to obtain
	\begin{align*}
	&\left\lvert
	M_2\left(\sqrt{3};\sqrt{3v}\left(2n_1+n_2\right),\sqrt{v}n_2\right)q^{\frac12  \boldsymbol{n}^TA_0\boldsymbol{n}}\right\rvert
	\leq \frac{e^{-\pi \left(3(2n_1+n_2)^2+n_2^2\right)v}}{\pi^2 n_1n_2}e^{\pi  \boldsymbol{n}^TA_0\boldsymbol{n}v}
	\\  &\qquad \leq
	c_1e^{-2\pi  \boldsymbol{n}^TA_0\boldsymbol{n} v}e^{\pi  \boldsymbol{n}^TA_0\boldsymbol{n} v}
	=
	c_1e^{-\pi  \boldsymbol{n}^TA_0\boldsymbol{n} v}
	\end{align*}
	for some $c_1\in\R^+$ and $(n_1,n_2)\in (a_3,a_4)+\Z^2$ with $n_1\neq 0$. By plugging in the definition, one can show that for some $c_2\in\R^+$ and $n=(0,n_2)\in (a_3,a_4)+\Z^2$
	\begin{align*}
		\left\lvert
		M_2\left(\sqrt{3};\sqrt{3v}n_2,\sqrt{v}n_2\right)q^{\frac12  \boldsymbol{n}^TA_0\boldsymbol{n}}\right\rvert
		\leq 	c_2e^{-\pi  \boldsymbol{n}^TA_0\boldsymbol{n} v}.
	\end{align*}
	Using that $A_0$ is positive definite, we obtain, for some $c_3\in\R^+$
	\begin{align*}
	\sum_{\boldsymbol{n}\in (a_3,a_4)+\Z^2} \left\lvert
	M_2\left(\sqrt{3};\sqrt{3v}\left(2n_1+n_2\right),\sqrt{v}n_2\right)q^{\frac12 \boldsymbol{n}^T A_0\boldsymbol{n}}\right\rvert
	\leq c_3
	\sum_{\boldsymbol{n}\in (a_3,a_4)+\Z^2} e^{-\pi  \boldsymbol{n}^TA_0\boldsymbol{n}v }  <\infty,
	\end{align*}
	implying the absolute and locally uniform convergence of $\Theta_{-A_0,P_0,(a_3,a_4)}$. Combining this with (1) and the convergence of the positive definite theta series $\Theta_{A_1,1,(a_1-a_3,a_2-a_4)}$, we obtain absolute and locally uniform convergence of the $M_2$-part of $\Theta_{A_1,P,\boldsymbol{a}}$.
	
	For the part containing only sign-terms
	\begin{align}\label{signseries}
	\sum_{\boldsymbol{n}\in \boldsymbol{a}+\Z^4}
	\left(\sgn(2n_3+n_4)+\sgn(n_1)\right)
	\left(\sgn(3n_3+2n_4)+\sgn(n_2)\right)
	q^{Q_1(\boldsymbol{n})},
	\end{align}
	we consider the determinant of  $\Delta_{A_1}(\boldsymbol{n},\boldsymbol{b}_1,\boldsymbol{b}_2,\boldsymbol{b}_3,\boldsymbol{b}_4)$, where $\left(\Delta_M(\boldsymbol{v}_1,\dots,\boldsymbol{v}_5)\right)_{j,\ell}:=\boldsymbol{v}_j^T M\boldsymbol{v}_\ell$ and
	\begin{align*}
	(\boldsymbol{b}_1,\boldsymbol{b}_2,\boldsymbol{b}_3,\boldsymbol{b}_4):=\frac13
	\begin{pmatrix}
	0 & 0 & 1 & 0 \\
	0 & 0 & 0 & 3 \\
	2 & -3& -1& 0 \\
	-3& 6 & 0 & -3
	\end{pmatrix}.
	\end{align*}
	We compute the determinant $\det(\Delta_{A_1}(\boldsymbol{n},\boldsymbol{b}_1,\boldsymbol{b}_2,\boldsymbol{b}_3,\boldsymbol{b}_4))$ via Laplace expansion to obtain
	\begin{align*}
	e^{-\pi v Q_1(\boldsymbol{n})}&\leq
	e^{-\pi  \left(\frac{15}{16}B_1(\boldsymbol{b}_1,\boldsymbol{n})^2+\frac{2}{9}B_1(\boldsymbol{b}_2,\boldsymbol{n})^2+B_1(\boldsymbol{b}_1,\boldsymbol{n})B_1(\boldsymbol{b}_3,\boldsymbol{n})+2B_1(\boldsymbol{b}_2,\boldsymbol{n})B_1(\boldsymbol{b}_4,\boldsymbol{n})
		\right)v}
	\\
	&\leq
	e^{-c_4 \left(B_1(\boldsymbol{b}_1,\boldsymbol{n})^2+B_1(\boldsymbol{b}_2,\boldsymbol{n})^2+\lvert B_1(\boldsymbol{b}_3,\boldsymbol{n})\rvert +\lvert B_1(\boldsymbol{b}_4,\boldsymbol{n})\rvert
		\right)v}
	\end{align*}
	with some $c_4\in\R^+$ for all $\boldsymbol{n}\in \boldsymbol{a}+\Z^4$ which satisfy the condition
	\begin{align*}
	\left(\sgn(2n_3+n_4)+\sgn(n_1)\right)
	\left(\sgn(3n_3+2n_4)+\sgn(n_2)\right)\neq 0.
	\end{align*}
	Thus \eqref{signseries} is dominated by
	\begin{align*}
	& \sum_{\boldsymbol{n}\in \boldsymbol{a}+\Z^4}\left\lvert \left(\sgn(2n_3+n_4)+\sgn(n_1)\right)
	\left(\sgn(3n_3+2n_4)+\sgn(n_2)\right)
	e^{-\pi  Q_1(\boldsymbol{n})v}
	\right\rvert\\
	&\qquad \leq4  \sum_{\boldsymbol{n}\in \boldsymbol{a}+\Z^4}
	e^{-c_4 \left(B_1(\boldsymbol{b}_1,\boldsymbol{n})^2+B_1(\boldsymbol{b}_2,\boldsymbol{n})^2+\lvert B_1(\boldsymbol{b}_3,\boldsymbol{n})\rvert +\lvert B_1(\boldsymbol{b}_4,\boldsymbol{n})\rvert
		\right)v} <\infty.
	\end{align*}
	
	To deal with the contribution of the third and fourth summand of $P$ one combines the approaches of the two previous terms.

	(3) 	We use Lemma \ref{limitsE2} to rewrite $P$ as a limit of $E_2$-functions, namely
	\begin{align*}
	P(\boldsymbol{n})=&\lim_{\varepsilon\to 0}\widehat{P}_\varepsilon(\boldsymbol{n}),
	\end{align*}
	where
	\begin{align*}
	\widehat{P}_\varepsilon(\boldsymbol{n}) :=&\Bigg(
	E_2\left(\frac{\varepsilon}{3}; \sqrt{3}\left(2n_3+n_4\right), -\varepsilon\left(n_1+n_3+\frac{n_4}{\sqrt{3}}\right)+\frac{3n_2}{\varepsilon\left(2\sqrt{3}-3\right)} \right)
	\\&+
	E_2\left(\frac{\varepsilon}{2}; \left(3n_3+2n_4\right),\frac{3n_1}{\varepsilon\left(2\sqrt{3}-3\right)}- \varepsilon\left(n_2+\sqrt{3}n_3+n_4\right) \right)
	+E_2\left(\sqrt{3};\sqrt{3}\left(2n_3+n_4\right),n_4\right) \\
	&+
	E_2\left(-\sqrt{3}; \frac{n_2}{2\varepsilon}-\frac{\varepsilon}{2}\left(n_2+2n_4\right),\frac{\sqrt{3}}{2\varepsilon}\left(2n_1+n_2\right)-\frac{\sqrt{3}}{2}\varepsilon\left(2n_1+n_2+4n_3+2n_4\right) \right)
	\Bigg).
	\end{align*}
	One can then verify that each occuring term $E_2(\kappa; b^Tn,c^Tn)$ satisfies the Vign\'eras differential equation given in Theorem \ref{Vigneras} with $\lambda=0$ and $A=A_1$. A straightforward calculation shows that the Vign\'eras differential equation is satisfied for $\widehat{P}_{\varepsilon}$ with respect to $A_1$ if and only if it is satisfied for $\widehat{P}_{\varepsilon,p}(\boldsymbol{n}):=\widehat{P}_{\varepsilon}(\sqrt{p}\boldsymbol{n})$ with respect to $pA_1$.
	Furthermore, we have
	\begin{align*}
	\Theta_{A_1,P,\boldsymbol{a}}(p\tau)&=\Theta_{pA_1,P_p,\boldsymbol{a}}(\tau)
	=
	\lim_{\varepsilon\to 0}\Theta_{pA_1,\widehat{P}_{\varepsilon,p},\boldsymbol{a}}(\tau)
	\end{align*}
	where $P_p(\boldsymbol{n}):=P(\sqrt{p}\boldsymbol{n})$. We can apply Theorem \ref{Vigneras} to obtain weight 2 modularity of $\Theta_{pA_1,\widehat{P}_{\varepsilon,p},\boldsymbol{a}}$ since $\boldsymbol{a}\in (pA_1)^{-1}\Z^4$. Now, taking the limit $\varepsilon\to0$ proves the claim.
	
\end{proof}

\subsection{Completion: weight two}
Similarly as in the previous Section \ref{Completion:weightone}, the function $\mathbb{E}_2$ may be related to a modular object of weight three. This connection becomes evident when writing $\mathbb{E}_2$ as a Jacobi derivative as in Lemma \ref{Eichler2}. We leave the details to the reader.

\subsection{Lowering}
%Recall the lowering operator $\xi_{2-k}:=2i v^{2-k} {\frac{\partial}{\partial \overline{\tau}}}  $, which plays a prominent role in the theory of harmonic Maass forms. For $f_1$ and $f_2$ as in Theorem \ref{quantheorem}, we immediately get
%$$\xi_{2-r_1} \left( {I}_{f_1,f_2}(\tau) \right)  \doteq f_1(- \overline{\tau})I_{f_2}(\tau)$$
%
%{\bf AM: can you write an intelligent sequence about higher depth mmf and maybe add a reference? KB: Not sure. I feel this would go better to the last section. }
%\bf AM: Is there an interesting second order differential operator which annihilates $I_{f_1,g_1}$ or at least spits out
%	$f_1(-\overline{\tau})f_2(-\overline{\tau})$? KB: This is mixed. So I doubt it.}
The indefinite theta series considered in Subsection \ref{Completion:weightone} are higher depth harmonic Maass forms following Zagier-Zwegers. Roughly speaking, by this we mean that applying the {\it Maass lowering operator} $L:=-2iv^2\frac{\partial}{\partial \overline{\tau}}$ makes the function simpler. In particular, for the iterated Eichler integral, we have
$$
L\left(I_{f_1,f_2}(\tau)\right)= 2^{k_1}v^{k_1}f_1\left(-\overline{\tau}\right)I_{f_2}(\tau).
$$
Now $v^{k_1}f_1\left(-\overline{\tau}\right)$ is $v^{k_1}$ times a conjugated modular form of weight $k_1$ (so transforming of weight $-k_1$) and $I_{f_2}$, defined in (\ref{Eichler}), is the non-holomorphic part of an harmonic Maass form of weight $2-k_2$.
%{\bf AM: do you want to modify the rest of the sentence? KB: Not sure what you mean}

\section{Conclusion and further questions}

We conclude here with several comments and research directions

\begin{itemize}

\item[(1)] We plan to more systematically study  higher depth quantum modular forms and to describe explicitly the quantum $S$-modular matrix of $F(q)$. This requires a modification of several arguments used here for $F_2(q)$ (note that we restricted ourselves to $\Gamma_p$ out of necessity). This result would allow us to make a more precise connection between  $W(p)_{A_2}$ and its irreducible modules.
For one, we should be able to associate an $S$-matrix to the set of atypical irreducible $W(p)_{A_2}$-characters, in parallel to \cite{CM}.
%Our conjecture is that the quantum $S$-matrix is equivalent to the $S$-matrix of level $p-3$ Wess-Zumino-Witten model of $sl_3$.
%For the Lie algebra of type $sl_2$, a similar result (also based on quantum modularity) was proven  in  \cite{CMW}.

\item[(2)] Iterated (or multiple) Eichler integrals studied in Section 5 are of independent interest.
As in other theories dealing with iterated integrals (e.g. non-commutative modular symbols, Chen's integrals and multiple zeta-values) shuffle relations are expected to play an important role.
Another goal worth pursuing is to connect  iterated Eichler integrals of half-integral weights to Manin's work  \cite{Manin}.

\item[(3)] We plan to investigate the asymptotic of $F(q)$ in terms of finite $q$-series evaluated at root of unity. This  requires certain hypergeometric type formulas for double rank two false theta functions.

\item[(4)] 

%Let  $f\in S_{\frac{3}{2}}(\chi_1,\Gamma)\otimes S_{\frac32}(\chi_2,\Gamma)$ with $\chi_j$ multipliers and $\Gamma\subset%\SL_2(\Z)$. 
In recent work \cite{BKM} we found a new expression for the error of modularity appearing in Propositions \ref{quantE} and \ref{Derivativetrans}, at least if $M \tau=-\frac{1}{\tau}$.  Our formulae involve what we end up calling,  ``double Mordell" integrals.  
%$$
%\int_{0}^{i\infty}\int_{0}^{w_1} \frac{f(w_1,w_2)}{\sqrt{-i(w_1+\tau)} \sqrt{-i(w_2+\tau)}}dw_2dw_1
%$$
%and one-dimensional as a "double Mordell" integral.
In the rank one case this connection is well-understood \cite[Theorem 1.16]{Zw}.

\item[(5)] Very recently, W. Yuasa  \cite{Wataru} gave an explicit formula for the  {\em  tail} of  $(2,2p)$-torus links associated to the sequence of colored Jones polynomials: $J_{n \omega_j}(K,q)$, $n \in \N$, where $\omega_j$, $j=1,2$ are the fundamental weights. We were able to identify
the same tail as a summand of $F(q)$, up to the factor $1-q$ (viz. extract the ``diagonal" $m_1=m_2$ in formula (\ref{thetaLie})). This raises the following question: Is it true that $F(q)$ is the tail of  $J_{n \rho}(K,q)$, ($n\in\N$) (here $\rho=\omega_1+\omega_2$), up to a rational function of $q$? For related computations of tails colored with $\frak{sl}_3$ representations see \cite{Garoufalidis}.

\end{itemize}

\end{document}